\documentclass{amsart}

\usepackage[hmargin=3.5cm,top=3cm,bottom=3cm,includehead]{geometry}

\usepackage{babel}
\usepackage{pifont}

\usepackage{amsmath, amssymb, amsthm, mathrsfs, bbm}
\usepackage{graphicx,color,enumitem,microtype} 
\usepackage{cancel}
\usepackage{abraces}

\usepackage[pdfborder={0 0 0}]{hyperref}

\newcommand\N{{\mathbb N}}

\newcommand\R{{\mathbb R}}

\renewcommand{\d}{\mathrm{d}}

\def\cN{\mathcal{N}}

\def\AA{{\mathcal A}}
\def\BB{{\mathcal B}}

\def\DD{{\mathcal D}}

\def\LL{{\mathcal L}}
\def\MM{{\mathcal M}}
\def\NN{{\mathcal N}}
\def\OO{{\mathcal O}}

\def\TT{{\mathcal T}}
\def\UU{{\mathcal U}}
\def\VV{{\mathcal V}}
\def\WW{{\mathcal W}}
\def\VV{{\mathcal V}}

\def\BBB{{\mathscr B}}
\def\CCC{{\mathscr C}}
\def\DDD{{\mathscr D}}
\def\EEE{{\mathscr E}}

\def\HHH{{\mathscr H}}

\def\LLL{{\mathscr L}}
\def\MMM{{\mathscr M}}

\def\RRR{{\mathscr R}}
\def\SSS{{\mathscr S}}
\def\TTT{{\mathscr T}}

\def\LLLL{{\mathfrak L}}

\def\widehatt{\widetilde}

\def\Wloc{W_{\rm loc}}

\def\Lloc{L_{\rm loc}}

\def\Hloc{H_{\rm loc}}

\newcommand{\wto}{\rightharpoonup}

\def\eps{{\varepsilon}}

\newcommand{\la}{\langle}
\newcommand{\ra}{\rangle}

\newcommand{\grad}{\nabla}

\newcommand{\Nt}{|\hskip-0.04cm|\hskip-0.04cm|}

\DeclareMathOperator{\Div}{div}
\DeclareMathOperator{\supp}{supp}

\newtheorem{theo}{Theorem}[section]
\newtheorem{prop}[theo]{Proposition}
\newtheorem{lem}[theo]{Lemma}
\newtheorem{cor}[theo]{Corollary}
\newtheorem*{thm*}{Theorem}
\theoremstyle{remark}
\newtheorem{rem}[theo]{Remark}

\newtheorem*{ex*}{Example}
\theoremstyle{definition}

\numberwithin{equation}{section}
\newcommand{\be}{\begin{equation}}
\newcommand{\ee}{\end{equation}}
\newcommand{\ba}{\begin{aligned}}
\newcommand{\ea}{\end{aligned}}
\newcommand{\beqn}{\begin{equation}}
\newcommand{\eeqn}{\end{equation}}
\newcommand{\bear}{\begin{eqnarray}}
\newcommand{\eear}{\end{eqnarray}}
\newcommand{\bean}{\begin{eqnarray*}}
\newcommand{\eean}{\end{eqnarray*}}
\newcommand{\bal}{\begin{aligned}}
\newcommand{\eal}{\end{aligned}}

\title[KFP equation in a domain]{Constructive Krein-Rutman result for 
\\ Kinetic Fokker-Planck equations in a domain}

\author[K. Carrapatoso]{K. Carrapatoso}
\author[P. Gabriel]{P. Gabriel}
\author[R. Medina]{R. Medina}
\author[S. Mischler]{S. Mischler}

\address[K.~Carrapatoso]{Centre de Math\'ematiques Laurent Schwartz, \'Ecole polytechnique, Institut Polytechnique de Paris, 91128 Palaiseau cedex, France}
\email{kleber.carrapatoso@polytechnique.edu}

\address[P.~Gabriel]{Institut Denis Poisson, Université de Tours, Université d'Orléans, CNRS, Parc de Grandmont, 37000 Tours, France}
\email{pierre.gabriel@univ-tours.fr}

\address[R.~Medina]{Centre de Recherche en Math\'ematiques de
  la D\'ecision (CEREMADE, CNRS UMR 7534),
  Universit\'es PSL \& Paris-Dauphine, Place de Lattre de
  Tassigny, 75775 Paris 16, France}
\email{richard.medina-rodriguez@dauphine.psl.eu}

\address[S.~Mischler]{Centre de Recherche en Math\'ematiques de
  la D\'ecision (CEREMADE, CNRS UMR 7534),
  Universit\'es PSL \& Paris-Dauphine, Place de Lattre de
  Tassigny, 75775 Paris 16, France \& Institut Universitaire de France (IUF)}
\email{mischler@ceremade.dauphine.fr}

\date{}



\begin{document}

\begin{abstract} 
We consider a general Kinetic Fokker-Planck (KFP) equation in a domain with Maxwell reflection condition on  the boundary, not necessarily with conservation of mass. 
 We establish the wellposedness in many spaces including Radon measures spaces, and in particular the existence and uniqueness of fundamental solutions. 
 We also establish a Krein-Rutman theorem with constructive rate of convergence in an abstract setting that we use for proving that  the solutions to the KFP equation converge toward the conveniently normalized first eigenfunction. 
 Both results use the  ultracontractivity of the associated  semigroup in a fundamental way. 

\end{abstract}

\maketitle

\tableofcontents
  
%
%
%
%
%
%

\section{Introduction}  
\label{sec:intro}

\subsection{The KFP equation in a domain}
\label{sec:intro-KFP}

In this paper, we consider the Kinetic Fokker-Planck (KFP) equation (also denominated sometimes as Kolmogorov equation or ultraparabolic equation) 
\beqn\label{eq:KFP} 
 \partial_t f + v \cdot \nabla_x f = \Delta_v f+ b\cdot \grad_v f+ cf \quad \hbox{on}\quad \UU  
\eeqn
on the function $f :  = f(t,x,v)$ depending on the time variable $t \ge 0$, the position variable $x \in  \Omega$, where $\Omega \subset \R^d$ is a suitably smooth bounded domain,  $d \ge 1$, 
and the velocity variable $v \in \R^d$. For $T \in (0,+\infty]$, we use the shorthands $ \UU := (0,T) \times \OO$, $\OO := \Omega \times \R^d$. We assume   that 
\beqn\label{eq:Assum-bc0}
b = b(x,v) \in \R^{d}, \quad c = c (x,v) \in \R, 
\eeqn
each of these functions being at least in $\Lloc^\infty(\OO)$. 
We complement the above KFP evolution equation with the Maxwell type reflection condition on  the boundary 
\beqn\label{eq:KolmoBdyCond}
\gamma_{\!-} f  = \RRR \gamma_{\!+}  f  = \iota_S \SSS  \gamma_{\!+}  f + \iota_D \DDD  \gamma_{\!+}  f     \quad\hbox{on}\quad \Gamma_-,
\eeqn
and with an initial condition 
\beqn\label{eq:initialDatum} 
f(0,x,v) = f_0(x,v) \quad \hbox{on}\quad \OO.
\eeqn

Here $\Gamma_-$ denotes the  incoming part of the boundary, $\SSS$ denotes the specular reflection operator, $\DDD$ denotes the diffusive reflection operator, $\iota_S$ and $\iota_D $ are nonnegative coefficients. 
More precisely, we assume that $\Omega := \{ x \in \R^d; \,  \delta(x) > 0 \}$  
for a $W^{2,\infty}(\R^d)$ function $\delta$ such that $\delta (x) := \mathrm{dist}(x,\partial\Omega)$ on a neighborhood of  the boundary set $\partial\Omega$ and thus $n_x = n(x) :=   - \nabla \delta(x)$  
coincides with the  unit normal outward vector field on $\partial\Omega$. We next define $\Sigma_\pm^x := \{ {v} \in \R^d; \pm \, {v} \cdot n_x > 0 \}$ the sets of outgoing ($\Sigma_+^x$) and incoming ($\Sigma_-^x$) velocities at point $x \in \partial\Omega$, then the sets
$$
 \Sigma_\pm := 
 \{ (x,{v}); \, x \in \partial\Omega, \, {v} \in \Sigma^x_\pm \}, 
 \quad
 \Gamma_\pm := (0,T) \times \Sigma_\pm,
$$
and finally   the outgoing and incoming trace functions $\gamma_\pm f := \mathbf{1}_{\Gamma_\pm}  \gamma f$.
The specular reflection operator $\SSS$ is defined by 
\beqn
\label{eq:FPK-def_Gamma}
(\SSS g) (x,v)  :=  g (x , \VV_x v), \quad \VV_x v := v - 2 n(x) (n(x) \cdot v),
\eeqn
and the diffusive operator $\DDD$ is defined by 
\beqn
\label{eq:FPK-def_D}
(\DDD g) (x,v)  := \MMM_x(v) \widetilde g (x), \quad \widetilde g (x) := \int_{\Sigma^x_+} g(x,w) \, n(x) \cdot w \, d w,
\eeqn
where $\MMM_x$ stands for the Maxwellian function 
\beqn
\label{eq:FPK-def_M}
\MMM_x(v) := (2\pi \Theta_x)^{-(d-1)/2} \exp(-|v|^2/(2\Theta_x)) > 0,
\eeqn
associated to the wall temperature $\Theta_x$ which is assumed to satisfy  
\begin{equation}\label{eq:Assum-Theta}
  \Theta_x \in W^{1,\infty}(\Omega), \quad 0 < \Theta_* \le \Theta_x \le \Theta^* < \infty.
\end{equation}
It is worth observing that $\MMM_x$ is  conveniently normalized in such a way that $\widetilde \MMM_x=1$. Denoting the 
accommodation coefficient $\iota := \iota_S + \iota_D$, we assume 
$$
\iota_S, \iota_D, \iota :  \partial\Omega \to [0,1].
$$

\medskip

Let us introduce some notations and then discuss some particular cases. 
In view of \eqref{eq:KFP}, we define the interior collisional operator 
\beqn\label{eq:intro-defCCCf}
\CCC f :=  \Delta_v f+ b\cdot \grad_v f+ cf 
\eeqn
and next the (full) interior operator
\beqn\label{eq:intro-defLLLf}
\LLL := \TTT + \CCC, \quad \TTT := - v \cdot \nabla_x.
\eeqn
We name {\it microscopic or interior  mass conservative} case,  the case when  
$$
\LLL^*1 = \CCC^*1 = 0, \ \hbox{ or equivalently } \ c = \Div b, 
$$
and we name  {\it macroscopic or boundary mass conservative} case, the case when 
$$
\RRR^*1 = 1,  \ \hbox{ or equivalently } \ \iota \equiv 1.
$$
Here and below, the operators $\CCC^*$, $\LLL^*$ and $\RRR^*$ denote the (formal) dual operators. 
  It is worth emphasizing that we always have $\RRR^* 1 \le 1$ from the very definition \eqref{eq:KolmoBdyCond}, \eqref{eq:FPK-def_Gamma}, \eqref{eq:FPK-def_D} and the assumption $\iota\le1$, so that mass is never added  from the boundary, it is only (possibly partially) returned. The boundary condition \eqref{eq:KolmoBdyCond} corresponds to the \emph{pure specular reflection} boundary condition when  $\iota = \iota_S \equiv 1$
and it corresponds to the \emph{pure diffusive} boundary condition when $\iota = \iota_D \equiv 1$.
When both mass conservation conditions are fulfilled then equation \eqref{eq:KFP}--\eqref{eq:KolmoBdyCond} is mass conservative, meaning that any solution (at least formally) satisfies 
$$
\int_\OO f(t,x,v) \, dxdv = 
\int_\OO f_0(x,v) \, dxdv, \quad \forall \, t \ge 0. 
$$
We name {\it equilibrium} or {\it detailed balance condition} the case when  the maxwellian $\MMM_x=\MMM$ with constant temperature $\Theta_x=\Theta$ is a stationary state   for each operator separately, namely 
$$
\LLL \MMM =  \CCC \MMM = 0, \quad \RRR \MMM = \MMM.
$$
When $\Theta \equiv 1$, that corresponds to the situation when $\iota \equiv 1$ ($\iota_S$ and $\iota_D$ can be space dependent) and $\CCC$ is the usual harmonic Fokker-Planck operator with  
$b(x,v)=v$, $c(x,v)=d$,
that is  
$$
\CCC f := \Delta f + \Div( v f). 
$$
This very specific but physically motivated situation has been studied in the recent paper \cite{CM-KFP**} where, in particular, a constructive exponential stability result is established.

On the other hand, in the general situation when at least one of the two above conservations fails, we rather look for an eigentriplet $(\lambda_1, f_1, \phi_1)$ satisfying $\lambda_1 \in \R$, $f_1\geq0$, $\phi_1\geq0$ and
\beqn\label{eq:1stEVP}
\LLL f_1 = \lambda_1 f_1, \quad \gamma_{_-} f_1 = \RRR \gamma_{_+} f_1, 
\quad \LLL^* \phi_1 = \lambda_1 \phi_1, \quad\gamma_{_+} \phi_1 = \RRR^*  \gamma_{_-} \phi_1. 
\eeqn
This issue has been tackled recently in \cite{sanchez:hal-04093201} with the restriction $\Theta_x = \Theta$ is a constant, where the existence and uniqueness of such eigenelements have been established as well a non-constructive exponential asymptotical stability of the associated eigenfunction $F = e^{\lambda_1 t} f_1$. 

We refer to \cite{MR1503147,MR222474,MR771811,MR875086,MR1200643,MR1634851,MR1949176,MR2721875,MR2562709,MR4253803}  for a general discussion and mathematical analysis of the kinetic Fokker-Planck equation set in the whole space or in a domain and to related problems.

\medskip

In the present paper, we carry on the analysis made in \cite{sanchez:hal-04093201,CM-KFP**} by establishing   the following results.

\smallskip
(1) We prove the existence and uniqueness of solutions in many weighted Lebesgue spaces by establishing dissipativity estimates on the associated operator and next growth bound on the corresponding semigroup. 
We also establish the existence and uniqueness of a fundamental solution. 
 
\smallskip
(2) We establish the ultracontractivity of the above semigroup associated to the evolution problem
\eqref{eq:KFP}--\eqref{eq:KolmoBdyCond}--\eqref{eq:initialDatum}, that is  some immediate gain of stronger Lebesgue integrability  and even  immediate gain of uniform bound. 

\smallskip
(3) We prove a constructive version of a Krein-Rutman-Doblin-Harris theorem providing existence, uniqueness and exponential asymptotic stability with constructive rate of the first eigentriplet for a general class of positive semigroup in an abstract framework. 

\smallskip
(4) We show that the KFP model addressed here satisfies the requirement of the above  Krein-Rutman-Doblin-Harris theorem and thus give a clear and constructive understanding of the large time behavior of the solutions.

\smallskip

These results generalize or make more accurate  some previous similar  known results.

\subsection{Confinement in the velocity variable and admissible weight functions}
\label{subsec:confinement_v} 
 
We introduce additional assumptions on $b$ and $c$ in order that the  interior collisional operator $\CCC$ provides a convenient  velocity confinement mechanism. 
We first assume  
\beqn\label{eq:Assum-bc1}
\liminf_{|v| \to \infty} \inf_\Omega b \cdot \frac{v }{ |v|} = + \infty, 
\quad \frac{ |b|  }{ \langle v \rangle}, \Div_v b, c = \OO \Bigl(b \cdot \frac{v }{ |v|^2} \Bigr), 
\eeqn
where here and below we use the notation $\langle v\rangle =\sqrt{1+|v|^2}$.
These two conditions are fundamental in order that first the interior collisional operator $\CCC$ and next the full operator are dissipative and even have a discrete spectrum in convenient functional spaces. 
In order to identify these spaces and to make the discussion simpler, we make more precise the confinement conditions by assuming that there exist $R_0, b_0,b_1 > 0$, $\gamma > 1$ and for any $p \in [1,\infty]$ there exists {$k^*_p \ge 0$} such that 
\beqn\label{eq:Assum-bc2}
\forall \, x \in \Omega, \, v \in B_{R_0}^c, \quad  b_0 |v|^\gamma \le b \cdot v \le b_1 |v|^\gamma, \quad c - \frac1p \Div_v b \le k^*_p \, b \cdot \frac{v }{ |v|^2}. 
\eeqn

\smallskip
We now introduce the class of the so-called {\it  admissible} weight functions $\omega : \R^d \to (0,\infty)$ we will work with, which will be either a polynomial weight function 
\beqn\label{eq:omega-poly}
\omega = \la v \ra^k := (1 + |v|^2)^{k/2}, \quad k >  k_* :=    \max( k'_1,k^*_\infty), 
\eeqn
 $k'_1:= k_1 +d/2 + \max(1, \gamma/2-1)$, $k_1:= \max(k^*_1,d+2)$,  
and we set $s := 0$ in that case, or either an exponential weight function 
\beqn\label{eq:omega-expo}
\omega = \exp(\zeta \la v \ra^s), 
\eeqn
with the restrictions
\be\label{eq:cond-zeta-s}
\ba
 &s < \min (\gamma,2), \ \zeta > 0; \quad s=\gamma < 2, \ \zeta \in (0,b_0/2); \\
 & s=\gamma = 2, \ \zeta \in (0,\min(1/\Theta^*,b_0)/2) ; \quad s=2 < \gamma , \ \zeta \in (0,1/(2\Theta^*)).
\ea
\ee
In order to explain that choice of weight functions, we introduce  the function
\beqn\label{def:varpi}
\varpi = \varpi^\CCC_{\omega,p}(x,v) := 2 \left(1 - \frac 1p \right) \frac{|\nabla_v\, \omega|^2 }{ \omega^2}  + \left(\frac 2p -1\right)\frac{ \Delta_v \, \omega } \omega   
- b  \cdot \frac{\nabla_v \omega}{\omega}
+  c  - \frac{1 }{ p}  \Div_v b,
\eeqn
which is the key quantity in order to  reveal the  velocity confinement mechanism. We may   notice that for  $\omega := \langle v \rangle^\ell e^{\zeta |v|^s}$, with $\ell \in \R$ and $s,\zeta \ge 0$, and because of the second condition in \eqref{eq:Assum-bc1}, we have 
\bear
\label{eq:varpi-expo}
&&\varpi^\CCC_{\omega,p} 
 \underset{|v|\to \infty}{\sim} (s\zeta)^2 |v|^{2s-2}   -  s \zeta b \cdot v |v|^{s-2}  \quad \hbox{if} \quad s > 0, 
\\
&&\varpi^\CCC_{\omega,p}   \underset{|v|\to \infty}{\sim}   c - \frac{1 }{ p}  \Div_v b  -\ell b \cdot v |v|^{-2}  \quad \hbox{if} \quad s = 0. 
\eear
As a consequence, whatever is the value $\gamma > 1$, we have $(\varpi^\CCC_{\omega,p})_+ \in L^\infty(\OO)$ for any admissible weight function, what is the key information in order to establish a growth estimate in the corresponding weighted $L^p$ space.
Setting $\varsigma:=\gamma+s-2$, we will call {\it  strongly confining} any admissible weight function satisfying $\varsigma>0$ and {\it weakly confining} the admissible weight functions for which $\varsigma\leq0$.
In the strongly confining case, we have
\beqn\label{eq:strongCadmissWeight}
\limsup_{|v| \to \infty}\sup_\Omega \varpi^\CCC_{\omega,p} = - \infty,
\eeqn
which gives a key information on the spectrum of $\CCC$ in the corresponding functional space.
Notice that we are in this strongly confining case when $\gamma > 2$ 
and for any exponential weight function with exponent $s \in (2-\gamma,\gamma]$ when $\gamma \in (1,2]$.
For further references, we also notice that 
\beqn
\label{eq:cond-Lpomega}
\sup_{\Omega} \MMM_x\, \omega\langle v \rangle  \in (L^1 \cap L^\infty)(\R^d) , \quad 
 \omega^{-1} \langle v \rangle \in (L^1 \cap L^2)(\R^d), 
\eeqn
for any admissible weight function because of the restrictions $k_* \ge d+1$, $s \le 2$ and $\zeta < 1 / (2\Theta^*)$ when $s = 2$. The bound \eqref{eq:cond-Lpomega}  provides the compatibility of the weight function  with the boundary condition.

\subsection{The main results}
\label{sec:intro-MainResults} 

In order to state our main results, we need to introduce some functional spaces. 
For a given measure space $(E,\EEE,\mu)$, a weight function $\rho : E \to (0,\infty)$ and an exponent $p \in [1,\infty]$, we define the weighted Lebesgue space $L^p_\rho$ 
associated to the norm 
$$
\| g \|_{L^p_\rho} = \| \rho g \|_{L^p}. 
$$
We also define  $M^1_{\omega,0}$ as the space of Radon measures $g$ on $\OO$  with vanishing mass at the boundary, that is such that $|g\omega|(\OO \backslash \OO_\eps) \to 0$ as $\eps \to 0$, where, for any $\eps > 0$, 
\beqn\label{def:OmegaEpsOeps}
\Omega_{\eps} :=\{x\in \Omega  \cap B_{\eps^{-1}}, \delta(x)   >\eps\},
\quad \OO_{\eps} := \Omega_{\eps}\times B_{\eps^{-1}}. 
\eeqn

We first state a general existence and uniqueness result for the kinetic
Fokker-Planck equation \eqref{eq:KFP}, \eqref{eq:KolmoBdyCond}, \eqref{eq:initialDatum}. 

\begin{theo}[Existence and uniqueness]\label{theo:existsLpM1} 
 We make the above assumptions on $\Omega$, $\Theta$, $b$ and $c$, in particular \eqref{eq:Assum-bc0}, \eqref{eq:Assum-Theta}, \eqref{eq:Assum-bc1} and \eqref{eq:Assum-bc2} hold.
For  any admissible weight function $\omega$  and any initial datum $f_0 \in L^p_{\omega}$, $p \in [1,\infty]$, or  $f_0 \in M^1_{\omega,0}$, there exists a unique global weak solution $f$ to the kinetic
Fokker-Planck equation \eqref{eq:KFP}--\eqref{eq:KolmoBdyCond}--\eqref{eq:initialDatum}. In particular, for any $(x_0,v_0) \in \OO$, there exists a unique fundamental solution associated to the
initial datum $f_0 := \delta_{(x_0,v_0)}$. 

\end{theo}

The precise sense of solution will be given in Proposition~\ref{prop-KFP-L2primal} (see also Theorem~\ref{theo:WP-KFP&KFP*}) in a $L^2$ framework, in Theorem~\ref{theo:E-UinLp} in a general $L^p$ framework, and in Theorem~\ref{theo:E-UinM1} in a Radon measures framework. This result extends the existence and unique result of \cite[Theorem~11.5]{sanchez:hal-04093201} stated in a more restrictive  $L^2$ framework (see also \cite{MR875086,AAMN2021} for further previous results). The $L^2$ framework is mainly based on Lions' variant of the Lax-Milgram theorem \cite[Chap~III,  \textsection 1]{MR0153974}, as used in \cite{MR875086,AAMN2021},  a trace theory developed in \cite{MR1776840,MR2721875,sanchez:hal-04093201,CM-Landau**} and boundary estimates in the spirit of \cite{MR1301931,MR2721875,MR4179249}. The growth estimate is obtained by cooking up a modified but equivalent weight function for which the dissipativity of the full operator can be established. 
On the other hand,  the general Lebesgue framework and the Radon measures framework are more involved and are also based on the ultracontractivity theorem below as well as some arguments adapted from the parabolic equation as developed in \cite{MR1025884,MR1489429,MR1876648}.
It is worth mentioning that  the well-posedness and some regularity issues for the KFP equation set in the torus have been obtained in \cite{AAMN2021}. 
For the whole space setting, we refer to the recent works \cite{AIN1,AIN2} and the references therein. Finally, the KFP equation in a bounded domain has been considered in \cite{MR4253803,MR4527757,MR3778533}.

\smallskip

We next consider the first eigenvalue problem and the longtime behavior providing a quantitative answer to the first eigenelements issue.

\begin{theo}[Long time asymptotic]\label{theo:KR}
Under the assumptions of Theorem~\ref{theo:existsLpM1},
there exist two weight functions $\omega_1,m_1$ and an exponent $r > 2$  with $L^r_{\omega_1} \subset (L^2_{m_1})' $ 
such that there exists a unique eigentriplet $(\lambda_1, f_1,\phi_1) \in \R \times L^r_{\omega_1} \times L^2_{m_1}$ satisfying the first eigenproblem  \eqref{eq:1stEVP} 
together with the normalization condition  $\| \phi_1 \|_{L^2_{m_1}} =1$, $\langle \phi_1, f_1 \rangle  = \langle \phi_1, f_1 \rangle_{L^2_{m_1},(L^2_{m_1})'} = 1$. These eigenfunctions are continuous functions and they also satisfy
\beqn\label{eq:theoKR-strictpo&Linftybound}
  0 < f_1 \lesssim \omega^{-1}, \quad 
0 < \phi_1 \lesssim \omega   \quad\hbox{on}\quad \OO, 
\eeqn
for any admissible weight function $\omega$. Furthermore, there exist some constructive constants $C \ge 1$ and $\lambda_2 < \lambda_1$ such that for any  strongly confining admissible weight function $\omega$, any exponent $p \in [1,\infty]$ and any 
initial datum  $f_0 \in L^p_\omega$, the associated solution $f$ to the kinetic
Fokker-Planck equation \eqref{eq:KFP}, \eqref{eq:KolmoBdyCond}, \eqref{eq:initialDatum} satisfies 
\beqn\label{eq:KRTh}
\| f(t) - \langle f_0,\phi_1 \rangle f_1 e^{\lambda_1 t}  \|_{L^p_\omega} \le C e^{\lambda_2 t} 
\|  f_0 - \langle f_0,\phi_1 \rangle f_1 \|_{L^p_\omega}, 
\eeqn
for any $t \ge 0$.
 \end{theo}
 This result  improves the recent work \cite[Section~11]{sanchez:hal-04093201} (see  in particular \cite[Theorem 11.6]{sanchez:hal-04093201}, \cite[Theorem 11.8]{sanchez:hal-04093201} and \cite[Theorem 11.11]{sanchez:hal-04093201}) by slightly generalizing the framework to a position dependent wall temperature and by providing a  fully constructive approach for the exponential stability of the first eigenfunction. 
We refer to the previous works \cite{MR4347490} (partially based on  \cite{MR2248986,MR3382587,MR4412380}) and \cite{MR4556285,Bernard2023} where similar results are established for the same kind of equation 
 in a domain with no-flow boundary condition. 
 We also emphasize that in the conservative case, many works have been done related to hypocoercivity and constructive rate of convergence to the steady state in  \cite{MR1787105,MR1946444,MR1969727,MR2034753,MR2130405,MR2562709} or more recently in 
\cite{MR3324910,MR3488535,MR4069622,MR4113786,AAMN2021,CM-KFP**}.   From a technical point of view, this result is a consequence of the abstract version of the Krein-Rutman-Doblin-Harris theorem that will be  presented in section~\ref{subsec:KR-statement} (and which is really in the spirit of the recent work \cite{sanchez:hal-04093201}) together with the ultracontractivity property stated below and the Harnack estimates established in \cite{MR3923847}.

\medskip 

Both the above well-posedness and the longtime behavior results are based on the following ultracontractivity property.  
 
\begin{theo}[Ultracontractivity]\label{theo:ultra}
There exist $\Theta, C > 0$ and $\kappa \ge 0$ such that any solution $f$ to the KFP equation  
\eqref{eq:KFP}--\eqref{eq:KolmoBdyCond}--\eqref{eq:initialDatum}  satisfies  
\beqn\label{eq:theo-ultra}
 \| f(T,\cdot) \|_{L^\infty_{\omega}} \le C \frac{e^{\kappa T} }{ T^\Theta} \|f_0 \|_{L^1_\omega}, \quad \forall \, T > 0, 
\eeqn
for any  strongly confining admissible weight function $\omega$. 
\end{theo}

This result slightly improves and generalizes \cite[Theorem~1.1]{CM-KFP**} which establishes a similar result in the conservative case.
A variant in the weakly confining case will be given in Corollary~\ref{prop:EstimL1L2-omegaWeak-final}.
The proof is very alike the one of \cite[Theorem~1.1]{CM-KFP**} although some steps are slightly simplified. 
The strategy is based on Nash's gain of integrability argument~\cite{MR0100158} which is performed however on a time  integral inequality as in  Moser's work~\cite{MR159139}, and is then more convenient in order to use the interior gain of  integrability deduced from Bouchut's regularity result  \cite{MR1949176}
following the way paved by \cite[Theorem 6]{MR3923847} for proving a  somehow similar local version. Contrary to the last reference, the gain of integrability is not formulated locally in $x,v$ and integrated in time but globally in $x,v$ and pointwisely in time as in the the ultracontractivity theory of Davies and Simon \cite{MR0712583,MR0766493}. 
Exactly as  in  \cite{CM-KFP**}, the key argument consists in exhibiting a suitable  twisted weight function which is  somehow slightly more elaborated than the one used during the proof of the growth estimates in Theorem~\ref{theo:existsLpM1}. 
  In Theorem~\ref{theo:existsLpM1} and Theorem~\ref{theo:ultra}, the boundedness assumption on $\Omega$ is really only needed in the proof of the uniqueness of the solution in the $L^p_\omega$ framework and it is likely that it can be removed. Here, we do not try to generalize these results to the case of an unbounded domain, 
 see however 
\cite{BMM-HalfSpace**} for partial results in that direction.

\subsection{Organization of the paper} 
Section~\ref{sec:WeightLp} is dedicated to the proof of some weighted $L^p$ a priori growth bounds for the primal and the dual problems. These estimates and the well-posedness in a $L^2$ framework for the same problems 
 (and thus part of Theorem~\ref{theo:existsLpM1}) are
 then established rigorously in Section~\ref{sec-Wellpodeness}. Section~\ref{sec:ultra} is devoted to the proof of the ultracontractivity property as stated in Theorem~\ref{theo:ultra}. 
 In section~\ref{sec-WellposednessGal} we come back to the  well-posedness in a general framework and we end the proof of Theorem~\ref{theo:existsLpM1}.
 Section~\ref{sec-Harnack} is dedicated to the proof of the Harnack inequality associated to our equations. 
 In Section~\ref{sec:KR} we state and prove a constructive version of the Krein-Rutman theorem and deduce Theorem~\ref{theo:KR}.

\section{Weighted $L^p$ a priori growth estimates}  \label{sec:WeightLp}

 This section is devoted to the proof of some a priori growth estimates in weighted $L^p$ spaces  for solutions
 to the KFP equation \eqref{eq:KFP}--\eqref{eq:KolmoBdyCond}--\eqref{eq:initialDatum} and its formal adjoint. 
The formal computations that we use for the sake of clarity can be rigorously justified for the solutions we build in Theorem~\ref{theo:WP-KFP&KFP*}.

\subsection{A priori estimates for the primal problem}

We recall that for two functions $f,\omega : \R^d \to \R_+$ and $p \in [1,\infty)$, we have 
\beqn\label{eq:identCCCffp}
\int_{\R^d} (\CCC f) f^{p-1} \omega^p dv =
 - \frac{4 (p-1)}{ p^2}  \int_{\R^d} |\nabla_v (f\omega)^{p/2} |^2  + \int f^p \omega^p \varpi^{\CCC}_{\omega,p},
\eeqn
with $\varpi^{\CCC}_{\omega,p}$ defined in \eqref{def:varpi}, 
what can be established by mere repeated  integrations by part, see for instance \cite[Lemma~7.8]{sanchez:hal-04093201} and the references therein. 
From the definition of the admissible weight functions in Section~\ref{subsec:confinement_v} and for further references, we may observe that the large velocity asymptotic of $\varpi^{\CCC}_{\omega,p}$ is controlled by 
\beqn\label{eq:varpi-asymptotic}
\limsup _{|v| \to \infty} \bigl( \sup_\Omega \varpi^{\CCC}_{\omega,p} - \varpi^\sharp_{\omega,p}) \le 0, 
  \quad \varpi^\sharp_{\omega,p} := - b_0^\sharp \langle v \rangle^\varsigma, \quad \varsigma := {\gamma+s-2}, 
\eeqn
with  $b^\sharp_0 > 0$ given by 
\be\label{eq:b0sharp}
\ba
 &b_0^\sharp := (k-k_p) b_0 \ \hbox{ if } \ s=0, \\
 &b_0^\sharp := b_0 s \zeta  \ \hbox{ if } \ s \in (0,\gamma), \\
 &b_0^\sharp := b_0 s \zeta - (s\zeta)^2  \ \hbox{ if } \ s = \gamma. 
 \ea
\ee
In a more quantitative way, for any $\vartheta \in (0,1)$, there exists $\kappa',R' > 0$ such that 
\beqn\label{eq:varpiC-varpisharp}
\sup_\Omega \varpi^{\CCC}_{\omega,p} \le \kappa' \chi_{R'} + \chi_{R'}^c \vartheta \varpi^\sharp_{\omega,p},
\eeqn
where $\chi_R(v) := \chi(|v|/R)$, $\chi \in C^2(\R_+)$, $\mathbf{1}_{[0,1]} \le \chi \le \mathbf{1}_{[0,2]}$, and $\chi^c_R : 1 - \chi_R$.

\begin{lem}\label{lem:EstimKFPLp12}
For any admissible weight function $\omega$, there exist $\kappa \ge 0$ and $C \ge 1$ such that for both exponents $p = 1, 2$, any solution $f$ to the KFP equation \eqref{eq:KFP}--\eqref{eq:KolmoBdyCond}--\eqref{eq:initialDatum} satisfies, at least formally,  
\beqn\label{eq:lem:EstimKFPLp12}
\| f_t \|_{L^p_\omega} \le C e^{\kappa t} \| f_0 \|_{L^p_\omega} , \quad \forall \, t \ge 0.
\eeqn
\end{lem}

The proof is based on moment estimates introduced in \cite[Proposition~3.3]{CM-Landau**} for the case $p=2$ and in \cite[Lemma~2.3]{CM-KFP**} for the case $p=1$, which are reminiscent of $L^1$ hypodissipativity techniques, see e.g.\ \cite{MR3489637,MR3779780,MR4179249}, and 
which are based on the usual multiplier used in order to control the diffusive reflection operator in previous works on the Boltzmann equation, see e.g.\ \cite{MR1301931,MR1776840,MR2721875,MR4179249}. 
For further references, we define the formal adjoints
\beqn\label{def:CCC*}
\LLL^* := v \cdot \nabla_x + \CCC^*, \quad \CCC^* g:=   \Delta_v g - b \cdot \nabla_v g + (c -\Div b) g.
\eeqn

\begin{proof}[Proof of Lemma \ref{lem:EstimKFPLp12}] 
Consider $ 0 \le f_0 \in L^p(\omega)$ and  $f = f(t,x,v) \ge 0$ a  solution to the Cauchy problem \eqref{eq:KFP}--\eqref{eq:KolmoBdyCond}--\eqref{eq:initialDatum}.
We introduce the modified weight functions $\omega_A$ and $\widetilde \omega$ defined by 
\beqn\label{def:omegaA}
 \omega_A^p := \MMM_x^{1-p} \chi_A + \omega^p (1-\chi_A),
 \qquad  \widetilde \omega^p :=   \left( 1 + \frac{ 1  }{ 2}  \frac{ n_x \cdot   v}{\la v \ra^4} \right) \omega_A^p,
\eeqn
with $A \ge 1$ to be chosen later, 
 $\hat v := v/\langle v \rangle$ and $\widetilde{v} :=  \hat v/\langle v \rangle $. It is worth emphasizing that 
\beqn\label{eq:omega&omegatilde}
c_A^{-1} \omega \le \tfrac12\omega_A \le \widetilde \omega  \le \tfrac32 \omega_A \le c_A \omega, 
\eeqn
with $c_A  \in (0,\infty)$. 
We then write 
\beqn\label{eq1:SLestimL1}
\frac1p\frac{d}{dt} \int_{\OO} f^p \, \widetilde \omega^p 
= \int_{\OO} (\CCC f)  f^{p-1} \widetilde \omega^p  
+   \frac1p  \int_{\OO} f^p \, \TTT^* \widetilde \omega^p 
- \frac1p  \int_\Sigma (\gamma f)^p \, \widetilde \omega^p   (n_x \cdot v),
\eeqn
and we estimate each term separately below.

\medskip\noindent
\textit{Step 1.} 
We first compute separately each contribution of the boundary term in \eqref{eq1:SLestimL1}, namely we write
$$
- \int_\Sigma (\gamma f)^p \, \widetilde \omega^p   (n_x \cdot v) = B_1 + B_2
$$
with
\bean
B_1 :=  - \int_\Sigma (\gamma f)^p \omega_A^p  \, n_x \cdot v, \quad
B_2 :=  -  \frac{ 1  }{ 2}  \int_\Sigma (\gamma f)^p  {(n_x \cdot \tilde v)^2 } \,  \omega_A^p. 
\eean

On the one hand, we have 
$$
\begin{aligned}
&B_1 = - \int_{\Sigma_{+}} (\gamma_+ f)^p \omega^p_A  |n_x \cdot v| 
+ \int_{\Sigma_{-}}  \{ \iota_S \SSS \gamma_+ f + \iota_D \DDD \gamma_{+} f   \}^p  \omega^p_A |n_x \cdot v| 
 \\
&\qquad \le  - \int_{\Sigma_{+}} (\gamma_+ f)^p \omega^p_A  (n_x \cdot v)_+ +
 \int_{\Sigma_-}  \iota_S(\SSS \gamma_{+} f)^p\omega_A^p (n_x \cdot v)_-
+ \int_{\Sigma_-} \iota_D ( \widetilde{\gamma_{+} f})^p \MMM_x^p \omega_A^p (n_x \cdot v)_-
\\
&\qquad \le -  \int_{\Sigma_+}  
\iota_D(  \gamma_{+} f)^p\omega_A^p (n_x \cdot v)_+
+ \int_{\partial\Omega} \iota_D ( \widetilde{\gamma_{+} f})^p K_1(\omega_A), 
\end{aligned}
$$
where we have used the convexity of the mapping $s \mapsto s^p$ in the first line, we have made the change of variables $  v \mapsto \VV_x v$ in the second integral in the second line
and we have set 
\begin{equation}\label{eq:K1omega}
K_1(\omega_A ) := \int_{\R^d}  \MMM_x^p \, \omega^p_A  \, (n_x \cdot v)_{-} \, dv < \infty.
\end{equation}
For $p=1$, we observe that $\omega_A \ge 1$ and we set $K_2(\omega_A) := 1$.
For $p = 2$, using the Cauchy-Schwarz inequality, we have 
 $$
(\widetilde{\gamma_{+} f})^2 \le K_2(\omega_A)^{-1}  \int_{\R^d} ( \gamma_{+} f)^2\omega_A^2 (n_x \cdot v)_+, 
$$
with 
$$
K_2 (\omega_A)^{-1} := \int_{\R^d}   \omega_A^{-2}  (n_x \cdot v)_+ \, dv < \infty.
$$
In both case we deduce  
\beqn\label{eq:lemL12-estimB1}
 B_1 
\le  \int_{\partial\Omega}  \iota_D ( K_1(\omega_A)-K_2(\omega_A)) (\widetilde{\gamma_+ f})^p.
\eeqn

On the other hand, using the boundary condition \eqref{eq:KolmoBdyCond} and making change of variables $  v \mapsto \VV_x v$, there holds
$$
\begin{aligned}
B_2 
&= -\frac12 \int_{\Sigma_+} (\gamma_+ f)^p \omega_A^p (n_x \cdot \tilde v)^2
 -\frac12 \int_{\Sigma_-} (\iota_S \SSS\gamma_+ f + \iota_D \DDD \gamma_+ f)^p \omega_A^p (n_x \cdot \tilde v)^2 \\
&= -\frac12 \int_{\Sigma_+} (\gamma_+ f)^p \omega_A^p (n_x \cdot \tilde v)^2
 -\frac12 \int_{\Sigma_+} (\iota_S \gamma_+ f + \iota_D \DDD \gamma_+ f)^p \omega_A^p (n_x \cdot \tilde v)^2.
\end{aligned}
$$
When $p=1$, denoting
\begin{equation}\label{eq:K0}
K_0 (\omega_A) := \frac12 \int_{\R^d} \MMM_x (n_x \cdot \tilde v)_+^2 \, \omega_A \, dv <\infty ,
\end{equation}
we therefore have 
\bean
B_2 
&\le& -\frac12 \int_{\Sigma_{+}} \iota_D (\DDD \gamma_{+} f )   (n_x \cdot \tilde v)^2  \omega_A
= -    \int_{\partial\Omega} \iota_D  K_0 (\omega_A)   (\widetilde{\gamma_{+} f}).
\eean
On the other hand,  when $p=2$, denoting 
\begin{equation}\label{eq:K0bis}
K_0 (\omega_A)^{-1} :=  2 \int_{\Sigma_+^x}  \, \langle v \rangle^4 \, \omega_A^{-2} \, dv < \infty ,
\end{equation}
and thanks to the Cauchy-Schwarz inequality, we have 
\bean
B_2 
\le - \frac12 \int_{\Sigma_{+}}  (\gamma_{+} f)^2    (n_x \cdot \tilde v)^2  \omega_A^2
\le
-   \int_{\partial\Omega}  K_0 (\omega_A)   (\widetilde{\gamma_{+} f})^2.
\eean
 In both cases $p=1$ and $p=2$, we have established
\beqn\label{eq:lemL12-estimB}
 B \le  \int_{\partial\Omega}  \iota_D \left[ K_1(\omega_A)-K_2(\omega_A)-K_0(\omega_A) \right] (\widetilde{\gamma_+ f})^p ,
\eeqn
and we observe that 
\bean
\lim_{A \to \infty} K_1(\omega_A) = \lim_{A \to \infty}  K_2(\omega_A) = 1, 
\eean
thanks to the dominated convergence theorem, the normalization condition on $\MMM_x$  and the condition \eqref{eq:cond-Lpomega}.
We similarly have 
$$
\lim_{A \to \infty} K_0 (\omega_A)  = \frac12\int_{\R^d}\MMM_x  (n_x \cdot \tilde v)_+^2 \, dv \ge C_1(\Theta_*,\Theta^*) > 0,
$$
when $p=1$, and 
$$
\lim_{A \to \infty} K_0 (\omega_A)^{-1}  =  2 \int_{\Sigma_-^x} \MMM_x \la v \ra^4 \, dv \le C_2(\Theta_*,\Theta^*) < \infty,
$$
when $p=2$. All these convergences being uniform in $x \in \partial\Omega$,  we can choose $A > 0$ large enough in such a way that 
$$
K_1(\omega_A)-K_2(\omega_A)-K_0(\omega_A) \le 0,
$$
and thus 
$$
- \int_\Sigma (\gamma f)^p \, \widetilde \omega^p   (n_x \cdot v) \le 0.
$$ 

\medskip\noindent
\textit{Step 2.} 
We now deal with the first term at the right-hand side of \eqref{eq1:SLestimL1}. 
On the one hand from \eqref{eq:identCCCffp}, we have 
$$
\int_{\R^d} (\CCC f)  f^{p-1} \widetilde\omega^p   =
 -\frac{4(p-1)}{p}  \int_{\R^d} |\nabla_v (f^{p/2} \widetilde\omega^{p/2} ) |^2  + \int_{\R^d} f^p \widetilde\omega^p  \varpi^\CCC_{\tilde \omega,p},
$$
with 
\bean
\varpi^\CCC_{\tilde \omega,p}
:=  2\left(1-\frac{1}{p}\right) \frac{|\nabla_v \widetilde \omega|^2 }{ \widetilde\omega^2}  
+ \left(\frac 2p -1\right)\frac{ \Delta_v  \widetilde \omega }{\widetilde\omega}  
- b  \cdot \frac{\nabla_v \widetilde\omega}{\widetilde\omega}
+  c  - \frac{1}{p}  \Div_v b.
\eean
Defining, $\wp^p := 1 + \tfrac12 \tfrac{n_x \cdot v}{\la v \ra^4}$ and $\wp_A^p := 1 + \chi_A (\MMM^{1-p} \omega^{-p} - 1)$ so that $\widetilde\omega = \wp \omega_A$ and $\omega_A =   \wp_A \omega$, 
we compute 
\begin{equation}\label{eq:tildevarpi}
\varpi^\CCC_{\tilde \omega,p}
= \varpi_{A,p}   + 2 \frac{\nabla_v \omega_A}{ \omega_A}\cdot \frac{\nabla_v \wp  }{ \wp}
 + 2\left(1-\frac{1}{p}\right) \frac{|\nabla_v \wp |^2 }{ \wp^2}  
 +\left(\frac{2}{p}-1\right) \frac{\Delta_v \wp}{\wp} 
 - b  \cdot \frac{\nabla_v \wp}{\wp},
\end{equation}
with 
\bean
 \varpi_{A,p} :=    2\left(1-\frac{1}{p}\right) \frac{|\nabla_v \omega_A|^2 }{ \omega_A^2} 
+\left(\frac{2}{p}-1\right) \frac{\Delta_v \omega_A}{\omega_A}
 - b  \cdot \frac{\nabla_v \omega_A}{\omega_A} +  c  - \frac{1}{p}  \Div_v b , 
\eean
and next 
\begin{equation}\label{eq:varpi_Ap}
 \varpi_{A,p} =  \varpi_{\omega, p}^\CCC   + 2 \frac{\nabla_v \omega}{ \omega}\cdot \frac{\nabla_v \wp_A  }{ \wp_A}
 + 2\left(1-\frac{1}{p}\right) \frac{|\nabla_v \wp_A |^2 }{ \wp_A^2}  
 +\left(\frac{2}{p}-1\right) \frac{\Delta_v \wp_A}{\wp_A} 
 - b  \cdot \frac{\nabla_v \wp_A}{\wp_A} .
\end{equation}  
Because $\chi_A$ has compact support in the velocity variable, the same holds for all the terms except the first one at the right-hand side of \eqref{eq:varpi_Ap}, and thus 
 $$ 
 \bigl|  \varpi_{A,p} - \varpi^\CCC_{\omega,p}    \bigr|  
 \lesssim  \frac{1 }{ \langle v \rangle^4} \lesssim \frac{|\varpi^\sharp_{\omega,p}| }{ \langle v \rangle^{2+\gamma+s}}. 
 $$
Similarly,  observing that 
$$
 \frac{\nabla_v \omega_A  }{ \omega_A}\cdot \frac{\nabla_v \wp  }{ \wp} = \frac{\nabla_v \omega  }{ \omega}\cdot \frac{\nabla_v \wp  }{ \wp} +  \frac{\nabla_v  \wp_A  }{ \wp_A}\cdot \frac{\nabla_v \wp  }{ \wp},
 $$
 where $|\frac{\nabla_v \wp  }{ \wp}| \lesssim \la v \ra^{-4}$ and the second term is compactly supported, we have
 $$
 \bigl| \varpi^\CCC_{\tilde \omega,p} -  \varpi_{A,p}  \bigr| 
 \lesssim \left(1 + |b| + \frac{| \nabla \omega| }{ \omega} \right) \frac{1 }{ \langle v \rangle^4} \lesssim  \frac{|\varpi^\sharp_{\omega,p}| }{ \langle v \rangle^2}. 
 $$
 Combining the last two estimates together with \eqref{eq:varpiC-varpisharp}, we deduce that for any $\vartheta \in (0,1)$, there exists $\tilde\kappa,\tilde R > 0$ such that 
 \beqn\label{eq:varpiCtilde-varpisharp}
\sup_\Omega \varpi^{\CCC}_{\tilde\omega,p} \le \tilde\kappa  \chi_{\tilde R} + \chi_{\tilde R}^c \vartheta \varpi^\sharp_{\omega,p}. 
\eeqn

\smallskip\noindent
\textit{Step 3.}  
We finally deal with the second term at the right-hand side of \eqref{eq1:SLestimL1}. When $p=1$ we have
$$
v \cdot \nabla_x \widetilde \omega
= \frac12 ( \hat v \cdot D_x n_x \hat v) \frac{\omega_A}{\la v \ra^2} \lesssim  \frac{1 }{ \langle v \rangle^2} \widetilde \omega 
 \lesssim \frac{|\varpi^\sharp_{\omega,p}| }{ \langle v \rangle^{s+\gamma}} \widetilde \omega,
$$
 thanks to the regularity assumption on $\Omega$.  
On the other hand, when $p=2$, we first compute
$$
v \cdot \nabla_x (\widetilde \omega^2)
= \frac12 ( \hat v \cdot D_x n_x \hat v) \frac{\omega_A^2}{\la v \ra^2}
+ \left( 1 + \frac12 \frac{n_x \cdot v}{\la v \ra^4} \right) v \cdot \nabla_x (\omega_A^2).
$$
Since 
$$
\nabla_x (\omega_A^2) 
=  \chi_A \MMM_x^{-1} \left[ \frac{(d-1)}{2} \frac{\nabla_x \Theta_x}{\Theta_x} - \frac{|v|^2}{2} \frac{\nabla_x \Theta_x}{\Theta_x^2} \right],
$$
assumption \eqref{eq:Assum-Theta} together with the fact that $\chi_A$ is compactly supported and the regularity assumption on $\Omega$ as above imply
$$
v \cdot \nabla_x (\widetilde \omega^2) \lesssim \frac1{ \langle v \rangle^2} \widetilde \omega^2
\lesssim \frac{|\varpi^\sharp_{\omega,p}| }{ \langle v \rangle^{s+\gamma}} \widetilde \omega.
$$

\medskip\noindent
\textit{Step 4.} 
Coming back to \eqref{eq1:SLestimL1} and using Step~1, we deduce that 
\beqn\label{eq:disssipSL-Lp}
\frac1p\frac{d}{dt} \int_\OO f^p \widetilde \omega^p \le  -\frac{4(p-1)}{p}    \int_\OO |\nabla_v (f\widetilde \omega)^{p/2}|^2 + \int_\OO f^p \widetilde \omega^p \varpi^{\LLL}_{\tilde\omega,p}
\eeqn
for both $p=1,2$ and with 
\beqn\label{eq:disssipSL-LpBIS}
\varpi^\LLL_{\tilde\omega,p} := \varpi^\CCC_{\tilde\omega,p} +  \frac{1 }{ \widetilde \omega^p} v \cdot \nabla_x \widetilde \omega^p. 
\eeqn
Gathering the estimates \eqref{eq:varpiC-varpisharp} and those established in Step~2 and Step~3, we deduce that for any $\vartheta \in (0,1)$, there exists $\kappa,R > 0$ such that 
\beqn\label{eq:varpiL-varpisharp}
  \varpi^{\LLL}_{\tilde\omega,p} \le \kappa \chi_{R} + \chi_{R}^c \vartheta \varpi^\sharp_{\omega,p}.
\eeqn
In particular, $  \varpi^{\LLL}_{\tilde\omega,p}  \le \kappa$ and we immediately conclude thanks to Gr\"onwall's lemma and the comparison \eqref{eq:omega&omegatilde} between $\omega$ and $\widetilde\omega$. 
\end{proof}

\begin{rem}\label{rem:lem:EstimKFPLp12}
 For later references, we emphasize that the same proof works for establishing \eqref{eq:lem:EstimKFPLp12} with $p:=1$ and $\omega := \langle v \rangle^k$, $k \ge k_1 := \max(k^*_1,d+2)$. 
\end{rem}

\subsection{A priori estimates for the dual problem}

We establish now a similar exponential growth a priori estimate in a  general weighted $L^q$ framework, $q=1,2$,  for the dual  backward problem associated to 
\eqref{eq:KFP}--\eqref{eq:KolmoBdyCond}--\eqref{eq:initialDatum}. More precisely we consider the equation
\begin{equation}\label{eq:dualKFP}
\left\{
\begin{aligned}
- \partial_t g &=  v \cdot \nabla_x g + \CCC^* g    \quad&\text{in} \quad (0,T) \times \OO ,\\
\gamma_+ g &=  \RRR^* \gamma_- g   \quad &\text{on} \quad (0,T) \times \Sigma_+ , \\ 
 g(T) &= g_T \quad &\text{in} \quad \OO,  
\end{aligned}
\right.
\end{equation}
for any $T \in (0,\infty)$ and any final datum $g_T$.
The adjoint Fokker-Planck operator $\CCC^*$ is defined in \eqref{def:CCC*}, and the adjoint reflection operator $\RRR^*$ is defined by 
$$
\RRR^* g (x,v) = \iota_S  \SSS g (x,v) + \iota_D \DDD^* g (x), $$
with
$$ 
 \DDD^* g (x) = \widehatt{\MMM_x g}(x) := \int_{\R^d} g(x,w) \MMM_x(w) (n_x \cdot w)_{-} \, dw, 
$$ 
for any function $g$ with support on $\Sigma_-$.

For two solutions $f$ to the forward Cauchy problem \eqref{eq:KFP}--\eqref{eq:KolmoBdyCond}--\eqref{eq:initialDatum} and $g$ to the  dual problem  \eqref{eq:dualKFP},   the usual identity
\begin{equation}\label{eq:identite-dualite}
\int_\OO f(T) g_T   
= \int_\OO f_0 g(0)
\end{equation}
then holds at least formally, see also Theorem~\ref{theo:WP-KFP&KFP*} below. We may indeed formally compute
\bean
\int_\OO f(T) g_T  
&=& \int_\OO f_0 g(0) + \int_0^T \int_\OO (\partial_t f g + f \partial_t g  )  \, ds \\
&=& \int_\OO f_0 g(0) - \int_0^T \!\!\! \int_\OO (v \cdot \nabla_x f  g + f v \cdot \nabla_x g ) \, ds
\\
&=& \int_\OO f_0 g(0) - \int_0^T \!\!\! \int_\Sigma (v \cdot n) \gamma f \gamma g \, ds
\\
&=& \int_\OO f_0 g(0) - \int_0^T \!\!\! \int_{\Sigma_+} (v \cdot n) (\gamma_+ f)  (\RRR^* \gamma_- g) \, ds
\\
&& \quad +  \int_0^T \!\!\! \int_{\Sigma_-} |v \cdot n| (\RRR \gamma_+ f) ( \gamma_- g)   \, ds,
\eean 
by using the Green-Ostrogradski formula and the reflection conditions at the boundary. From the very definition of $\RRR$ and $\RRR^*$, we then deduce \eqref{eq:identite-dualite}.

\begin{lem}\label{lem:EstimKFP*Lq12}
For any admissible weight function $\omega$ and any exponent $q=1$ or $q=2$, there exist $\kappa \in \R$ and $C \ge 1$ such that for any $T > 0$ and any $g_T \in L^q_m$ with $m:=\omega^{-1}$, the associated solution $g$ to the backwards dual problem \eqref{eq:dualKFP} satisfies 
\begin{equation}\label{eq:g0&g-<gtC}
\| g(0) \|_{L^q_m}  
\le C e^{\kappa T} \| g_T \|_{L^q_m}. 
\end{equation}
\end{lem}

\begin{proof}[Proof of Lemma~\ref{lem:EstimKFP*Lq12}]  
 Without loss of generality we may suppose that $m \ge \MMM_x$, otherwise we replace $m$ by $c m$ where $c>0$ is such that $ m \ge c^{-1}  \MMM_x$.
 
Consider a final time $T \in (0,\infty)$, a final datum $0 \le g_T \in L^q_m$ and $g = g(t,x,v) \ge 0$ a solution  to the backward dual Cauchy problem \eqref{eq:dualKFP}. 
For $A\geq 1$, we introduce the weight functions 
\be\label{def:mA}
m_A^q := \chi_A \MMM_x + (1-\chi_A)  m^q,
\quad
\widetilde m^q   :=  \left( 1 -  \frac12 \frac{n_x \cdot  v}{\la v \ra^4}  \right) m_A^q,  
\ee
with the notations introduced in \eqref{eq:varpiC-varpisharp}. 
It is worth emphasizing that 
\beqn\label{eq:m&mtilde} 
\MMM_x \le m_A \le m \quad\hbox{and}\quad
c_A^{-1} m \le \frac{1}{ 2}m_A \le \widetilde m  \le \frac{3}{ 2} m_A \le \frac{3}{ 2}  m,
\eeqn
with $c_A \in (0,\infty)$. Similarly as in the proof of Lemma \ref{lem:EstimKFPLp12}, we compute
\be\label{eq:dualGrowthL12}
- \frac{1}{ q}\frac{d}{dt} \int_\OO g^q \, \widetilde m = \int_{\OO}  g^{q-1}  \, (\CCC^* g)\, \widetilde m^q   +  \frac{1}{ q}\int_\Sigma (\gamma g)^q  \, \widetilde m^q \, (n_x \cdot v) - \frac{1}{ q}\int_\OO g^q\, \left( v\cdot \grad_x \widetilde m^q\right) ,
\ee
and we estimate each term separately.

\medskip\noindent
\emph{Step 1.} 
In order to estimate the boundary term in \eqref{eq:dualGrowthL12}, we split it into
\bean
\int_\Sigma (\gamma g)^q  \, \widetilde m^q \, (n_x \cdot v) &=& B_1 - B_2
\eean
with
$$
B_1= \int_\Sigma (\gamma g)^q  \, m_A^q \, (n_x \cdot v) 
\quad\text{and}\quad
B_2= \frac{1}{ 2}\int_\Sigma (\gamma g)^q  \, m_A^q \, (n_x \cdot \hat v)^2 .
$$

For the first term, we have
\bean
B_1 
&=&\int_{\Sigma_+}\left(\iota_S \SSS\gamma_- g + \iota_D \DDD^*\gamma_- g\right)^q   \, m_A^q \, (n_x \cdot v)_+ -  \int_{\Sigma_-}  (\gamma_- g)^q \, m_A^q \, (n_x \cdot v)_-\\
&\leq& 
\int_{\Sigma_+}\iota_S \left(\SSS\gamma_- g\right)^q  \, m_A^q \, (n_x \cdot v)_+ + \int_{\Sigma_+}\iota_D \left(\DDD^*\gamma_- g\right)^q   \, m_A^q \, (n_x \cdot v)_+ 
-  \int_{\Sigma_-}  (\gamma_- g)^q \, m_A^q \, (n_x \cdot v)_-\\
&\leq&  -\int_{\Sigma_-} \iota_D (\gamma_- g)^q  \, m_A^q \, (n_x \cdot v)_- + \int_{\Sigma_+} \iota_D \left(\widehatt{\MMM_x\gamma_- g}\right)^q  \, m_A^q \, (n_x \cdot v)_+ ,
\eean
where we have used the boundary condition in \eqref{eq:dualKFP} on the first line, the convexity of the mapping $s\mapsto s^q$ on the second line, and the change of variables $v\to \VV_x \,v$ on the third one. Defining
$$
K_1(m_A) := \int_{\R^d} m_A^q (n_x \cdot v)_+ \, dv < \infty, 
$$
we equivalently have 
$$
\begin{aligned}
B_1 \le -\int_{\Sigma_-} \iota_D (\gamma_- g)^q  \, m_A^q \, (n_x \cdot v)_- + \int_{\partial\Omega} \iota_D K_1(m_A) (\widehatt{\MMM_x \gamma_- g})^q.
\end{aligned}
$$
When $q=1$, we set $K_2(m_A):=1$ and use the fact $m_A \ge \MMM_x$ in order to obtain
$$
\begin{aligned}
B_1 \le \int_{\partial\Omega} \iota_D \left\{ K_1(m_A) -1 \right\} \widehatt{\MMM_x \gamma_- g} .
\end{aligned}
$$
On the other hand, when  $q=2$, the Cauchy-Schwarz inequality yields
$$
\begin{aligned}
(\widehatt{\MMM_x \gamma_- g})^2 \le K_2(m_A)^{-1} \int_{\R^d} (\gamma_- g)^2 m_A^2 (n_x \cdot v)_- \, dv, 
\end{aligned}
$$
where we have set 
$$
K_2(m_A)^{-1} := \int_{\R^d} \MMM_x^2 m_A^{-2} (n_x \cdot v)_- \, dv < \infty,
$$ 
and we thus obtain 
$$
\begin{aligned}
B_1 \le \int_{\partial\Omega} \iota_D \left\{ K_1(m_A) -K_2(m_A) \right\} (\widehatt{\MMM_x \gamma_- g})^2 .
\end{aligned}
$$

We now deal with the second term $B_2$, observing first that
$$
\begin{aligned}
B_2 
&= \frac{1}{ 2}\int_{\Sigma_-} (\gamma_- g)^q  \, m_A^q \, (n_x \cdot \hat v)^2
+ \frac{1}{ 2}\int_{\Sigma_-} \left\{ \iota_S \gamma_- g + \iota_D \DDD^* \gamma_- g \right\}^q  \, m_A^q \, (n_x \cdot \hat v)^2 ,
\end{aligned}
$$
where we have used the boundary condition in \eqref{eq:dualKFP} and the change of variables $v\to \VV_x v$. 
If $q=1$, we through away the two first terms and we have 
$$
\begin{aligned}
B_2 
&\ge \frac{1}{ 2}\int_{\Sigma_-}  \iota_D (\widehatt{\MMM_x \gamma_- g})   \, m_A \, (n_x \cdot \hat v)^2 
\ge \frac{1}{ 2}\int_{\partial\Omega}  \iota_D K_0(m_A) (\widehatt{\MMM_x \gamma_- g})   ,
\end{aligned}
$$
where we have set
$$
K_0 (m_A) := \int_{\R^d} m_A (n_x \cdot \hat v)_+^2 \, dv < \infty .
$$
Otherwise if $q=2$, we through away the last integral and using the Cauchy-Schwarz inequality and $\iota_D \le 1$, we obtain
$$
\begin{aligned}
B_2 
&\ge \frac{1}{ 2}\int_{\Sigma_-} (\gamma_- g)^q  \, m_A^q \, (n_x \cdot \hat v)^2 
\ge \frac{1}{ 2} \int_{\partial\Omega}  \iota_D K_0(m_A) (\widehatt{\MMM_x \gamma_- g})^2 , 
\end{aligned}
$$
where we have set
$$
K_0(m_A)^{-1} := \int_{\Sigma_-^x} \MMM_x^2 m_A^{-2} \la v \ra^2 \, dv < \infty.
$$
In both cases, gathering previous estimates yields
\begin{equation}
\begin{aligned}
\int_\Sigma (\gamma g)^q  \, \widetilde m^q \, (n_x \cdot v)
& \le \int_{\partial\Omega} \iota_D \left\{ K_1(m_A) - K_2(m_A) - \frac12 K_0(m_A) \right\} (\widehatt{\MMM_x \gamma_- g})^q.
\end{aligned}
\end{equation}
We observe that in both cases $q=1,2$, we have
\bean
\lim_{A \to \infty} K_1(m_A) = \lim_{A \to \infty}  K_2(m_A) = 1, 
\eean
thanks to the dominated convergence theorem, the normalization condition on $\MMM_x$  and the condition \eqref{eq:cond-Lpomega}.
We similarly have 
$$
\lim_{A \to \infty} K_0 (m_A)  = \int_{\R^d}\MMM_x  (n_x \cdot \hat v)_+^2 \, dv \ge C_1(\Theta_*,\Theta^*) > 0,
$$
when $q=1$, and 
$$
\lim_{A \to \infty} K_0 (m_A)^{-1}  =  \int_{\Sigma_-^x} \MMM_x \la v \ra^2 \, dv \le C_2(\Theta_*,\Theta^*) < \infty,
$$
when $q=2$. All these convergences being uniform in $x \in \partial\Omega$,  we can choose $A > 0$ large enough in such a way that 
$$
K_1(m_A)-K_2(m_A)-\frac12 K_0(m_A) \le 0,
$$
which implies
\begin{equation}
\int_\Sigma (\gamma g)^q  \, \widetilde m^q \, (n_x \cdot v) \le 0 .
\end{equation}

\smallskip\noindent
\textit{Step 2.} 
We now estimate the first term at the right-hand side of \eqref{eq:dualGrowthL12}. 
First of all, from \eqref{eq:identCCCffp}-\eqref{def:varpi}, we have
$$
\int_{\R^d} (\CCC^* g)  g^{q-1} \widetilde m^q   =
 -\frac{4(q-1)}{q}  \int_{\R^d} |\nabla_v (g^{q/2} \widetilde m^{q/2} ) |^2  + \int g^q \widetilde m^q  \varpi^{\CCC^*}_{\tilde m, q},
$$
with 
\bean
 \varpi^{\CCC^*}_{\varphi,q}
=  2\left(1-\frac{1}{q}\right) \frac{|\nabla_v \varphi|^2 }{ \varphi^2}  
+ \left(\frac 2q -1\right)\frac{ \Delta_v  \varphi }{\varphi}  
+ b  \cdot \frac{\nabla_v \varphi}{\varphi}
+  c  + \left(\frac{1}{q} -1\right) \Div_v b.
\eean
Arguing exactly as in Step~2 of the proof of Lemma~\ref{lem:EstimKFPLp12}, we can write
$$
 \varpi^{\CCC^*}_{\tilde m, q} = \varpi^{\CCC^*}_{ m, q} + \mathfrak W, 
 $$
 with 
$$
\varpi^{\CCC^*}_{m, q}  = \varpi^{\CCC}_{\omega,p}, \quad \frac{1 }{ p} + \frac{1 }{ q} =1, \quad  \mathfrak W = o (|\varpi^{\sharp}_{\omega,p}|).
$$

\medskip\noindent
\textit{Step 3.}  
We finally deal with the second term at the right-hand side of \eqref{eq:dualGrowthL12}. We compute
$$
v \cdot \nabla_x (\widetilde m^q)
= \left( 1 - \frac12 \, \frac{n_x \cdot v}{\la v \ra^4} \right) v \cdot \nabla_x (m_A^q)
-\frac12 ( \hat v \cdot D_x n_x \hat v) \frac{m_A^q}{\la v \ra^2}
$$
and observe that
$$
\nabla_x (m_A^q) 
=  \chi_A  \MMM_x \left[ \frac{|v|^2}{2} \frac{\nabla_x \Theta_x}{\Theta_x^2} - \frac{(d-1)}{2} \frac{\nabla_x \Theta_x}{\Theta_x} \right].
$$
Hence assumption \eqref{eq:Assum-Theta} together with the fact that $\chi_A$ is compactly supported and the regularity assumption on $\Omega$ as above imply
$$
\frac{1}{q} v \cdot \nabla_x (\widetilde m^q)   \lesssim \frac{1}{\langle v \rangle^2} \widetilde m^q \lesssim \frac{|\varpi^{\sharp}|_{\omega,p} }{ \langle v \rangle^{s+\gamma} }  \widetilde m^q.
$$

\medskip\noindent
\textit{Step 4.} 
Coming back to \eqref{eq:dualGrowthL12} and gathering previous estimates, we deduce that 
\beqn\label{eq:disssipSL*-Lq}
-\frac1q\frac{d}{dt} \int_\OO g^q \widetilde m^q \le  -\frac{4(q-1)}{q}    \int_\OO |\nabla_v (g\widetilde m)^{q/2}|^2 + \int_\OO g^q \widetilde m^q \varpi^{\LLL^*}_{\tilde m,q}
\eeqn
for both $q=1,2$ and with 
\beqn\label{eq:disssipSL*-LqBIS}
 \varpi^{\LLL^*}_{\tilde m,q} := \varpi^{\CCC^*}_{\tilde m,q}   +  \frac{1 }{ \widetilde m^q} v \cdot \nabla_x \widetilde m^q.
\eeqn
Arguing as in the end of the proof of Lemma~\ref{lem:EstimKFPLp12}, we obtain that for any $\vartheta \in (0,1)$, there exists $\kappa,R > 0$ such that 
\beqn\label{eq:varpiL*-varpisharp}
  \varpi^{\LLL^*}_{\tilde m,q} \le \kappa \chi_{R} + \chi_{R}^c \vartheta \varpi^\sharp_{\omega,p},
\eeqn
in particular, $  \varpi^{\LLL^*}_{\tilde m,q}  \le \kappa$, and we  
immediately conclude thanks to   Gr\"onwall's lemma and the fact that $m \lesssim \widetilde m \lesssim m$.  
\end{proof}

\section{Well-posedness in a weighted $L^2$ framework} \label{sec-Wellpodeness}

We briefly discuss the  well-posedness in a weighted $L^2$  framework for both the primal and the dual Cauchy problems,
using  some material developed in \cite{sanchez:hal-04093201,CM-Landau**}.

\subsection{Trace results in a $L^2$ framework}\label{subsec:trace}

In this section, we consider the kinetic Fokker-Planck  equation
\beqn\label{eq:KolmogorovTrace}
\partial_t g + v \cdot \nabla_x g =  \mathbf{L} g +  G, \quad   \mathbf{L} g :=  \Delta_v g  + b_i \partial_{v_i} g + \eta g   
\eeqn
for a    vector field $b=b(x,{v})$, a function $\eta = \eta(x,v)$ and a source term $G = G(t,x,{v})$.  
We formulate some trace results for solutions to  the Vlasov-Fokker-Planck equation developed in  \cite[Sec.~11]{sanchez:hal-04093201}
and \cite[Sec.~2]{CM-Landau**} (see also  \cite[Section~4.1]{MR2721875}) and  and which are mainly a consequence of
the two following facts:

 \smallskip
 (1)  If $g \in L^2_{tx}H^1_{v}$ is a weak solution to \eqref{eq:KolmogorovTrace},  then it is a renormalized solution; 
 
 \smallskip
 (2) If $g \in L^\infty_{txv}$ and $\nabla_v g \in L^2_{txv}$ is a weak solution to \eqref{eq:KolmogorovTrace}, then it admits a trace $\gamma g \in L^\infty$  in a renormalized sense.
 
\medskip

We introduce some notations. We denote 
\beqn\label{eq:defdxi1&2}
 d\xi_1 :=  |n_x \cdot v| \, dv \, d\sigma_{\!x} \ \hbox{ and } \
  d\xi_2 := (n_x \cdot   \hat v)^2 \, dv \, d\sigma_{\!x} 
\eeqn
the measures on the boundary set $\Sigma$. We denote 
by $\BBB_1$ the class of renormalized functions $\beta \in
\Wloc^{2,\infty}(\R)$ such that $\beta''$ has a compact support,  by $\BBB_2$ the class of functions $\beta \in
\Wloc^{2,\infty}(\R)$ such that $\beta'' \in L^\infty(\R)$ and by $ \DD_0(\bar \UU)$ the space of test functions $\varphi \in \DD(\bar
\UU)$ such that $\varphi = 0$ on $\Gamma_0$. 
We finally define the operators 
\bean
{\bf M}^*_0 \varphi := -\partial_t \varphi - v \cdot \nabla_x \varphi, 
\quad
{\bf M}^*_i \varphi :=  \partial_{v_i} \varphi -  b_i  \varphi, 
\eean
and we assume
\beqn\label{eq:Kolmogorov-hyp}
b_i,\eta    \in L^\infty_{tx}L^\infty_{{\rm loc},v}.
\eeqn

\begin{theo} 
\label{theo:Kolmogorov-trace}
We consider $g \in  L^2((0,T) \times \Omega;\Hloc^1(\R^d))  $, $G \in L^{2}_{tx}  (H^{-1}_{{\rm loc},v})$, $b_i,\eta$  satisfying \eqref{eq:Kolmogorov-hyp} and we assume that $g$ is a solution to the  kinetic Fokker-Planck equation \eqref{eq:KolmogorovTrace} in the sense of distribution $\DD'(\UU)$.  

(1)  There exists $\gamma g \in \Lloc^2(\Gamma,d\xi_2dt)$, $ g  \in C([0,T];\Lloc^2(\OO))
$
and the following Green renormalized formula
\bear\label{eq:FPK-traceL2}
&& \int_{\UU} \bigl( \beta(g) \, {\bf M}^*_0 \varphi  + \partial_{v_i}  \beta(g)   {\bf M}^*_i \varphi + \beta''(g) \, |\nabla_v g|^2     \varphi) \, d{v} dx dt
 \\ \nonumber
&&\quad
+  \int_{\Gamma} \beta(\gamma \, g) \, \varphi \,\,   n_x \cdot {v} \,  d{v} d\sigma_{\! x}dt 
+ \Bigl[ \int_\OO (\beta(g) \varphi)(t,\cdot) dxdv \Bigr]_0^T =  \langle G + \eta g , \beta'(g)   \varphi \rangle
\eear
holds for any renormalized function $\beta \in \BBB_1$ and any test functions  $\varphi \in \DD(\bar \UU)$, 
as well as  for any  renormalized function  $\beta \in \BBB_2$ and any test functions
$\varphi \in \DD_0(\bar \UU)$.   It is worth emphasizing that  $\beta'(g) \varphi \in L^2_{tx}H^1_v$ so that the duality product 
 $\langle G  , \beta'(g)   \varphi \rangle$ is well defined. 

(2)  If furthermore $g_0 \in \Lloc^2(\bar\OO)$ and $\gamma_- g  \in \Lloc^2( \Gamma; d\xi_1 dt)$, then $g  \in C([0,T];\Lloc^2(\bar\OO))$,  $\gamma_+ g  \in \Lloc^2( \Gamma; d\xi_1dt)$ and \eqref{eq:FPK-traceL2}
holds  for any renormalized function $\beta \in \BBB_2$ and any test functions
$\varphi \in \DD(\bar \OO)$.

(3) Alternatively to point (2), if furthermore $g \in \Lloc^\infty(\bar\UU)$, then $\gamma g \in \Lloc^\infty(\Gamma)$ and \eqref{eq:FPK-traceL2}
holds  for any renormalized function $\beta \in \BBB_2$ and any test functions
$\varphi \in \DD(\bar \OO)$.  

 \end{theo}

 \smallskip
  We will also use the following stability result in the spirit of  \cite[Theorem 5.2]{MR2721875}. 

 \begin{theo}\label{theo:Kolmogorov-stability}   
Let us consider four  sequences $(g^k)$,  $(b^k)$, $(\eta^k)$ and  $(G^k)$  and four functions  $g$, $b$, $\eta$, $G$ which all  satisfy the requirements of Theorem~\ref{theo:Kolmogorov-trace}. 

(1) If    $g^k  \wto g$  weakly in $\Lloc^2(\bar\UU)$, $b^k \to b$  strongly  in $\Lloc^2(\bar\OO)$, $\eta^k \to \eta$  strongly  in $\Lloc^2(\bar\OO)$ and  $G^k \wto G$ weakly in $L^{2}_{{\rm loc},x}H^{-1}_{{\rm loc},v}$, then  $g$ satisfies \eqref{eq:KolmogorovTrace} so that it admits a trace function $ \gamma g \in \Lloc^2 (\Gamma; d\xi_2dt)$ and $\gamma g^k \wto  \gamma g$ weakly in $\Lloc^2 (\Gamma; d\xi_2 dt)$.

(2) If  $g^k  \to g$  strongly in $\Lloc^2(\bar\UU)$,  $b^k \wto b$ weakly  in $\Lloc^2(\bar\OO)$, $\eta^k \wto \eta$  weakly  in $\Lloc^2(\bar\OO)$ and  $G^k \to G$ weakly in $L^{2}_{{\rm loc},x}H^{-1}_{{\rm loc},v}$, 
then  $g$ satisfies \eqref{eq:KolmogorovTrace} so that it admits a trace function $ \gamma g \in \Lloc^2 (\Gamma; d\xi_2dt)$ and $\gamma g^k \wto  \gamma g$ weakly in $\Lloc^2 (\Gamma; d\xi_2dt)$.

\end{theo}

\subsection{Well-posedness for the primal equation}
\label{subsec:Well-posednessKolmogrov}

For further reference, for an admissible weight function $\omega$, we define 
the Hilbert norm $\| \cdot \|_\HHH = \| \cdot \|_{\HHH_\omega}$   by 
$$
\| f \|^2_\HHH := \| f \|^2_{L^2_\omega} +  \| f \|^2_{H^{1,\dagger}_\omega}, \quad   \| f \|^2_{H^{1,\dagger}_\omega}  :=  \int_\UU  \bigl\{   |\nabla_v f |^2 +   \langle \varpi^\sharp_{\omega,2} \rangle f^2 \bigr\} \omega^2 \, d v \, d x \, d t, 
$$
with $ \varpi^\sharp_{\omega,2}$ defined in \eqref{eq:varpi-asymptotic},
and we denote by $\HHH =\HHH_\omega$ the associated  Hilbert space. 
We now state the well-posedness result for the primal problem which is nothing but  \cite[Theorem~2.12]{CM-Landau**}.

\begin{prop}\label{prop-KFP-L2primal}
We make the  regularity assumptions on $\Omega$, $\Theta$, $b$ and $c$ as presented in  Section~\ref{sec:intro},  in particular \eqref{eq:Assum-bc0}, \eqref{eq:Assum-Theta}, \eqref{eq:Assum-bc1} and \eqref{eq:Assum-bc2} hold. 
For any admissible weight function $\omega$ and  any $f_0 \in L^2_\omega(\OO)$, there exists a  unique global solution $f \in  X_T := L^\infty(0,T;L^2_\omega)\cap C([0,T];\Lloc^2) \cap  \HHH $, $\forall \, T > 0$,  to the kinetic Fokker-Planck equation \eqref{eq:KFP} complemented with the Maxwell reflection boundary condition \eqref{eq:KolmoBdyCond} and associated to the initial datum $f_0$. 
 More precisely, the function $f$ satisfies  equation \eqref{eq:KolmogorovTrace} in the sense of distributions in $\DD'(\UU)$ with trace functions, defined thanks to Theorem~\ref{theo:Kolmogorov-trace}, satisfying $\gamma f \in L^2_\omega(\Gamma,d\xi_2)$ as well as  the Maxwell reflection boundary condition \eqref{eq:KolmoBdyCond} pointwisely and $f(t,\cdot) \in L^2_\omega$, $\forall \, t \in [0,T]$, and   the  initial condition $f(0,\cdot) = f_0$ pointwisely. 
\end{prop}

Because we will need to adapt it in the next section, we  allude the proof and we refer to   \cite{CM-Landau**} for details. 

\begin{proof}[Proof of Proposition~\ref{prop-KFP-L2primal}] 
We split the proof into four steps.

\smallskip\noindent
\textit{Step 1.} Given $\mathfrak f \in L^2_\omega(\Gamma_-; d\xi_1)$, we solve the inflow problem
\begin{equation}\label{eq:linear_primal_inflow}
\left\{
\begin{aligned}
& \partial_t f + v \cdot \nabla_x f =\CCC f  \quad &\text{in}&\quad (0,\infty) \times \OO \\
& \gamma_{-} f  = \mathfrak f \quad &\text{on}&\quad (0,\infty) \times \Sigma_{-} \\
& f_{| t=0} = f_0  \quad &\text{in}&\quad   \OO,
\end{aligned}
\right.
\end{equation}
thanks to Lions' variant of the Lax-Milgram theorem \cite[Chap~III,  \textsection 1]{MR0153974}. More precisely, we define  $\widetilde\omega$ as during the proof of  Lemma~\ref{lem:EstimKFPLp12} 
and  the bilinear form
$\EEE : \HHH \times C_c^1([0,T) \times \OO \cup \Gamma_-) \to \R$ by 
\bean
\EEE(f,\varphi) 
&:=& \int_\UU (\lambda f - \CCC f) \varphi \widetilde\omega^2   -   \int_\UU f ( \partial_t   + v \cdot \nabla_x   ) (\varphi\widetilde\omega^2) . 
\eean
Using the Green-Ostrogradski formula, we observe that 
\bean
\EEE(\varphi,\varphi) 
&=&  \int_\UU (\lambda \varphi - \CCC \varphi) \varphi \widetilde\omega^2  +  \frac12 \int_\OO \varphi(0,\cdot)^2 \widetilde\omega^2  
\\
&& - \frac12 \int_\UU \varphi^2 v \cdot \nabla_x \widetilde\omega^2   + \frac12 \int_{\Gamma_-} ( \gamma_- \varphi)^2 \widetilde\omega^2 \, d\xi_1 , 
\eean
 for any $\varphi \in C_c^1([0,T) \times \OO \cup \Gamma_-)$. The same computations as presented during the proof of  Lemma~\ref{lem:EstimKFPLp12} imply that 
 \beqn\label{eq:EEE-coercive}
\EEE(\varphi,\varphi) 
\ge (\lambda-\lambda_0) \| \varphi \|^2_{L^2_{\tilde\omega}} +  \| \varphi \|^2_{H^{1,\dagger}_{\tilde\omega}} 
+ \tfrac12 \| \varphi(0) \|^2_{L^2_{\tilde\omega}} + \tfrac12 \| \gamma_- \varphi \|^2_{L^2_{\tilde\omega}(\Gamma_-; d\xi_1)}, 
\eeqn
for some $\lambda_0 \in \R$. 
%
%
For $\lambda>\lambda_0$, the  bilinear form $\EEE$ is thus coercive and 
the above mentioned Lions' theorem implies the existence of a function $f_\lambda \in \HHH$
which satisfies the variational equation 
\beqn\label{eq:EEE-cequation}
\EEE(f_\lambda,\varphi) = \int_{\Gamma_-} \mathfrak{f} e^{-\lambda t} \varphi \widetilde\omega^2 \, d\xi_1  + \int_\OO f_0 \varphi(0,\cdot) \widetilde\omega^2  \, d v \, d x, \quad \forall \, \varphi \in C_c^1([0,T) \times \OO \cup \Gamma_-). 
\eeqn
Defining $f := f_\lambda e^{\lambda t}$ and using Theorem~\ref{theo:Kolmogorov-trace}, we deduce that $f \in \HHH \cap C([0,T];L^2_\omega(\OO))$ is a renormalized solution to the inflow problem \eqref{eq:linear_primal_inflow} and that $\gamma f \in L^2_{\tilde\omega}(\Gamma;   d\xi_1)$. From the renormalization formulation, we have the uniqueness of such a solution. 
Because of the trace Theorem~\ref{theo:Kolmogorov-trace}-(2), we can take $\beta(s) = s^2$ in \eqref{eq:FPK-traceL2} and we get 
\bean
&& \int_{\UU} \bigl(  f^2 (- v \cdot \nabla_x \varphi^2 - 2 \varphi^2 \varpi^\CCC_{\varphi,2}) + 2  |\nabla_v (f \varphi)|^2 ) \, d{v} dx dt
\\&&\quad
+  \int_{\Gamma} (\gamma \, f)^2 \, \varphi^2 \,\,   n_x \cdot {v} \,  d{v} d\sigma_{\! x}dt 
+ \Bigl[ \int_\OO (f^2 \varphi^2 )(t,\cdot) dxdv \Bigr]_0^T = 0, 
\eean
for any $\varphi \in \DD(\bar\OO)$. Taking $\varphi := \widetilde \omega$ in that last identity thanks to an approximation procedure and next using \eqref{eq:EEE-cequation}, the same computations as presented during the proof of  Lemma~\ref{lem:EstimKFPLp12}  and  the Gronwall lemma, 
we also deduce the energy estimate 
\bean
&&\| f_t \|^2_{L^2_{\tilde\omega}} + \int_0^t \left( \| \gamma_+ f_s \|^2_{L^2_{\tilde\omega}(\Gamma_+; d\xi_1)} +  2  \| f \|^2_{H^{1,\dagger}_{\tilde\omega}} \right) \, e^{\lambda_0(t-s)} \, d s 
\\
&&\qquad\le
 \| f_0 \|^2_{L^2_{\tilde\omega}} e^{\lambda_0 t} + \int_0^t  \| \mathfrak{f}_s \|^2_{L^2_{\tilde\omega}(\Gamma_-; d\xi_1)}   \, e^{\lambda_0(t-s)} \, d s. 
\eean

\medskip\noindent
\textit{Step 2.} For any $\alpha \in (0,1)$ and $ h \in \HHH \cap C([0,T];L^2_\omega(\OO))$ solution to the problem \eqref{eq:linear_primal_inflow} for some $\mathfrak h \in L^2_{\tilde\omega}(\Gamma_-; d\xi_1)$, and thus  $\gamma h \in L^2_{\omega}(\Gamma; d\xi_1)$,  we then consider the modified Maxwell reflection boundary condition problem
 \begin{equation}\label{eq:linear_gGk}
\left\{
\begin{aligned}
& \partial_t f + v \cdot \nabla_x f = \CCC f \quad &\text{in}&\quad (0,T) \times \OO \\
& \gamma_{-} f   =  \alpha \RRR \gamma_{+}  h  \quad &\text{on}&\quad (0,T) \times \Sigma_{-} \\
& f(t=0, \cdot) = f_0  \quad &\text{in}&\quad   \OO, 
\end{aligned}
\right.
\end{equation}
for which a solution $f \in \HHH \cap C([0,T];L^2_\omega(\OO))$ such that $\gamma f \in L^2_{\omega}(\Gamma; d\xi_1)$ is given by the first step. 
Repeating the arguments of Step~1 in the proof of  Lemma~\ref{lem:EstimKFPLp12},  we have 
\bean
\| \RRR \gamma_+ h \|^2_{L^2_{\tilde\omega}(\Sigma_-; d\xi_1)} 
&\le&   \int_{\Sigma_{+}} \iota_S  (\gamma_+ h)^2 \omega^2_A  d\xi_1 
+   \int_{\partial\Omega}  \iota_D  (K_1(\omega_A) - K_0(\omega_A)) (\widetilde{\gamma_+ h})^2  
\\
&\le&   \int_{\Sigma_{+}} (1- \iota_D)  (\gamma_+ h)^2 \omega^2_A  d\xi_1 
+   \int_{\partial\Omega}  \iota_D K_2(\omega_A) (\widetilde{\gamma_+ h})^2  
\\
&\le&  \| \gamma_+ h \|^2_{L^2_{\tilde\omega}(\Gamma_+; d\xi_1)}.
\eean
Thanks to the energy estimate stated in the first step, we immediately deduce that 
the mapping $h \mapsto f$ is $\alpha$-Lipschitz for the norm defined by 
$$
\sup_{t \in [0,T]} \left\{ \| f_t \|^2_{L^2_{\tilde\omega}}  e^{- \lambda_0 t }+ \int_0^t  \| \gamma f_s \|^2_{L^2_{\tilde\omega}(\Gamma_+; d\xi_1)}  \, e^{- \lambda_0s} \, \d s \right\}.
$$
From the Banach fixed point theorem, we deduce the existence of a unique fixed point, and that provides a solution to \eqref{eq:linear_gGk}.

\medskip\noindent
\textit{Step 3.} For a sequence $\alpha_k \in (0,1)$, $\alpha_k \nearrow 1$, we next consider the sequence $(f_k)$ obtained in Step~2 as the solution to the modified Maxwell reflection boundary condition problem
\begin{equation}\label{eq:linear_gak}
\left\{
\begin{aligned}
& \partial_t f_k + v \cdot \nabla_x f_k = \CCC f_k \quad &\text{in}&\quad (0,T) \times \OO \\
& \gamma_{-} f_k   =  \alpha_k \RRR \gamma_{+} f_k  \quad &\text{on}&\quad (0,T) \times \Sigma_{-} \\
& f_k(t=0, \cdot) = f_0  \quad &\text{in}&\quad   \OO.
\end{aligned}
\right.
\end{equation}
From the fact that $\RRR: L^2_{\tilde\omega}(\Sigma_+;d\xi_1) \to   L^2_{\tilde\omega}(\Sigma_+;d\xi_1)$ with norm less than $1$ as established in Step~2 and the energy estimate stated at the end of Step~1, $f_k$ satisfies 
\bear\label{eq:linear_gak_bdd}
&&\quad \| f_{kt} \|^2_{L^2_{\tilde\omega}} + \int_0^t \left\{ (1-\alpha^2_k) \| \gamma f_{ks} \|^2_{L^2_{\tilde\omega}(\Gamma_+; \d\xi_1)}
+  2
 \| f_{ks} \|^2_{H^{1,\dagger}_{\tilde\omega}} 
\right\} \, e^{\lambda_0(t-s)} \, \d s 
\le
 \| f_0 \|^2_{L^2_{\tilde\omega}} e^{\lambda_0 t}  , 
\eear
for any $t \in (0,T)$ and any $k \ge 1$. Choosing $\beta(s) := s^2$ and $\varphi := (n_x \cdot v) \la v \ra^{-2} \omega^2(v)$ in the Green formula \eqref{eq:FPK-traceL2}, we additionally have 
\beqn\label{eq:boundL2dxi2}
\int_{\Gamma} (\gamma f_k)^2 \omega^2 \, \d\xi_2  \, \d t 
\lesssim   \| f_0 \|^2_{L^2_\omega} e^{\lambda_0 T}. 
\eeqn
From the above estimates, we deduce that, up to the extraction of a subsequence, there exist $f \in \HHH \cap L^\infty(0,T;L^2_\omega(\OO))$ and $\mathfrak{f}_\pm \in L^2_{\omega}(\Gamma_{\pm}; \d\xi_2 \d t)$ such that 
$$
f_k \wto f \hbox{ weakly in } \ \HHH \cap L^\infty(0,T;L^2_\omega(\OO)), 
\quad
\gamma_\pm f_k \wto \mathfrak{f}_\pm \hbox{ weakly in } \ L^2_{\omega}(\Gamma; \d\xi_2 \d t). 
$$
Because $\langle v \rangle  \omega^{-1} \in L^2(\R^d)$, 
 we have $L^2_{\omega}(\Gamma; \d\xi_2) \subset L^1(\Gamma; \d\xi_1)$. 
 On the other hand, we recall that from the very definition \eqref{eq:KolmoBdyCond}, we have 
\beqn\label{eq:KolmogorovWP-hypL1}
\RRR : L^1(\Sigma_+;d\xi_1) \to L^1(\Sigma_-;d\xi_1), \quad   \| \RRR \|_{L^1(\Sigma;  d\xi_1 )} \le 1.
\eeqn
These three pieces of information together imply 
that $\RRR(\gamma f_{k+})  \wto \RRR(\mathfrak{f}_+)$   weakly in $L^1(\Gamma_-; \d\xi_1)$. 
On the other hand, from Theorem~\ref{theo:Kolmogorov-stability}, we have $\gamma f_k \wto \gamma f$ weakly in $\Lloc^2 (\Gamma; \d\xi_2)$. 
Using both convergences in the boundary condition $\gamma_- f_k = \RRR(\gamma_+ f_k)$, we obtain $\gamma_- f = \RRR(\gamma_+ f)$.
We may thus pass to the limit in equation \eqref{eq:linear_gak} and we obtain that $f \in  X_T$ 
is a renormalized solution to the KFP equation \eqref{eq:KFP} complemented with the Maxwell reflection boundary condition \eqref{eq:KolmoBdyCond} and associated to the initial datum $f_0$. Passing to the limit in \eqref{eq:linear_gak_bdd}, we also  have 
\beqn\label{eq:bddL2&HHH}
 \| f_{t} \|^2_{L^2_{\tilde\omega}} + 2 \int_0^t  \| f_{ks} \|^2_{H^{1,\dagger}_{\tilde\omega}}  \, e^{\lambda_0(t-s)} \, \d s 
\le
 \| f_0 \|^2_{L^2_{\tilde\omega}} e^{\lambda_0 t}  , 
\eeqn
for any $t \in (0,T)$.

\medskip\noindent
\textit{Step 4.} We consider now two solutions $f_1$ and $f_2 \in   X_T $ to the KFP equation   \eqref{eq:KFP}-\eqref{eq:KolmoBdyCond}  associated to the same initial datum $f_0$, so that the function $f := f_2 - f_1 \in  X_T$  is  a solution to the KFP equation   \eqref{eq:KFP}-\eqref{eq:KolmoBdyCond} associated to the initial datum $f(0) = 0$.
 We write  \eqref{eq:FPK-traceL2} with the choice $\varphi := \widetilde\omega_1\chi_R$, with $\widetilde\omega_1$ given by \eqref{def:omegaA} associated to $p=1$, $A > 0$ large enough, $\omega_1 := \langle v \rangle^{k_1}$ {defined in Remark~\ref{rem:lem:EstimKFPLp12}}, with $\chi_R(v) := \chi(v/R)$, $\mathbf{1}_{B_1} \le \chi \in \DD(\R^d)$,  and with the choice $\beta \in C^2(\R)$, $\beta(0) = 0$, $\beta''  $ with compact support. 
We get
\bean
&&  \int_\OO \beta(f_T)  \varphi \, \d v \, \d x  +  \int_{\Gamma} \beta(\gamma f) \,  \varphi  \,  (n_x \cdot v) \,  \d v \, \d\sigma_{\! x} \, \d t +  \int_{\UU}  \beta''(f) \, | \nabla_v f |^2  \varphi  \, \d v \, \d x \, \d t 
\\
 &&\qquad =  
 \int_{\UU}  \left\{ \beta(f) \, \TTT^*  \varphi 
+  \beta(f) \left(  \Delta_v \varphi   - \partial_{v_i} (b_i   \varphi ) \right)  +  cf   \beta'(f)  \varphi  \right\} \d v \, \d x \, \d t. 
\eean
We assume $0 \le \beta (s) \le |s|$, $|\beta'(s)| \le 1$, and $\beta'' \ge 0$ so that we may get rid of the last term at the left-hand side of the above identity. 
Recalling that $\Omega$ is bounded, we observe that 
\bean
|\beta(f) \left(  \Delta_v \varphi   - \partial_{v_i} (b_i   \varphi ) \right)  +  cf   \beta'(f)  \varphi |
&\le& |f| \bigl( | \Delta_v \varphi| + | \partial_{v_i} (b_i   \varphi )| + |c\varphi| \bigr)
\\
&\lesssim& |f| \omega_1 (1 + |v|^{\gamma-2}) \in  L^1(\UU),
\eean
because of the condition $ k_* > k_1 + (d+\gamma-2)/2$ and the bound $(1+\langle v \rangle^{\varsigma/2}) \, \omega f \in L^2(\UU)$. The same uniform estimate holds for the term
$\beta(f) \, \TTT^*  \varphi$.  We also observe that 
$$
| \beta(\gamma f) \,  \varphi  \,  (n_x \cdot v)|   \lesssim |\gamma f | \omega_1 |n_x \cdot v| \in  L^1(\Gamma),
$$
because of the condition $k_* > k_1 + d/2 + 1$ and the bound $\gamma f \omega \in L^2 (\Gamma; \d\xi_2 \d t)$ provided by \eqref{eq:boundL2dxi2}. 
We may thus pass to the   limit $R \to \infty$
and  $\beta(s) \nearrow |s|$ such that $0 \le s \beta'(s) \nearrow |s|$, and we deduce 
\bean
 \int_\OO |f_T| \, \widetilde \omega_1 \d v \, \d x + \int_{\Gamma} |\gamma f| \widetilde \omega_1  \, (n_x \cdot v) \,  \d v \, \d\sigma_{\! x} \, \d t
 \le 
\int_{\UU}    |f|  \varpi^\LLL_{\tilde\omega,1} \widetilde \omega_1 \d v \, \d x \, \d t .  
\eean
Using finally the estimate of Step~1 in the proof of  Lemma \ref{lem:EstimKFPLp12} in order to get rid of the boundary term as well as   the estimates obtained in Step~2 and Step~3
in  the  proof of  Lemma \ref{lem:EstimKFPLp12} in order to deal with the RHS term, we get 
$$
\int_\OO |f_T|  \widetilde \omega_1 \, \d v \, \d x 
\le \kappa \int_0^T \!\! \int_\OO   |f|  \widetilde \omega_1 \, \d v \, \d x \, \d t, 
$$
and   we conclude to $f = 0$ thanks to Gr\"onwall's lemma.
\end{proof}

\subsection{Well-posedness for the   dual equation and conclusions}
\label{subsec:Well-posednessKFP}

We first establish the well-posedness of the dual KFP equation in a $L^2$ framework.

\begin{prop}\label{prop-KFP-L2dual}
 Under the assumptions of  Proposition~\ref{prop-KFP-L2primal}, for any admissible weight function $\omega$, any final time $T > 0$ and any final datum $g_T \in L^2_m(\OO)$, $m := \omega^{-1}$, there exists a  unique   $g\in  Y_T :=  L^\infty(0,T;L^2_m) \cap C([0,T];\Lloc^2) \cap  \HHH_m $ solution to the backward dual kinetic Fokker-Planck equation \eqref{eq:dualKFP} in a similar sense as stated in Proposition~\ref{prop-KFP-L2primal}.  
\end{prop}

\begin{proof}[Sketch of the proof of Proposition~\ref{prop-KFP-L2dual}] 
We follow the same strategy as during the proof of Pro-position~\ref{prop-KFP-L2primal}.

\smallskip\noindent
\textit{Step 1.} Given $\mathfrak g \in L^2_m(\Gamma_-; d\xi_1)$, we consider the backward inflow problem
\begin{equation}\label{eq:linear_gGinflow}
\left\{
\begin{aligned}
& -\partial_t g -v \cdot \nabla_x g =\CCC^* g  \quad &\text{in}&\quad (0,T) \times \OO \\
& \gamma_{+} g  = \mathfrak g \quad &\text{on}&\quad (0,T) \times \Sigma_{+} \\
&g_{| t=T} = g_T  \quad &\text{in}&\quad   \OO,
\end{aligned}
\right.
\end{equation}
We define  $\widetilde m$ as during the proof of  Lemma~\ref{lem:EstimKFP*Lq12} 
and  the bilinear form
$\EEE : \HHH_m \times C_c^1((0,T] \times \OO \cup \Gamma_+) \to \R$,  by \bean
\EEE(g,\varphi) 
&:=& \int_\UU (\lambda g - \CCC^* g) \varphi \widetilde m^2   + \int_\UU g ( \partial_t   + v \cdot \nabla_x   ) (\varphi\widetilde m^2) , 
\eean
which is coercive for $\lambda$ large enough thanks to  Lemma~\ref{lem:EstimKFP*Lq12} (and more precisely \eqref{eq:disssipSL*-Lq}-\eqref{eq:disssipSL*-LqBIS}). 
Using Lions' variant of the Lax-Milgram theorem \cite[Chap~III,  \textsection 1]{MR0153974}, we obtain a variational solution $g \in \HHH_m$ to \eqref{eq:linear_gGinflow}, and more precisely
$$
\EEE(g,\varphi) = \int_{\Gamma_+} \mathfrak{g}  \varphi \widetilde m^2 \, d\xi_1  + \int_\OO g_T \varphi(T,\cdot) \widetilde m^2  \, d v \, d x, \quad \forall \, \varphi \in C_c^1((0,T] \times \OO \cup \Gamma_+).
$$
Thanks to the trace Theorem~\ref{theo:Kolmogorov-trace} and the dissipativity property \eqref{eq:disssipSL*-Lq} of $\LLL^*$, we deduce that $g\in C([0,T];L^2_m) \cap  \HHH_m $.

\medskip\noindent
\textit{Step 2.} For a sequence $\alpha_k \in (0,1)$, $\alpha_k \nearrow 1$, we build a  sequence $(g_k)$ of solutions to the modified Maxwell reflection boundary condition problem
\begin{equation}\label{eq:linearDual_gak}
\left\{
\begin{aligned}
& -\partial_t g_k - v \cdot \nabla_x g_k = \CCC^* g_k \quad &\text{in}&\quad (0,T) \times \OO \\
& \gamma_{+} g_k   =  \alpha_k \RRR^* \gamma_{-} f_k  \quad &\text{on}&\quad (0,T) \times \Sigma_{+} \\
&g_k(t=T, \cdot) = g_T \quad &\text{in}&\quad   \OO, 
\end{aligned}
\right.
\end{equation}
by using Step~1, the fact that $\RRR^* : L^2_{\tilde m} (\Sigma_-;d\xi_1) \to L^2_{\tilde m} (\Sigma_-;d\xi_1) $ from Step~1 in  Lemma~\ref{lem:EstimKFP*Lq12} and the Banach fixed point Theorem. 
This sequence satisfies  
$$
\sup_{[0,T]} \| g_{kt} \|^2_{L^2_{\tilde m}} + \int_0^T \left\{ \| \gamma g_{ks} \|^2_{L^2_m(\Gamma_+; \d\xi_2)}
+   
 \| g_{ks} \|^2_{H^{1,\dagger}_{\tilde m}} 
\right\} \,  \d s 
\le
C_T  \| g_T \|^2_{L^2_{\tilde m}}  
$$
for some constant $C_T$ and  any $k \ge 1$. We may extract converging subsequences $(g_{k'})$ and $(\gamma g_{k'})$ with associated limits $g$ and $\bar\gamma$, and passing to the limit in \eqref{eq:linearDual_gak} with the help of Theorem~\ref{theo:Kolmogorov-stability}, we deduce that $\bar\gamma = \gamma g$ and that $g$ is a renormalized solution to \eqref{eq:dualKFP}. 

\medskip\noindent
\textit{Step 3.} We may assume $\omega \lesssim e^{\zeta\langle v \rangle^{s_*}}$ with $s_* = \min(2,\gamma)$ and $\zeta \in (0,\zeta^*)$, defining $\zeta^* := b_0/2$ if $\gamma < 2$, 
$\zeta^* := \min (b_0,1/\Theta^*)/2$ if $\gamma = 2$, $\zeta^*  := 1/(2\Theta^*)$ if $\gamma > 2$. We set $m_1 := e^{-\zeta_1\langle v \rangle^{s_*}}$, with $\zeta_1 \in (\zeta,\zeta_*)$. 
We consider two solutions $g_1$ and $g_2 \in  C([0,T];L^2_m) \cap \HHH_m$ to the backward dual KFP equation \eqref{eq:dualKFP} associated to the same final datum $g_T$, so that the function $g := g_2 - g_1 \in  C([0,T]; L^2_m) \cap \HHH_m$  is  a solution to the KFP equation   \eqref{eq:KFP}-\eqref{eq:KolmoBdyCond} associated to the initial datum $g(T) = 0$.
Choosing $\varphi := \widetilde m_1\chi_R$  in \eqref{eq:FPK-traceL2}, with the notations of  Step~4 of the proof of Proposition~\ref{prop-KFP-L2primal}, and proceeding similarly, we get 
\bean
&&  \int_\OO \beta(g(0))  \varphi \, \d v \, \d x  +  \int_{\Gamma} \beta(\gamma g) \,  \varphi  \,  (n_x \cdot v) \,  \d v \, \d\sigma_{\! x} \, \d t +  \int_{\UU}  \beta''(f) \, | \nabla_v f |^2  \varphi  \, \d v \, \d x \, \d t 
\\
 &&\qquad =  
 \int_{\UU} \beta(g) \, \TTT  \varphi + 
 \int_{\UU}  \left\{  \beta(g) \left(  \Delta_v \varphi  + \Div_v (b  \varphi ) \right)  +  (c-\Div_v b)f   \beta'(f)  \varphi  \right\} \d v \, \d x \, \d t. 
\eean
Taking advantage of the fact that all the terms in the interior and at the boundary are now well defined, we may argue as for the proof of the $L^1_{\widetilde m_1}$ estimate performed in Lemma~\ref{lem:EstimKFP*Lq12}, 
and we conclude that $g \equiv 0$ as in Step~4 of the proof of Proposition~\ref{prop-KFP-L2primal}.
 \end{proof}

We conclude the section by reformulating and slightly improving the two previous well-posedness results.

\begin{theo}\label{theo:WP-KFP&KFP*}
Consider an admissible weight function $\omega$ and set $m:=\omega^{-1}$.  

(1) There exists a semigroup $ S_{\LLL}$ on $ L^2_\omega(\OO)$ such that for any  $f_{0} \in L^2_\omega(\OO)$, the function $f_t := S_{\LLL}(t) f_{0}$  is the  unique solution in  $C([0,T];L^2_\omega) \cap  \HHH$, $\forall \, T >0$, to the KFP equation \eqref{eq:KFP}--\eqref{eq:KolmoBdyCond}--\eqref{eq:initialDatum}. 
Furthermore \eqref{eq:lem:EstimKFPLp12} holds if additionally $f_0 \in L^p_\omega$ for some $p \in [1,\infty]$. 

(2) Similarly, there exists a semigroup $ S_{\LLL^*}$ on $ L^2_m(\OO)$ such that for any  $g_{T} \in L^2_m(\OO)$, the function $g_t := S_{\LLL^*}(T-t) g_{T}$  is the  unique solution in  $C([0,T];L^2_m) \cap  \HHH_m$, $\forall \, T >0$, to the dual KFP problem \eqref{eq:dualKFP}.  
Furthermore \eqref{eq:g0&g-<gtC} holds if additionally $g_T \in L^q_m$ for some $q \in [1,\infty]$. 

(3) The semigroups $ S_{\LLL}$ and $ S_{\LLL^*}$ are dual one toward the other. In other words, the equation \eqref{eq:identite-dualite} holds for any $f_0 \in L^2_\omega(\OO)$ and $g_T \in L^2_m(\OO)$.
\end{theo}

\begin{proof}[Proof of Theorem~\ref{theo:WP-KFP&KFP*}] The proof is split into four steps. 

\smallskip\noindent
{\sl Step~1.} We may define the semigroup $S_\LLL$ by setting $S_{\LLL}(t) f_{0} := f_t $ for any $f_0 \in L^2_\omega$ and $t \ge 0$, where $f_t$ is the   unique solution in   $X_T$, $\forall \, T >0$,  to the KFP equation \eqref{eq:KFP}--\eqref{eq:KolmoBdyCond}--\eqref{eq:initialDatum} provided by  Proposition~\ref{prop-KFP-L2primal}. In  particular  \eqref{eq:lem:EstimKFPLp12} holds for $p=2$.  Proceeding as in Step~4 during the proof of Proposition~\ref{prop-KFP-L2primal}, we may justify the computations performed during the proof of Lemma~\ref{lem:EstimKFPLp12} and we get that  \eqref{eq:lem:EstimKFPLp12} holds for $p=1$ when $f_0 \in L^2_\omega \cap L^1_{\omega}$. By interpolation, we obtain that  \eqref{eq:lem:EstimKFPLp12} holds for any $p \in [1,2]$ when $f_0 \in L^2_\omega \cap L^p_{\omega}$. 

\smallskip\noindent
{\sl Step~2.} From the well-posedness result of Proposition~\ref{prop-KFP-L2dual}, we may define the semigroup $S_{\LLL^*}$ by setting $S_{\LLL^*}(t) g_T := g(T-t)$ for any $g_T \in L^2_m$ and any $0 \le t \le T$, where $g$ is unique solution in  $ Y_T$ 
to the KFP problem  \eqref{eq:dualKFP}. We obtain as in Step 1 that furthermore \eqref{eq:g0&g-<gtC} holds for any $q \in [1,2]$ if $g_0 \in L^2_m \cap L^q_m$. 

\smallskip\noindent
{\sl Step~3.} We change $\iota$ by $\iota_n : = \iota_{S,n} + \iota_{D,n} \le 1- 1/n$, with $ \iota_{S,n} :=  \iota_S (1-1/n)$, $ \iota_{D,n} :=  \iota_D (1-1/n)$, and we denote by $\RRR_n$ and $\RRR_n^*$ the corresponding reflection operators. 
Denoting by $f_n$ the solution associated to the KFP equation, the reflection operator $\RRR_n$  and the initial datum $f_0$ given by Step~1 (or Proposition~\ref{prop-KFP-L2primal}), we have the additional property 
$\gamma f_n \in L^2_\omega(\Gamma;d\xi_1)$. The solution $g_n$ associated  to the dual problem \eqref{eq:dualKFP} for the reflection operator $\RRR^*_n$  and the final datum $g_T$ given by Step~2 (or Proposition~\ref{prop-KFP-L2dual}) also satisfies the additional property $\gamma g_n \in L^2_m(\Gamma;d\xi_1)$. Because of these additional estimates on the boundary, we may justify the computations leading  to the identity \eqref{eq:identite-dualite} by starting from \eqref{eq:FPK-traceL2} applied with $g := f_n$, $\varphi := g_n$ and $\beta(s) := s$ or by applying (a variant of) Theorem~\ref{theo:Kolmogorov-trace} to the function $\beta(f_ng_n)$, noticing that 
$$
\partial_t(f_ng_n) + v \cdot \nabla_x (f_ng_n) = \Delta_v f_ng_n - f_n \Delta_v g_n + \Div_v (bf_ng_n) \  \hbox{ in } \ \DD'(\UU).
$$
In that way, we obtain 
\begin{equation}\label{eq:identite-dualite-fngn}
\int_\OO f_n(T) g_T   
= \int_\OO f_0 g_n(0), \quad \forall \, n \ge 1.
\end{equation}
Because $(f_n)$ is bounded in $L^\infty(0,T;L^2_\omega)$ and in $W^{1,\infty}(0,T;\DD'(\OO))$, we deduce that $f_n(T) \wto f (T) := S_\LL(T) f_0$ weakly in $L^2_\omega$. Similarly, we have 
$g_n(0) \wto g(0) = S_{\LLL^*}(T) g_T$ weakly in $L^2_m$. We may thus pass to the limit $n\to\infty$ in \eqref{eq:identite-dualite-fngn} and we deduce that \eqref{eq:identite-dualite} holds, which 
exactly means that $(S_{\LLL^*})^* = S_\LLL$. 

\smallskip\noindent
{\sl Step~4.} For any $p \in (2,\infty]$, we know from Step~2 that for any $g_T \in L^2_m \cap L^{p'}_m$, $T > 0$, there holds 
\beqn\label{eq:cor:LpwellposednessWithReflectionDual}
\| g (t, \cdot) \|_{L^{p'}_m} \le C_1 e^{C_2(T-t)} \| g_T \|_{L^{p'}_m}, 
\quad \forall \, t \in [0,T].
\eeqn
Now, for $f_0 \in L^p_\omega$, we have 
 \bean
 \| f(T) \|_{L^p_\omega} 
 &=& \sup_{g_T \in L^2_m, \| g_T \|_{L^{p'}_m} \le 1} \int f(T) g_T
 \\
 &=&\sup_{g_T \in L^2_m, \| g_T \|_{L^{p'}_m} \le 1} \int f_0 g(0)
 \\
&\le& \| f_0 \|_{L^p_\omega} \sup_{g_T \in L^2_{m}, \| g_T \|_{L^{p'}_m} \le 1}  \| g(0) \|_{L^{p'}_m} 
 \\
&\le&\| f_0 \|_{L^p_\omega}  \sup_{g_T \in L^2_{m}, \| g_T \|_{L^{p'}_m} \le 1} C_1 e^{C_2T}  \| g_T  \|_{L^{p'}_m}
=  C_1 e^{C_2T}\| f_0 \|_{L^p_\omega}, 
\eean
where we have successively used the Riesz representation theorem, the duality identity  \eqref{eq:identite-dualite}, 
the Holder inequality, the estimate \eqref{eq:cor:LpwellposednessWithReflectionDual} and  the Riesz representation theorem again.
We have thus established that  \eqref{eq:lem:EstimKFPLp12} holds for any $f_0 \in L^p_\omega \cap L^2_\omega$,  $p \in [1,\infty]$. 
We establish in the same way that \eqref{eq:g0&g-<gtC} holds for any $g_T \in L^q_m \cap L^2_m$,  $q \in [1,\infty]$. 

\smallskip\noindent
{\sl Step~5.} For $f_0 \in L^2_\omega$, let us introduce a sequence $f_{0,n} \in L^2_{\omega^\sharp} \cap L^\infty_\omega$, such that $\omega^\sharp$ is an admissible weight function satisfying $\omega^\sharp/\omega \to \infty$ and such that $f_{0,n} \to f_0$ in $L^2_\omega$. From the previous analysis, the solution $f_n$ to the KFP equation \eqref{eq:KFP}--\eqref{eq:KolmoBdyCond} associated to the initial datum $f_{0,n}$   satisfies $f_{n} \in L^\infty(0,T;L^2_{\omega^\sharp} \cap L^\infty_\omega) \cap C([0,T];\Lloc^2) \subset C([0,T];L^2_\omega)$. 
From \eqref{eq:lem:EstimKFPLp12}, $(f_n)$ is a Cauchy sequence in $C([0,T];L^2_\omega)$ and thus converges to the solution $f$  to the KFP equation \eqref{eq:KFP}--\eqref{eq:KolmoBdyCond} associated to the initial datum $f_0$.  We have established that $f \in C([0,T];L^2_\omega)$ and the same argument holds for the dual problem.
\end{proof}

\section{Ultracontractivity}  \label{sec:ultra}

In this section, we explain how to adapt to  the  KFP equation in a domain the 
 De Giorgi-Nash-Moser theory developed   for parabolic equations, in particular in \cite{MR0093649,MR0100158,MR0170091,MR159139}, and generalized recently to the KFP equation in 
 the whole space, in particular in  \cite{MR2068847,MR3923847}. The gain of integrability $L^1 \to L^2$ essentially follows and slightly simplifies the proofs presented in \cite{CM-KFP**,CM-Landau**} which are very in the spirit of  Nash approach \cite{MR0100158}. 
When necessary, we will deal with the strongly and weakly confining cases separately.

\subsection{An improved weighted $L^2$ estimate at the boundary}

Let us observe that for  a  solution $f$ to the  KFP equation \eqref{eq:KFP}, we may write 
$$
\TT \frac{f^2 }{ 2} = f  \, \TT f = f  \, \CCC f,
$$
where we define 
\beqn\label{def:TT}
\TT := \partial_t + v \cdot \nabla_x, 
\eeqn
and we recall  that $\CCC$ has been defined in \eqref{eq:intro-defCCCf}. Multiplying that equation by $\Phi^2 := \varphi^2 \widetilde\omega^2$ with a time truncation function $\varphi \in \DD(0,T)$ and a weight function $\widetilde\omega : \OO \to (0,\infty)$, and next integrating in all the variables with the help of  \eqref{eq:identCCCffp}, 
we obtain
\beqn\label{eq:weakKolmogorovLq}
\frac12\int_\Gamma (\gamma f)^2 \Phi^2  n_x \cdot v - \frac12 \int_\UU f^2 \TT \Phi^2
= -  \int_\UU |\nabla_v (f\Phi)|^2 + \int_\UU f^2 \Phi^2  \widetilde\varpi,  
\eeqn
with $\UU := (0,T) \times \OO$, $\Gamma := (0,T) \times \Sigma$, $T \in (0,\infty)$ and $ \widetilde\varpi  := \varpi_{\tilde\omega,2}^\CCC$
is defined in  \eqref{def:varpi}, or equivalently by 
\beqn\label{def:varpiBIS}
\widetilde\varpi
=  \frac{1 }{ 4}  \frac{|\nabla_v \widetilde \omega^2|^2 }{ \widetilde \omega^4}  
- \frac12 b   \cdot \frac{\nabla_v \widetilde \omega^2 }{ \widetilde \omega^2}
+  c  - \frac{1 }{ 2}  \Div_v b . 
\eeqn

 \smallskip
 We first establish a key estimate on  the KFP equation \eqref{eq:KFP}--\eqref{eq:KolmoBdyCond}--\eqref{eq:initialDatum} which makes possible to control a solution near the boundary. 
 The proof is based on the introduction of an appropriate weight function which combines the twisting term used in the previous section and the  twisting term used in \cite[Section~11]{sanchez:hal-04093201}, that last one being in the spirit of  moment arguments used in \cite{MR1166050,MR3591133}.

\begin{prop}\label{prop:TheLqEstim}
For any admissible weight function $\omega$ there exists  $C = C(\omega,\Omega)>0$  such that for any solution $f$  to the KFP equation \eqref{eq:KFP}--\eqref{eq:KolmoBdyCond}, any $T>0$ and any smooth function $0 \le \varphi \in \DD((0,T))$,  there holds
$$
\int_\UU f^2  \omega^2 \left\{  \frac{(n_x \cdot \hat v)^{2}}{\delta^{1/2}} + \la v \ra^\varsigma \right\} \varphi^2+  
\int_\UU  |\nabla_v(f \omega )|^2  \varphi^2
\le 
C \int_\UU  f^2 \omega^2 \left[ |\partial_t \varphi^2| +  \varphi^2 \right],
$$
with  $\varsigma := \gamma+s-2$. 
\end{prop}

 It is worth emphasizing that an admissible weight function $\omega$ is strongly confining if and only if  the associated parameter $\varsigma \in \R$ satisfies $\varsigma > 0$.

\begin{proof}[Proof of Proposition~\ref{prop:TheLqEstim}] 
We introduce the function 
\bean
\Phi^2 := \varphi^2 \,  \widetilde \omega^2, \quad   \widetilde \omega^2 :=   \left( 1 + \frac{1}{4}\, \frac{ n_x \cdot v}{\la v \ra^4}+ \frac{1}{4 D^{1/2}} \, \delta(x)^{1/2}  \, \frac{ n_x \cdot v}{\la v \ra^2} \right) \omega_A^2,
\eean
where the weight function $\omega_A$ is defined in \eqref{def:omegaA} for $A \ge 1$ large enough (to be fixed below) and 
where  $D= \sup_{x\in \Omega}\delta (x)$ is half the diameter of $\Omega$, so that in particular an estimate similar to \eqref{eq:m&mtilde} holds. From \eqref{eq:weakKolmogorovLq} we have 
\begin{equation}\label{eq:integrate_fq}
\begin{aligned}
 2  \int_\UU |\nabla_v (f\Phi)|^2
 &+ \int_\Gamma (\gamma f)^2 \Phi^2  (n_x \cdot v )
 - \int_\UU f^2 v \cdot \nabla_x \Psi_2
 \\
&=  \int_\UU f^2 v \cdot \nabla_x \Psi_{1} 
+ 2  \int_\UU f^2 \Phi^2 \widetilde \varpi
+ \int_\UU f^2 \partial_t (\Phi^2), 
\end{aligned}
\end{equation}
for $ \widetilde \varpi$ as defined on \eqref{def:varpiBIS} and we denote 
\bean
 \Psi_{1} := \varphi^2 \omega_A^2 \left( 1  +  \frac{1}{4}\,   \frac{ n_x \cdot v}{\la v \ra^4} \right), \quad
\Psi_2 :=  \frac{\varphi^2 \omega_A^2}{4 D^{1/2}} \, \delta(x)^{1/2}  \, \frac{ n_x \cdot v}{\la v \ra^2}.
\eean
We now compute each term separately.

\smallskip\noindent
\textit{Step 1.} 
Observing that $\widetilde \omega^2 = \left( 1 + \frac{1}{4}\,  n_x \cdot \frac{v}{\la v \ra^4}\right) \omega_A^2$ on the boundary $\Gamma=(0,T)\times \Sigma \times \R^d$, we can argue as in Step~1 in the proof of Lemma~\ref{lem:EstimKFPLp12} to deduce that, choosing $A\ge 1$ large enough, the contribution of the boundary term in \eqref{eq:integrate_fq} is nonnegative, that is
$$
\int_\Gamma (\gamma f)^2 \Phi^2  (n_x \cdot v ) \ge 0.
$$

\smallskip\noindent
\textit{Step 2.}
In order to deal with the third term at the left-hand side of \eqref{eq:integrate_fq}, we define  $\psi := \delta(x)^{1/2}  (n_x \cdot   v) \la v \ra^{-2} $. Observing that 
 $\langle v \rangle \psi \in L^\infty(\OO)$, $ \nabla_v \psi \in L^\infty(\OO)$ and 
$$
- v \cdot \nabla_x \psi = \frac12 \frac{1}{\delta(x)^{1/2}} (n_x \cdot v)^2 \la v \ra^{-2} -  \delta(x)^{1/2} ( v \cdot D_x n_x  v) \la v \ra^{-2}, 
$$
we have
\bean
-\int_\UU f^2 v \cdot \nabla_x \Psi_2  =
\frac{1}{  4 D^{1/2}}  \int_\UU f^2 \varphi^2 \omega_A^2 \left\{ \frac12 \frac1{\delta(x)^{1/2}} (n_x \cdot \hat v)^2 -  \delta(x)^{1/2} (\hat v \cdot D_x n_x \hat v) \right\}.
\eean

For the first term at the right-hand side of \eqref{eq:integrate_fq}, a direct computation gives
$$
\int_{\UU} f^2 v \cdot \nabla_x \Psi_{1}
= \frac14\int_{\UU} f^2 \varphi^2 \omega_A^2 \la v \ra^{-2} (\hat v \cdot D_x n_x \hat v).
$$

\smallskip\noindent
\textit{Step 3.} 
Writing $\widetilde \varpi \le 2 \langle \widetilde \varpi_+ \rangle - \langle \widetilde \varpi_- \rangle$ and gathering previous estimates yields
$$
\begin{aligned}
&2  \int_\UU |\nabla_v (f \widetilde \omega)|^2 \varphi^2
+\frac{1}{8 D^{1/2}}  \int_\UU f^2 \varphi^2 \omega_A^2 \frac{(n_x \cdot \hat v)^2}{\delta(x)^{1/2}} 
+2\int_{\UU} f^2 \varphi^2 \widetilde \omega^2   \langle \widetilde \varpi_- \rangle
 \\
&\le  \int_{\UU} f^2 \varphi^2 \omega_A^2 (\hat v \cdot D_x n_x \hat v) \left\{ \frac{1}{4 \la v \ra^2} + \frac{\delta^{1/2}}{4 D^{1/2}}\right\}
+ 4  \int_\UU f^2 \varphi^2 \omega_A^2 \langle \widetilde \varpi_+ \rangle
+ \int_\UU f^2 \widetilde \omega^2 \partial_t (\varphi^2) \\
&\le C_{\Omega,A}   \int_\UU  f^2  \omega^2 \langle \widetilde \varpi_+ \rangle \varphi^2  + C_A \int_\UU f^2  \omega^2  | \partial_t \varphi^2| ,
\end{aligned}
$$
where we recall that $\delta \in W^{2,\infty}(\Omega)$. 
Using that $\langle \widetilde \varpi_- \rangle  \ge \kappa_0 \langle v \rangle^\varsigma$ with $\kappa_0 > 0$, because of \eqref{eq:varpi-asymptotic}-\eqref{eq:disssipSL-LpBIS}, and also that 
$\widetilde\varpi_+ $ is bounded, we deduce
$$
\begin{aligned}
&2  \int_\UU |\nabla_v (f \widetilde \omega)|^2 \varphi^2
+\frac{1}{8 D^{1/2}}  \int_\UU f^2 \varphi^2 \omega_A^2 \frac{(n_x \cdot \hat v)^2}{\delta(x)^{1/2}} 
+2\kappa_0 \int_{\UU} f^2 \varphi^2 \widetilde \omega^2 \la v \ra^\varsigma\\
&\le C_{\Omega,A}   \int_\UU  f^2  \omega^2  \varphi^2  + C_A \int_\UU f^2  \omega^2  | \partial_t \varphi^2| .
\end{aligned}
$$
We then conclude by observing that
$$
\int_{\OO} |\nabla_v (f \omega)|^2 + f^2 \omega^2
\lesssim \int_{\OO} |\nabla_v (f \widetilde \omega)|^2 + f^2 \widetilde \omega^2
$$
and using that $\omega \lesssim \omega_A \lesssim \widetilde \omega \lesssim \omega_A \lesssim \omega$.
\end{proof}

\medskip
We may write the above weighted $L^2$ estimate in a more convenient way, where the penalization of a neighborhood of the boundary is made clearer. 
For that purpose we state the following interpolation result which formalizes and improves some estimates used during the proof of  \cite[Lemma~11.9]{sanchez:hal-04093201} (see also \cite[Lemma~3.2]{CM-KFP**}). 
 
\begin{lem}\label{lem:EstimL2poids} For $d \ge 1$, for any $\varsigma > 0$ and any function $g : \OO \to \R$, there holds 
\beqn\label{eq:EstimL2poids}
\int_{\OO}  \frac{g^2 }{ \delta^{\beta}} \lesssim
\int_\OO |\nabla_v g|^2  + \int_\OO g^2 \left(   \frac{(n_x \cdot \hat v)^2 }{   \delta^{1/2} }  +   \langle v \rangle^\varsigma \right), 
\eeqn
 with $\beta = [2(d'+1 + 2d'/\varsigma)]^{-1}$, $d' := \max(d,3)$. 
\end{lem}

\begin{proof}[Proof of Lemma~\ref{lem:EstimL2poids}] 
For $\alpha,\beta,\eta > 0$,  we start   writing
\bean
\int_{\OO}  \frac{g^2 }{ \delta^{\beta}}
&=& 
\int_{\OO} \frac{ g^2 } {  \delta^{ \beta}}    \mathbf{1}_{\langle v \rangle \delta^{\alpha} \ge 1 }
+
\int_{\OO} \frac{ g^2 }{  \delta^{ \beta}}    \mathbf{1}_{(n_x \cdot v)^{2} > \delta^{2\eta}}  \mathbf{1}_{\langle v \rangle \delta^{\alpha} < 1 }
\\
&&
+ \int_{\OO}   \frac{ g^2 }{  \delta^{ \beta}}  \mathbf{1}_{|n_x \cdot v|\le \delta^{\eta}}   \mathbf{1}_{\langle v \rangle \delta^{\alpha} < 1 } = : T_1 + T_2 + T_3.
\eean
For the first term, we have 
\bean
T_1 
\le
\int_{\OO}  \frac{ g^2 } {  \delta^{ \beta}}    {\langle v \rangle^\varsigma  \delta^{\varsigma\alpha}}
=
\int_{\OO} g^2   \langle v \rangle^\varsigma ,
\eean
by choosing $\alpha\varsigma = \beta$. 
For the second term,   we have 
\bean
T_2 
\le
\int_{\OO} \frac{ g^2 }{  \delta^{ \beta}}  \frac{1 }{ \langle v \rangle^2 \delta^{2\beta/\varsigma}}   \frac{(n_x \cdot v)^{2} }{ \delta^{2\eta}}
= \int_{\OO} g^2   \frac{(n_x \cdot \hat v)^{2} }{   \delta^{1/2}} ,
\eean
by choosing   $2\eta + \beta (1+2/\varsigma) = 1/2$.
For the third term, we define $2^* := 2d'/(d'-2)$  the  Sobolev exponent in dimension $d'   \ge 3$, 
and we compute
\bean
T_3
&\le&  
\int_\Omega  \delta^{-\beta}  \Bigl( \int_{\R^d}  |g|^{2^*}  \Bigr)^{2/2^*}  
\Bigl( \int_{\R^d}     \mathbf{1}_{ |n_x \cdot v | \le \delta^{\eta}}    \mathbf{1}_{\langle v \rangle < \delta^{-\beta/\varsigma} }  \Bigr)^{2/d}
\\
&\lesssim&  
  \int_{\Omega} \delta^{-\beta} (\delta^\eta \delta^{-(d-1)\beta/\varsigma})^{2/d}    \int_{\R^d} (|\nabla_v g|^2 + g^2)
\eean
where we have used the H\"older inequality in the first line and the Sobolev inequality in the second line together with the observation that 
 $\langle v' \rangle \le \langle v \rangle$ in the orthonormal representation $v = v_1 n_x + v'$.  
Choosing $\eta$   such that 
$$
- \beta -  2  (1-1/d) \beta/\varsigma + \eta2/d = 0, 
$$
we deduce that $\eta = (d/2 + (d-1)/\varsigma) \beta$,  then $\beta = [2(d+1 + 2d/\varsigma)]^{-1}$  and we conclude to \eqref{eq:EstimL2poids}. 
\end{proof}

As an immediate consequence of Proposition~\ref{prop:TheLqEstim} and the interpolation inequality \eqref{eq:EstimL2poids}, we get the following estimate
which holds for strongly confining weight functions.

\begin{cor}\label{cor:EstimBord}  Consider a strongly confining weight function $\omega$ and recall that $\varsigma := \gamma + s -2 > 0$ and 
   $\beta  := [2(d'+1 + 2d'/\varsigma)]^{-1}  $. 
Under the conditions of Proposition~\ref{prop:TheLqEstim}, there exists $C_\Omega$ such that 
$$
\int_\UU \frac{f^2 }{ \delta^{\beta}} \, \omega^2    \varphi^2
+ \int_\UU f^2 \la v \ra^{\varsigma} \omega^2  \varphi^2
+ \int_\UU  |\nabla_v(f \omega )|^2  \varphi^2 \lesssim
\int_\UU  f^2  \omega^2 [|\partial_t \varphi^2| + C_\Omega \varphi^2]. 
$$
\end{cor}

\smallskip

For a weakly confining admissible weight function, we obtain the following weaker estimate which is similar to  \cite[Proposition 3.3]{CM-KFP**} and \cite[Proposition 5.3]{CM-Landau**}.

\begin{cor}\label{cor:EstimBordWeak}  
Let us consider a weakly confining admissible weight function $\omega$ corresponding to the case when  $\varsigma := \gamma + s -2 \le 0$.
 We set $\beta  := [2(d'+1)]^{-1}  $, recalling that $d' := \max(d,3)$.
Under the conditions of Proposition~\ref{prop:TheLqEstim}, there holds 
$$
\int_\UU \frac{f^2 }{ \delta^{\beta}} \, \frac{\omega^2  }{ \langle v \rangle^2}   \varphi^2
+ \int_\UU  |\nabla_v(f \omega )|^2  \varphi^2 \lesssim
\int_\UU  f^2  \omega^2 [|\partial_t \varphi^2| +   \varphi^2]. 
$$
\end{cor}

\begin{proof}[Proof of Corollary~\ref{cor:EstimBordWeak}] 
We just use the inequality 
$$
\int_{\OO}  \frac{g^2 }{  \langle v \rangle^2 \delta^{\beta}} \lesssim
\int_\OO  ( |\nabla_v g|^2 + g^2)  + \int_\OO g^2   \frac{(n_x \cdot \hat v)^2 }{    \delta^{1/2} }  , 
$$
 that we may establish proceeding exactly as in 
\cite[Lemma~3.2]{CM-KFP**}, with $g := f   \omega$ and the conclusion of Proposition~\ref{prop:TheLqEstim}. 
\end{proof}

\subsection{A downgraded weighted $L^2-L^p$ estimate}
Taking advantage of a known $L^2-L^p$ estimate available for the KFP equation set in the whole space, and thus in the interior of the domain, we deduce a downgrade weighted $L^2-L^p$ estimate.

\begin{prop}\label{prop:EstimLploc}
 We set $\nu := \max(2,\gamma-1)$.
There exists $p > 2$,  $\alpha > p$  and $C  \in (0,\infty)$ such that any solution $f$ to the KFP equation \eqref{eq:KFP}--\eqref{eq:KolmoBdyCond} satisfies 
\beqn\label{eq:EstimLploc}
 \left\| f \varphi  \frac{\omega}{\langle v \rangle^\nu}  \delta^{\alpha/p} \right\|_{L^p(\UU)} 
\le C  \| (\varphi + |\varphi'|) f \omega \|_{L^2(\UU)},
\eeqn
for any $0 \le \varphi \in \DD((0,T))$, any $T>0$ and any admissible weight function $\omega$.
\end{prop}

We follow the proof of \cite[Lemma~10]{MR3923847} and of Step 2 in the proof of \cite[Proposition 3.5]{CM-KFP**}.

\begin{proof}[Proof of Proposition~\ref{prop:EstimLploc}] 
We split the proof into two steps.

\smallskip\noindent
{\sl Step 1.} 
Consider a subset $\Omega'$ of $\Omega$ such that $\overline{\Omega'} \subset \Omega$.
We introduce a truncation function $\chi \in \DD(\Omega)$ such that $\mathbf{1}_{\Omega'} \le \chi \le 1$, and the function $\omega_0 := \langle v \rangle^{-\nu} \omega$. 
We define the function $\bar f := \varphi \chi \omega_0 f$ which satisfies
$$
\partial_t \bar f + v \cdot \nabla_x \bar f =  F \quad\hbox{in}\quad \R   \times \R^{d} \times \R^d, 
$$
where $F := F_0 + \Div_v F_1$ with 
$$
\begin{aligned}
F_0 
&:=  f \omega_0 (\varphi' \chi + \varphi v \cdot \nabla_x \chi) + \varphi \chi  (\omega_0 b - \nabla_v \omega_0) \cdot \nabla_v f + c \bar f 
\end{aligned}
$$
and
$$
F_1 :=  \varphi \chi \omega_0 \nabla_v f  .
$$
From \cite[Theorem 1.3]{MR1949176} with $p = 2$, $r = 0$, $\beta = 1$, $m = 1$, $\kappa =1$ and $\Omega = 1$, we have 
$$
\begin{aligned}
\| D^{1/3}_t \bar f \|_{L^2(\R^{1+2d})}^2 + \| D^{1/3}_x \bar f \|_{L^2(\R^{1+2d})}^2 
&\lesssim
\| \bar f \|_{L^2(\R^{1+2d})}^2 + \| \nabla_v \bar f \|_{L^2(\R^{1+2d})}^2\\
&\quad + \| \langle v \rangle F_0 \|_{L^2(\R^{1+2d})}^2 + \| \langle v \rangle^2 F_1 \|_{L^2(\R^{1+2d})}^2 .
\end{aligned}
$$
On the one hand, a straightforward computation gives
$$
\| \bar f \|_{L^2(\R^{1+2d})} \lesssim \| \varphi \omega \langle v \rangle^{-\nu}  f  \|_{L^2(\UU)}
$$
and
$$
\| \nabla_v \bar f \|_{L^2(\R^{1+2d})} \lesssim \| \varphi \omega \langle v \rangle^{-1-\nu}  f  \|_{L^2(\UU)} + \| \varphi \la v \ra^{-\nu} \nabla_v ( f \omega) \|_{L^2(\UU)}.
$$
On the other hand, we have 
$$
\begin{aligned}
\| \langle v \rangle F_0 \|_{L^2(\R^{1+2d})}
&\lesssim \| \varphi' \omega \langle v \rangle^{1-\nu} f  \|_{L^2(\UU)} + \| \nabla \chi \|_{L^\infty (\Omega)} \| \varphi \omega \langle v \rangle^{2-\nu }  f  \|_{L^2(\UU)} \\
&\quad
+   \| \varphi \omega \langle v \rangle^{ \gamma - 2 - \nu}  f  \|_{L^2(\UU)}
+  \| \varphi \omega \langle v \rangle^{ \max(\gamma-1,s)-\nu} \nabla_v f \|_{L^2(\UU)}
\end{aligned}
$$
using the growth conditions  \eqref{eq:Assum-bc1}-\eqref{eq:Assum-bc2},  and
$$
\| \langle v \rangle^2 F_1 \|_{L^2(\R^{1+2d})}
\lesssim \| \varphi \omega \langle v \rangle^{2-\nu } \nabla_v f  \|_{L^2(\UU)}.
$$
Observing that $|\omega \nabla_v f| \lesssim |\nabla_v (f \omega)| + \la v \ra^{s-1} \omega |f|$, it follows
$$
\begin{aligned}
&\| \langle v \rangle F_0 \|_{L^2(\R^{1+2d})}
+\| \langle v \rangle^2 F_1 \|_{L^2(\R^{1+2d})} \\
&\qquad
\lesssim \| \varphi' \omega \langle v \rangle^{1-\nu} f  \|_{L^2(\UU)} + \| \nabla_x \chi \|_{L^\infty (\Omega)} \| \varphi \omega \langle v \rangle^{2-\nu }  f  \|_{L^2(\UU)} \\
&\qquad\quad
+   \|   \varphi \langle v \rangle^{ \max(2,\gamma-1,s)-\nu}  f  \|_{L^2(\UU)}
+  \| \varphi \langle v \rangle^{ \max(2,\gamma-1,s)-\nu} \nabla_v (f \omega) \|_{L^2(\UU)}.
\end{aligned}
$$
As $s \le 2$, we have   $\max(2,\gamma-1,s)=\max(2,\gamma-1) = \nu$.
Therefore, Corollary~\ref{cor:EstimBord} and the Sobolev embedding $H^{1/3}(\R^{1+2d}) \subset L^p(\R^{1+2d})$, with $p := 6d/(3d-2)> 2$, yield
\bean 
\| \bar f \|_{L^p(\R^{1+2d})}
&\lesssim& \| D^{1/3}_t \bar f \|_{L^2(\R^{1+2d})} + \| D^{1/3}_x \bar f \|_{L^2(\R^{1+2d})} + \| \nabla_v \bar f \|_{L^2(\R^{1+2d})} + \|  \bar f \|_{L^2(\R^{1+2d})} 
\\
&\lesssim&
 \| \varphi' \omega  f  \|_{L^2(\UU)} + \| \nabla_x \chi \|_{L^\infty(\Omega)} \| \varphi \omega   f  \|_{L^2(\UU)} +  \| \varphi \omega   f  \|_{L^2(\UU)}.
 \eean

\smallskip\noindent
{\sl Step 2.} Choosing  $\chi_k \in \DD(\Omega)$ such that $\mathbf{1}_{\Omega_{k+1}} \le \chi_k \le \mathbf{1}_{\Omega_k}$, with  $\Omega_k := \{ x \in \Omega \mid \delta > 2^{-k} \}$, and  $2^{-k} \| \nabla_x \chi_k \|_{L^\infty} \lesssim 1$ uniformly in $k \ge 1$,
we deduce from the last estimate  that 
$$
\| f \varphi  \omega_0 \|_{L^p(\UU_{k+1})}  \lesssim      \|  \varphi' f \omega \|_{L^2(\UU)} + 2^k    \| \varphi   f \omega \|_{L^2(\UU)}, 
$$
for any $ k \ge 1$, where we denote $\UU_k := (0,T) \times \Omega_k \times \R^d$.
Summing up, we obtain 
\bean
\int_\UU \delta^{\alpha}  (\varphi f\omega_0) ^{p} 
&=& \sum_k \int_{\UU_{k+1} \backslash \UU_k} \delta^{\alpha}  (\varphi  f \omega_0)^p
\\
&\lesssim& \sum_k  2^{-k   \alpha} \int_{\UU_{k+1}}  (\varphi f\omega_0)^p
\\
&\lesssim&  \sum_k 2^{ k(p-  \alpha)}   (  \|  \varphi' f \omega \|_{L^2(\UU)} +   \| \varphi   f \omega \|_{L^2(\UU)})^p
\\
&\lesssim& ( \|  \varphi' f \omega \|_{L^2(\UU)} +    \| \varphi   f \omega \|_{L^2(\UU)})^p
\eean
because $\alpha > p$, what is nothing but \eqref{eq:EstimLploc}.
\end{proof}

\subsection{The $L^1-L^r$ estimate up to the boundary}

We are now in position for stating our weighted $L^1-L^r$ estimate up to the boundary which is the well-known cornerstone step in the proof of De Giorgi-Nash-Moser gain of integrability estimate.
 
\begin{prop}\label{prop:EstimL1Lr-omega} 
There exists an exponent $r > 2$ such that any solution $f$ to the KFP equation \eqref{eq:KFP}--\eqref{eq:KolmoBdyCond} satisfies 
\beqn\label{eq:EstimL1LrStrong}
 \left\| f \varphi  \omega \right\|_{ L^2(\UU)} 
\lesssim    \|   (\varphi + |\varphi'|)^{2 \tfrac{r-1}{r-2}} \varphi^{- \tfrac{r}{r-2}} f \omega   \|_{L^1(\UU)}, 
\eeqn
for any $\varphi \in C^1_c((0,T))$, any $T > 0$ and any  strongly confining admissible weight function $\omega$.
\end{prop} 

The proof is a variant of the proof of \cite[Proposition 3.7]{CM-KFP**}. 

\begin{proof}[Proof of Proposition~\ref{prop:EstimL1Lr-omega}]
We set $\Delta := \delta^{\alpha/(p\nu)} \langle v \rangle^{-1}$ and we observe that Proposition~\ref{prop:EstimLploc}  writes 
\beqn\label{eq:EstimLplocBIS}
 \left\| f \varphi  \omega \Delta^\nu   \right\|_{L^p(\UU)} 
\lesssim  \| (\varphi + |\varphi'|) f \omega \|_{L^2(\UU)}.
\eeqn
From Corollary~\ref{cor:EstimBord} and H\"older's inequality, we have next
\bean
\left\| f \varphi \omega \langle v \rangle^{\theta \varsigma/2}  \delta^{-\beta (1-\theta)/2}  \right\|_{L^2(\UU)} 
&\le& \left\| f \varphi \omega \langle v \rangle^{\varsigma/2} \right\|_{L^2(\UU)}^\theta  \left\| f \varphi \omega   \delta^{-\beta /2}  \right\|_{L^2(\UU)}^{1-\theta}
\\
&\lesssim& \left\|  (\varphi + \sqrt{\varphi|\varphi'|}) f \omega   \right\|_{L^2(\UU)}, 
\eean
for any $\theta \in (0,1)$. Choosing $\theta = \theta_0$ such that 
$$
\frac{\beta}{\varsigma} \left( \frac{1}{\theta_0}-1 \right) = \frac{\alpha}{p\nu}, 
$$
and setting $\mu := \theta_0\varsigma/2$, we thus deduce 
\beqn\label{eq:EstimL2strongBIS}
 \left\| f \varphi  \omega \Delta^{-\mu }  \right\|_{L^2(\UU)} 
\lesssim \|  (\varphi + \sqrt{\varphi|\varphi'|}) f \omega   \|_{L^2(\UU)}.  
\eeqn
The H\"older inequality 
\beqn\label{eq:HolderIneq}
\| h \|_{L^r(\UU)} \le \| h  \|_{L^{p}_{\sigma_p}(\UU)}^{1-\theta} \| h\|_{L^{2}_{\sigma_2}(\UU)}^\theta, 
\eeqn
with $1/r := (1-\theta)/p + \theta/2$ and $1 = \sigma_p^{1-\theta} \sigma_2^\theta$ for any $\theta \in (0,1)$, implies
\bean
 \left\| f \varphi  \omega \right\|_{L^r(\UU)} 
&\lesssim& 
 \| f \varphi  \omega \Delta^{\mu (\theta /(1-\theta))}  \|_{L^p(\UU)}^{1-\theta} 
\| f \varphi  \omega \Delta^{-\mu }   \|_{L^2(\UU)}^\theta
\\
&\lesssim& 
 \| (\varphi + |\varphi'|) f \omega \|_{L^2(\UU)}^{1-\theta_1}
 \|  (\varphi + \sqrt{\varphi|\varphi'|}) f \omega   \|_{L^2(\UU)}^{\theta_1}  , 
\\
&\lesssim& 
 \| (\varphi + |\varphi'|) f \omega \|_{L^2(\UU)}, 
\eean
where we have chosen $\theta = \theta_1$ such that   $\mu (\theta_1/(1-\theta_1)) = \nu$, we have used the two inequalities \eqref{eq:EstimLplocBIS} and \eqref{eq:EstimL2strongBIS} in the second line, and finally the most classical form of the Young inequality in the last line.   We now use the  H\"older inequality
 $$
 \|  (\varphi + |\varphi'|) f \omega   \|_{L^2(\UU)} \le 
 \|  \varphi f \omega   \|_{L^r(\UU)}^{\theta_2}  \|   (\varphi + |\varphi'|)^{2 \tfrac{r-1}{r-2}} \varphi^{- \tfrac{r}{r-2}} f \omega   \|_{L^1(\UU)}^{1-\theta_2}
 $$
 with $\theta_2 := r/(2r-2)$ at the RHS of  the last estimate. After simplification of the terms involving the $L^r$ norm and taking the power $(1-\theta_2)^{-1}$ of the resulting inequality, we conclude to \eqref{eq:EstimL1LrStrong}.
\end{proof}
 
For a weakly confining weight function $\omega$, we have the following slightly weaker estimate.

\begin{prop}\label{prop:EstimL1Lr-omegaWeak} 

Let $T>0$ and $\omega$ a weakly confining admissible weight function, then there is $r>2$ such that any solution $f$ to the KFP equation \eqref{eq:KFP}--\eqref{eq:KolmoBdyCond} satisfies 
\beqn\label{eq:EstimL1LrWeak}
 \|  (\varphi + |\varphi'|) f \omega   \|_{L^2(\UU)} 
\lesssim    \|   (\varphi + |\varphi'|)^{2 \tfrac{r-1}{r-2}} \varphi^{- \tfrac{r}{r-2}} f \omega \langle v \rangle^K  \|_{L^1(\UU)}, 
\eeqn
for any $\varphi \in C^1_c((0,T))$ and $K := 2(3d+1)(2d+3)$.
\end{prop} 

\begin{proof}[Proof of Proposition~\ref{prop:EstimL1Lr-omega}]
We observe that for a  weakly confining weight function $\omega$, there always holds $\gamma \le 2$, $s \le 2-\gamma$, and thus $\nu =  \max(2,\gamma-1) = 2$. 
The estimates established in Corollary~\ref{cor:EstimBordWeak} and Proposition~\ref{prop:EstimLploc} then write for instance
$$
 \left\| f \varphi  \omega \langle v \rangle^{-1}  \delta^{-\beta/2} \right\|_{L^2(\UU)} 
\le C  \| (\varphi + |\varphi'|) f \omega \|_{L^2(\UU)},
$$
with $\beta  := [2(d+1)]^{-1}  $, and
$$
 \left\| f \varphi   \omega \langle v \rangle^{-2}  \delta^{1+1/p} \right\|_{L^p(\UU)} 
\le C  \| (\varphi + |\varphi'|) f \omega \|_{L^2(\UU)},
$$
with $1/p = 1/2 - 1/(2(2d+1))$. Using an interpolation argument, we get
\be \label{eq:HolderIneqWeakly}
 \left\| f \varphi   \omega \langle v \rangle^{-\mu} \right\|_{L^r(\UU)} 
\le C  \| (\varphi + |\varphi'|) f \omega \|_{L^2(\UU)},
\ee
by choosing $r \in (2,p)$ and $\theta_1 \in (0,1)$ in such a way that 
$$
\frac{1 }{ r} = \frac{\theta_1}{p} +\frac{1-\theta_1}2,
\quad
(1-\theta_1) \beta/2 = \theta_1 (1+1/p)
$$
and thus $\mu := 2 \theta_1 + (1-\theta_1) = \theta_1 +1$.
From the Holder inequality, we also have 
 \be \label{eq:InterpolationIneqWeakly}
 \|  (\varphi + |\varphi'|) f \omega   \|_{L^2(\UU)} \le 
 \|  \varphi f \omega \langle v \rangle^{-\mu} \|_{L^r(\UU)}^{\theta_2}  \|   (\varphi + |\varphi'|)^{2 \tfrac{r-1}{r-2}} \varphi^{- \tfrac{r}{r-2}} f \omega \langle v \rangle^K  \|_{L^1(\UU)}^{1-\theta_2}
 \ee
 with $\theta_2 := r/(2r-2)$ and $K= \frac{\theta_2}{ 1-\theta_2} \mu $. We now compute
$$
 p= \frac{2d+1}{ d}, \quad \theta_1 = \frac{2d+1}{ 12d^2+20d+5}, \quad r =  \frac{12d^2+20d+5 }{ 6d^2+10d+2},
$$
 so that
 \bean
 K &=& \frac{\theta_2}{ 1-\theta_2} \mu = \frac{r}{ r-2} (1+\theta_1)  = (12d^2+20d+5) \left( 1 + \frac{2d+1 }{ 12d^2+20d+ 5} \right) \\
 &=& 12d^2+22d+6 = 2(3d+1)(2d+3).
 \eean
The both last estimates together imply \eqref{eq:EstimL1LrWeak}. 
\end{proof}

\subsection{The $L^1-L^p$ estimate on the dual problem} 

We consider the dual  backward problem \eqref{eq:dualKFP} for which we establish the same kind of estimate as for the forward KFP  problem  \eqref{eq:KFP}--\eqref{eq:KolmoBdyCond}.
 In order to make the discussion simpler, we   separate the analysis for strongly and weakly confining admissible weight function in this section.

\begin{prop}\label{prop:EstimL1Lr-m} 
There exist some exponent $r > 2$ such that any solution $g$ to the  dual  backward problem \eqref{eq:dualKFP} satisfies 
\beqn\label{eq:EstimDualL1LrStrong}
 \left\| g \varphi  m \right\|_{L^r(\UU)} 
\lesssim    \|   (\varphi + |\varphi'|)^{2 \tfrac{r-1}{r-2}} \varphi^{- \tfrac{r}{r-2}}g m    \|_{L^1(\UU)}, 
\eeqn
for any test function $0 \le \varphi \in \DD((0,T))$, any $T>0$ and any function $m = \omega^{-1}$ which is the inverse of a strongly confining admissible weight function $\omega$.
\end{prop} 

\begin{proof}[Proof of Proposition~\ref{prop:EstimL1Lr-m}] 
The proof follows the same steps as for the proof of Proposition~\ref{prop:EstimL1Lr-omega} and we thus repeat it without too much details. 
 
\smallskip\noindent
{\sl Step 1. An improved weighted $L^2$ estimate at the boundary.} 
Let $\omega$ be an admissible weight function and define $m = \omega^{-1}$.

We define the  modified weight function  $\widetilde m$ by 
 $$
\widetilde m^2 :=m_A^2 \left( 1 +   \frac14 \,   \frac{ n_x \cdot v}{\la v \ra^4} - \frac1{4D^{1/2}}    \delta(x)^{1/2} \frac{n_x \cdot  v}{\la v \ra^2}   \right), 
$$
where $m_A$ has been defined in \eqref{def:mA} and $D = \sup_{x \in \Omega} \delta(x)$.
Considering a solution $g$ to the  dual  backward problem \eqref{eq:dualKFP}, multiplying the equation \label{eq:TT*gq} by $\Phi^2 := \varphi^2 \widetilde m^2$ with $\varphi \in \DD(0,T)$, and integrating in all the variables, we obtain
\bean
- \frac12\int_\Gamma (\gamma g)^2 \Phi^2  (n_x \cdot v) + \frac12 \int_\UU g^2 \TT \Phi^2
= \int_\UU g (\CCC^* g)   \Phi^2, 
\eean
with $\TT$ defined in \eqref{def:TT}.

Since $\widetilde m^2 :=m_A^2 ( 1 +   \frac14 \,   \frac{ n_x \cdot v}{\la v \ra^4} )$ on the boundary $\Gamma$, Step~1 of the proof of Lemma~\ref{lem:EstimKFP*Lq12} implies that
$$
- \frac12\int_\Gamma (\gamma g)^2 \Phi^2  (n_x \cdot v) \ge 0.
$$
Arguing exactly as in the proof of Proposition~\ref{prop:TheLqEstim} and using the estimates from Step~2 of the proof of Lemma~\ref{lem:EstimKFP*Lq12}, we obtain 
\bean
&&
\int g^2  m^2  \left[ \frac{(n_x \cdot  v)^{2} }{ \langle v \rangle  \delta^{1/2}} + \langle v \rangle^\varsigma \right]  \varphi^2
+\int  |\nabla_v(g  m)|^2  \varphi^2
\lesssim
\int g^2 m^2 [|\partial_t \varphi^2| +  \varphi^2 ].
\eean
Proceeding next as for Corollary~\ref{cor:EstimBord} with the help of the interpolation  Lemma~\ref{lem:EstimL2poids}, we deduce 
\beqn\label{eq:Estim-g1}
\int g^2 \Bigl( \frac{1 }{ \delta^\beta} +  \langle v \rangle^\varsigma \Bigr)    m^2 \varphi^2 
+ \int  |\nabla_v(g  m)|^2  \varphi^2
\lesssim   \int g^2 m^2 [|\partial_t \varphi^2| +  \varphi^2 ], 
\eeqn
for some  $\beta > 0  $.

\smallskip\noindent
{\sl Step 2. Downgraded weighted $L^2-L^p$ estimate, $p>2$.} 
For $0 \le \varphi \in \DD(0,T)$, $0 \le \chi \in \DD(\Omega)$ and the weight function $m_0 = m \la v \ra^{-\nu}$ with $\nu >0$ to be chosen later, the function $\bar g := g \varphi \chi m_0$  satisfies 
$$
- \partial_t \bar g - v \cdot \nabla_x \bar g = G
$$
where $G = G_0 + \Div_v G_1$ with 
$$
G_0 : = - g m_0 (\varphi' \chi + \varphi \nabla_x \chi \cdot v) - \varphi \chi (m_0 b - \nabla_v m_0) \cdot \nabla_v g + (c- \Div_v b)\bar g 
$$
and
$$
G_1 := \varphi \chi m_0 \nabla_v g.
$$
Proceeding as in the proof of Proposition~\ref{prop:EstimLploc}, we get first 
\bean
\lVert \bar g \rVert^2_{L^{p}(\R^{2d+1})}  
\lesssim   \| \varphi' m   g \|_{L^2(\UU)}^2 + \| \nabla \chi \|_{L^\infty}^2 \| \varphi m g  \|_{L^2(\UU)}^2 +  \| \varphi mg  \|_{L^2(\UU)}^2.
\eean
for some $p \in (2,\infty)$, and next, by interpolation, we conclude with 
\beqn\label{eq:Estim-g2}
 \left\lVert  g \varphi    \frac{m }{ \langle v \rangle^\nu} \delta^{\alpha/p} \right\rVert_{L^p(\UU)} 
\lesssim   \| (\varphi + |\varphi' |) g m \|_{L^2(\UU)}
\eeqn
for some  $\alpha > p > 2$ and $\nu \ge 0$.

\smallskip\noindent
{\sl Step 3. Up to the boundary $L^2-L^r$ estimate, $r > 2$.} Proceeding as during the proof of  Proposition~\ref{prop:EstimL1Lr-omega}, we may use the estimates \eqref{eq:Estim-g1} and \eqref{eq:Estim-g2} together with the H\"older inequality \eqref{eq:HolderIneq} in order to obtain that there exists an exponent $r > 2$ such that any solution $g$ to the dual problem \eqref{eq:dualKFP} satisfies 
 \bean
 \left\| g \varphi  m \right\|_{L^r(\UU)} 
\lesssim 
 \| (\varphi + |\varphi'|) gm \|_{L^2(\UU)}. 
\eean
for any $\varphi \in C^1_c((0,T))$, $T > 0$. We conclude to \eqref{eq:EstimDualL1LrStrong} by using the H\"older inequality once more, exactly as during the proof of Proposition~\ref{prop:EstimL1Lr-omega}. 
\end{proof}

We now prove an estimate similar to Proposition \ref{prop:EstimL1Lr-m} for weakly confining admissible weights in the spirit of Proposition \ref{prop:EstimL1Lr-omegaWeak} 

\begin{prop}\label{prop:EstimL1Lr-mWeakly} 
Let $\omega$ be a weakly confining admissible weight function and define $m=\omega^{-1}$. Any solution $g$ to the  dual  backward problem \eqref{eq:dualKFP} satisfies 
\beqn\label{eq:EstimDualL1LrWeakly}
 \left\| (\varphi + \lvert \varphi'\rvert) g \varphi  m \right\|_{L^2(\UU)} 
\lesssim    \|   (\varphi + |\varphi'|)^{2 \tfrac{r-1}{r-2}} \varphi^{- \tfrac{r}{r-2}}g m \la v\ra^K   \|_{L^1(\UU)}, 
\eeqn
for any test function $0 \le \varphi \in \DD((0,T))$, any $T>0$ and $K=   2(3d+1)(2d+3)$.
\end{prop} 

\begin{proof}[Proof of Proposition~\ref{prop:EstimL1Lr-mWeakly}]
The proof follows the same steps as for the proof of Proposition~\ref{prop:EstimL1Lr-m}.
 
\smallskip\noindent
{\sl Step 1. A weighted $L^2$ estimate at the boundary.} 
Let $\omega$ be an admissible weight function and define $m = \omega^{-1}$.

We define the modified weight function  $\widetilde m$ by 
 $$
\widetilde m^2 :=m_A^2 \left( 1 +   \frac14 \,   \frac{ n_x \cdot v}{\la v \ra^4} - \frac1{4D^{1/2}}    \delta(x)^{1/2} \frac{n_x \cdot  v}{\la v \ra^2}   \right), 
$$
as done in the Step 1 of Proposition~\ref{prop:EstimL1Lr-m}.
Considering a solution $g$ to the  dual  backward problem \eqref{eq:dualKFP}, indeed multiplying the equation by $\Phi^2 := \varphi^2 \widetilde m^2$ with $\varphi \in \DD(0,T)$, and integrating in all the variables, we obtain
\bean
- \frac12\int_\Gamma (\gamma g)^2 \Phi^2  (n_x \cdot v) + \frac12 \int_\UU g^2 \TT \Phi^2
= \int_\UU g (\CCC^* g)   \Phi^2, 
\eean
with $\TT$ defined in \eqref{def:TT}.

Since $\widetilde m^2 :=m_A^2 ( 1 +   \frac14 \,   \frac{ n_x \cdot v}{\la v \ra^4} )$ on the boundary $\Gamma$, Step~1 of the proof of Proposition~\ref{prop:EstimL1Lr-m} implies that
$$
- \frac12\int_\Gamma (\gamma g)^2 \Phi^2  (n_x \cdot v) \ge 0.
$$
Arguing similarly as in the proof of Proposition~\ref{prop:TheLqEstim} and using the estimates from Step~2 of the proof of Lemma~\ref{lem:EstimKFP*Lq12} with the difference that, since $\omega$ is a weakly confining weight function, we will have $\lvert \varpi^{\CCC^*}_{\widetilde m, 2}\rvert \lesssim \lvert \varpi^{\CCC^*}_{m, 2}\rvert <\infty$. Then we obtain 
\bean
&&
\int g^2  m^2   \frac{(n_x \cdot  v)^{2} }{ \langle v \rangle  \delta^{1/2}}  \varphi^2
+\int  |\nabla_v(g  m)|^2  \varphi^2
\lesssim
\int g^2 m^2 [|\partial_t \varphi^2| +  \varphi^2 ].
\eean
Proceeding next as during Corollary~\ref{cor:EstimBordWeak} we deduce 
\beqn\label{eq:Estim-g1Weakly}
\int  \frac{g^2  }{ \delta^\beta}    m^2 \varphi^2 
+ \int  |\nabla_v(g  m)|^2  \varphi^2
\lesssim   \int g^2 m^2 [|\partial_t \varphi^2| +  \varphi^2 ], 
\eeqn
for some  $\beta =[2(d+1)]^{-1} > 0  $.

  \smallskip\noindent
{\sl Step 2. Downgraded weighted $L^2-L^p$ estimate, $p>2$.} 
We remark that the computations from Proposition \ref{prop:EstimL1Lr-m} hold for any admissible weight so proceeding similarly we get that there is some $p \in (2,\infty)$ such that 
\beqn\label{eq:Estim-g2Weakly}
 \left\lVert  g \varphi    \frac{m }{ \langle v \rangle^\nu} \delta^{\alpha/p} \right\rVert_{L^p(\UU)} 
\lesssim   \lVert (\varphi + |\varphi' |) g m \rVert_{L^2(\UU)}
\eeqn
for any $0 \le \varphi \in \DD(0,T)$ and some $\alpha > p > 2$ and $\nu \ge 0$.

\smallskip\noindent
{\sl Step 3. Up to the boundary $L^2-L^r$ estimate, $r > 2$.} Proceeding as during the proof of  Proposition~\ref{prop:EstimL1Lr-omegaWeak}, we may use the estimates \eqref{eq:Estim-g1Weakly} and \eqref{eq:Estim-g2Weakly} together with the H\"older inequality \eqref{eq:HolderIneqWeakly} in order to obtain that there exists an exponent $r > 2$ such that any solution $g$ to the dual problem \eqref{eq:dualKFP} satisfies 
 \bean
 \left\| \varphi g  m \la v\ra^{-\mu} \right\|_{L^r(\UU)} 
\lesssim 
 \| (\varphi + |\varphi'|) gm \|_{L^2(\UU)}. 
\eean
for any $\varphi \in C^1_c((0,T))$, $T > 0$ and some $\mu>0$. We conclude to \eqref{eq:EstimDualL1LrWeakly} by using the H\"older inequality once more, exactly as in \eqref{eq:InterpolationIneqWeakly} during the proof of Proposition~\ref{prop:EstimL1Lr-omegaWeak}. 
\end{proof}

\subsection{Conclusion of the proof of the ultracontractivity property} 

We now conclude the proof of Theorem~\ref{theo:ultra} in several elementary and classical (after Nash's work) steps.

\begin{proof}[Proof of Theorem~\ref{theo:ultra}]
Consider a strongly confining admissible weight function $\omega$ and denote by $m=\omega^{-1}$ its inverse.
We split the proof into four steps.

 \medskip\noindent
 {\sl Step 1.} We first establish that there exist a constant $\eta > 0$ such that 
 any solution $f$ to the KFP equation \eqref{eq:KFP}--\eqref{eq:KolmoBdyCond}--\eqref{eq:initialDatum} satisfies 
\beqn\label{eq:EstimL1L2-omega}
 \| f(T, \cdot) \|_{L^2_\omega(\OO)} \le C_{12} T^{-\eta}   \| f_0  \|_{L^1_\omega(\OO)}, \quad \forall \, T \in (0,1) .
\eeqn
First indeed, from Proposition~\ref{prop:EstimL1Lr-omega}, there exist an exponent $r > 2$ such that  
$$
 \left\| f \varphi  \omega \right\|_{L^r(\UU)} 
\lesssim    \left\lVert   (\varphi + |\varphi'|)^{2 \tfrac{r-1}{r-2}} \varphi^{- \tfrac{r}{r-2}} f \omega   \right\rVert_{L^1(\UU)}.
$$
Thanks to the estimate  \eqref{eq:lem:EstimKFPLp12} from Lemma~\ref{lem:EstimKFPLp12} for $p=1$ and $p=r$ provided by Theorem~\ref{theo:WP-KFP&KFP*}, we have 
\bean
\|  \varphi\|_{L^r(0,T)}   \left\| f_T   \omega \right\|_{L^r(\OO)} 
 &\lesssim&
e^{\kappa T} \left\| f \varphi  \omega \right\|_{L^r(\UU)} 
\eean
and
\bean
  \left\lVert   (\varphi + |\varphi'|)^{2 \tfrac{r-1}{r-2}} \varphi^{- \tfrac{r}{r-2}} f \omega   \right\rVert_{L^1(\UU)}
 &\lesssim&
e^{\kappa T}    \left\lVert   (\varphi + |\varphi'|)^{2 \tfrac{r-1}{r-2}} \varphi^{- \tfrac{r}{r-2}}    \right\rVert_{L^1(0,T)}   \|     f_0 \omega   \|_{L^1(\OO)}.
\eean
We choose $\varphi(t) := \psi(t/T)$ with $\psi \in C^1_c((0,1))$ such that  $0 \le \psi \le \mathbf{1}_{[1/4,3/4]}$, $\psi \not\equiv 0$ and $ |\psi'|^{ 2\frac{r-1}{r-2}} \psi^{- \frac{r}{r-2}}  \in L^1(0,1)$, which is possible by taking
$\psi(s) := (s-1/4)_+^n (3/4-s)_+^n$ for $n > 0$ large enough, and we easily compute 
\bean
&&\|  \varphi\|_{L^r(0,T)} = T^{1/r} \| \psi \|_{L^r(0,1)}, \quad 
\|  \varphi\|_{L^1(0,T)}  =  T \| \psi \|_{L^1(0,1)}, \quad 
\\
&& \left\lVert   |\varphi'|^{2 \tfrac{r-1}{r-2}} \varphi^{- \tfrac{r}{r-2}}    \right\rVert_{L^1(0,T)} 
= T^{1 - {2 \tfrac{r-1}{r-2}}}   \left \lVert   \, |\psi'|^{2 \tfrac{r-1}{r-2}} \psi^{- \tfrac{r}{r-2}}    \right\rVert_{L^1(0,1)} .
\eean
Gathering the three last estimates and the three last identities, we finally obtain 
$$
T^{1/r}  \| f_T  \omega \| _{L^r(\OO)} 
\lesssim \left( T +  T^{1 - {2 \tfrac{r-1}{r-2}}}  \right) \|     f_0 \omega   \|_{L^1(\OO)},
$$
which implies
\bean
 \| f_T  \omega \| _{L^r(\OO)} 
&\lesssim& T^{-\frac{1}{ r} - 1 - \frac{2}{ r-2}} \left( T^{1+ \frac{2}{ r-2}} +1  \right) \|     f_0 \omega   \|_{L^1(\OO)}\\
&\lesssim& T^{-\frac{1}{ r} - 1 - \frac{2}{ r-2}}  \|     f_0 \omega   \|_{L^1(\OO)}.
\eean
Then by using again Lemma~\ref{lem:EstimKFPLp12} for $p=1$ and an interpolation argument as before, choosing $\theta:= \frac{r}{ 2(r-1)}>0$ such that $1/2 = 1- \theta + \theta/r$, we deduce
\bean
\lVert f_T\omega\rVert_{L^2(\OO)} &\leq & \lVert f_T \omega\rVert_{L^r(\OO)}^\theta \lVert f_T\omega\rVert_{L^1(\OO)}^{1-\theta} \\
&\lesssim& T^{\theta(-\frac{1}{ r} - 1 - \frac{2}{ r-2})} \lVert f_0\omega \rVert_{L^1(\OO)}
\eean
from what we immediately conclude to \eqref{eq:EstimL1L2-omega} with $\eta :=   \frac{r}{ 2(r-1)} \left( \tfrac1r + 1 + \tfrac{2}{r-2} \right)$. 

\smallskip\noindent
{\sl Step 2.}  Arguing in a similar fashion as above but using the estimates from Lemma~\ref{lem:EstimKFP*Lq12} and Proposition~\ref{prop:EstimL1Lr-m} (instead of Lemma~\ref{lem:EstimKFPLp12} and Proposition~\ref{prop:EstimL1Lr-omega}), we deduce a  similar result for the dual problem. More precisely, there exists $\eta' > 0$ such that any solution $g$ to the  dual  backward problem \eqref{eq:dualKFP} satisfies 
\beqn\label{eq:EstimDualL1L2-omega}
 \| g(0, \cdot) \|_{L^2_m} \le C^*_{12} T^{-\eta'}   \| g_T  \|_{L^1_m}, \quad \forall \, T \in (0,1).
\eeqn

 \medskip\noindent{\sl Step 3.}  
As the dual counterpart of \eqref{eq:EstimDualL1L2-omega}, we have that any solution $f$ to the KFP equation \eqref{eq:KFP}--\eqref{eq:KolmoBdyCond}--\eqref{eq:initialDatum} satisfies 
\beqn\label{eq:EstimL2Linfty-omega}
 \| f(T, \cdot) \|_{L^\infty_\omega} \le C^*_{12}  T^{-\eta'}   \| f_0  \|_{L^2_\omega}, \quad \forall \, T \in (0,1).
\eeqn
Indeed we may argue by duality in the following way:  For $f_0 \in L^2_\omega$ and $g_T \in L^1_m$ we denote respectively by $f$ and $g$ the solution to the primal~\eqref{eq:KFP}--\eqref{eq:KolmoBdyCond}--\eqref{eq:initialDatum} and dual~\eqref{eq:dualKFP} KFP problem, then we have
 \bean
 \| f(T,\cdot) \|_{L^\infty_\omega} 
 &=& \sup_{g_T \in L^1_m, \| g_T \|_{L^1_m} \le 1} \int f(T,\cdot)  g_T
 \\
 &=& \sup_{g_T \in L^1_{m}, \| g_T \|_{L^1_m} \le 1} \int f_0 g(0,\cdot)
 \\
&\le& \| f_0 \|_{L^2_\omega} \sup_{g_T \in L^1_{m}, \| g_T \|_{L^1_m} \le 1}  \| g(0,\cdot)  \|_{L^2_m} 
 \\
&\le& \| f_0 \|_{L^2_\omega} 
C_{12}^*T^{-\eta'}, 
\eean
where we have used an usual representation formula in the first line, the duality formula \eqref{eq:identite-dualite} in the second line, the Cauchy-Schwarz inequality in the third line and estimate \eqref{eq:EstimDualL1L2-omega} in the last line.

 \medskip\noindent{\sl Step 4.} For $T \in (0,1]$, the estimates \eqref{eq:EstimL1L2-omega} and \eqref{eq:EstimL2Linfty-omega} also write 
$$
 \| f(T, \cdot) \|_{L^\infty_\omega} \le C_{12} ( T/2)^{-\eta'}   \| f(T/2,\cdot)  \|_{L^2_\omega}, \quad  
 \| f(T/2,\cdot)  \|_{L^2_\omega} \le C_{12}^* ( T/2)^{-\eta}   \| f_0 \|_{L^1_\omega},
 $$
 so that 
\beqn\label{eq:estimL1Linfty}
 \| f(T, \cdot) \|_{L^\infty_\omega} \le C_{12} C_{12}^* 2^{\eta+\eta'} T ^{-\eta- \eta'}      \| f_0 \|_{L^1_\omega}. 
\eeqn
which is nothing but \eqref{eq:theo-ultra} for $T \in (0,1]$ with $\Theta := \eta + \eta'$, $\kappa = (\alpha+\alpha')/2$ and $C \ge C_{12} C_{12}^* 2^{\eta+\eta'}$. For $T >1$ we write $f(T) = S_\LLL(1) f(T-1)$, then we compute
$$
\begin{aligned}
\| f(T,\cdot) \|_{L^\infty_\omega}
&= \| S_\LLL(1) f(T-1,\cdot) \|_{L^\infty_\omega} \\
&\lesssim \|  f(T-1,\cdot) \|_{L^1_\omega} \\
&\lesssim e^{\kappa (T-1)} \|  f_0 \|_{L^1_\omega} ,
\end{aligned}
$$ 
where we have used \eqref{eq:estimL1Linfty} at the second line and Lemma~\ref{lem:EstimKFPLp12} at the third one, which concludes the proof of \eqref{eq:theo-ultra} for $T >1$.  
\end{proof}

\smallskip 
We end this section by formulating a variant of Theorem~\ref{theo:ultra} for weakly   confining weight functions.

\begin{prop}\label{prop:EstimL1L2-omegaWeak-final} 
Let $\omega$ be a weakly confining admissible weight function such that $s > 0$ or $s=0$ and $k > K + k^*$ for $K :=  4(3d+1)(2d+3)$. Define $\omega_\infty := \omega^{1/2}$ if $s > 0$, $\omega_\infty := \omega \langle v \rangle^{-K}$ if $s= 0$. 
For any $T > 0$, there exists $\kappa, \eta > 0$ such that  any solution $f$ to the KFP equation \eqref{eq:KFP}--\eqref{eq:KolmoBdyCond} satisfies 
\beqn\label{eq:EstimL1LrWeak_bis}
 \|  f(T,\cdot)    \|_{L^\infty_{\omega_\infty}(\OO)}   
\lesssim e^{\kappa T} T^{-\eta}  \| f_0    \|_{L^1_{\omega}(\OO)}   
\eeqn
\end{prop}

\begin{proof}[Proof of Proposition~\ref{prop:EstimL1L2-omegaWeak-final}] 
We adapt the first step of the proof of Theorem~\ref{theo:ultra} by using the estimate established in Proposition~\ref{prop:EstimL1Lr-omegaWeak} instead of  Proposition~\ref{prop:EstimL1Lr-omega} as well as the estimate  \eqref{eq:lem:EstimKFPLp12} from Lemma~\ref{lem:EstimKFPLp12} for $p=1$ and $p=2$ provided by Theorem~\ref{theo:WP-KFP&KFP*}.

\smallskip\noindent
\emph{Step 1.} We set $  \omega_2 = \omega^{3/4}$ if $s>0$, $ \omega_2 = \omega \la v\ra^{-K/2}$ if $s=0$ and we prove first that 
\be \label{eq:EstimL1LrStrongTle1}
 \|  f(T,\cdot)    \|_{L^2_{\omega_2}(\OO)}   
\lesssim  T^{-\eta}  \| f_0    \|_{L^1_{ \omega}(\OO)} \quad  \forall T\in (0,T].
\ee
Indeed from Proposition~\ref{prop:EstimL1Lr-omegaWeak} and the definitions of $\omega_\infty$ and $ \omega_2$ we have that there is $r>2$, from Proposition \ref{prop:EstimL1Lr-omegaWeak}, such that
\bear
 \|  (\varphi + |\varphi'|) f \omega_2   \|_{L^2(\UU)} 
&\lesssim&    \left\lVert   (\varphi + |\varphi'|)^{2 \tfrac{r-1}{r-2}} \varphi^{- \tfrac{r}{r-2}} f \omega_2  \la v\ra^{K/2} \right\rVert_{L^1(\UU)}, \nonumber\\
&\lesssim&  \left\lVert   (\varphi + |\varphi'|)^{2 \tfrac{r-1}{r-2}} \varphi^{- \tfrac{r}{r-2}} f  \omega \right\rVert_{L^1(\UU)}  
\eear
for any $0<\varphi \in C^1_c((0,T))$ and any $T > 0$. 

From Lemma~\ref{lem:EstimKFPLp12} for $p=1$ and $p=2$ we have 
\be
\lVert \varphi + \lvert \varphi'\rvert \rVert_{L^2(0,T)} \lVert f_T\omega_2\rVert_{L^2(\OO)} \leq e^{\kappa T} \lVert (\varphi + \lvert \varphi'\rvert)  f\omega_2\rVert_{L^2(\UU)}
\ee
and 
\be
 \left\lVert   (\varphi + |\varphi'|)^{2 \tfrac{r-1}{r-2}} \varphi^{- \tfrac{r}{r-2}} f  \omega \right\rVert_{L^1(\UU)}  \lesssim  e^{\kappa T} \lVert (\varphi + \lvert \varphi'\rvert)^{2 \frac{r-1}{ r-2}} \varphi^{- \frac{r}{r-2}} \rVert_{L^1(0,T)} \lVert f_0  \omega\rVert_{L^1(\OO)}
\ee

As done during the proof of Theorem~\ref{theo:ultra} we choose $\varphi(t) := \psi(t/T)$ with $\psi \in C^1_c((0,1))$ such that  $0 \le \psi \le \mathbf{1}_{[1/4,3/4]}$, $\psi \not\equiv 0$ and $ |\psi'|^{ 2\frac{r-1}{r-2}} \psi^{- \frac{r}{r-2}}  \in L^1(0,1)$, which is possible by taking
$\psi(s) := (s-1/4)_+^n (3/4-s)_+^n$ for $n > 0$ large enough, and we easily compute 
\bean
&&\|  \varphi\|_{L^2(0,T)} \geq T^{1/2} \| \psi + \lvert \psi'\rvert  \|_{L^2(0,1)}, \quad 
\|  \varphi\|_{L^1(0,T)}  =  T \| \psi \|_{L^1(0,1)}, \quad 
\\
&& \left\lVert   |\varphi'|^{2 \tfrac{r-1}{r-2}} \varphi^{- \tfrac{r}{r-2}}    \right\rVert_{L^1(0,T)} 
= T^{1 - {2 \tfrac{r-1}{r-2}}}   \left \lVert   \, |\psi'|^{2 \tfrac{r-1}{r-2}} \psi^{- \tfrac{r}{r-2}}    \right\rVert_{L^1(0,1)} .
\eean
Then we deduce 
\be
 \lVert f_T\omega_2\rVert_{L^2(\OO)} \lesssim T^{-\frac{1}{ 2} - \frac{r}{ r-2}} \left( 1 + T^{1+ \frac{r}{r-2}} \right) \lVert f_0\overline \omega\rVert_{L^1(\OO)}
\ee
which is nothing but \eqref{eq:EstimL1LrStrongTle1} for $T\in (0,1]$ with $\eta = \frac{1}{ 2} + \frac{r}{ r-2}$.

\smallskip\noindent
\emph{Step 2.} Now we prove a similar result for the dual problem, for this we define $m_\infty=\omega_\infty^{-1}$ and $ m_2 =  \omega_2^{-1}$. 
Indeed by using Lemma \ref{lem:EstimKFP*Lq12} and Proposition~\ref{prop:EstimL1Lr-mWeakly} we proof, by arguing similarly as in the Step 1, that there exists $ \eta' > 0$ such that any solution $g$ to the  dual  backward problem \eqref{eq:dualKFP} satisfies 
\beqn\label{eq:EstimDualL1L2-mWeakly}
 \| g(0, \cdot) \|_{L^2_{ m_2}} \le C^*_{12}  T^{-\eta'}   \| g_T  \|_{L^1_{m_\infty}}, \quad \forall \, T \in (0,1).
\eeqn

 \medskip\noindent{\sl Step 3.}  
As the dual counterpart of \eqref{eq:EstimDualL1L2-mWeakly}, we have that any solution $f$ to the KFP equation \eqref{eq:KFP}--\eqref{eq:KolmoBdyCond}--\eqref{eq:initialDatum} satisfies 
\beqn\label{eq:EstimL2Linfty-omegaWeakly}
 \| f(T, \cdot) \|_{L^\infty_{ \omega_\infty}} \le C^*_{12} T^{-\eta'}   \| f_0  \|_{L^2_{\omega_2}}, \quad \forall \, T \in (0,1).
\eeqn
Indeed we may argue by duality as in the Step 3 of the proof of Theorem~\ref{theo:ultra}:  For $f_0 \in L^2_{\omega_2}$ and $g_T \in L^1_{ m_\infty}$ we denote respectively by $f$ and $g$ the solution to the primal~\eqref{eq:KFP}--\eqref{eq:KolmoBdyCond}--\eqref{eq:initialDatum} and dual~\eqref{eq:dualKFP} KFP problem, then we have
 \bean
 \| f(T,\cdot) \|_{L^\infty_{\omega_\infty}} 
 &=& \sup_{g_T \in L^1_{m_\infty}, \| g_T \|_{L^1_{m_\infty}} \le 1} \int f(T,\cdot)  g_T
 \\
 &=& \sup_{g_T \in L^1_{m_\infty}, \| g_T \|_{L^1_{m_\infty}} \le 1} \int f_0 g(0,\cdot)
 \\
&\le& \| f_0 \|_{L^2_{ \omega_2}} \sup_{g_T \in L^1_{m}, \| g_T \|_{L^1_m} \le 1}  \| g(0,\cdot)  \|_{L^2_{ m_2}} 
 \\
&\le& \| f_0 \|_{L^2_{ \omega_2}} 
C_{12}^* T^{-\eta'},
\eean
where we have used an usual representation formula in the first line, the duality formula \eqref{eq:identite-dualite} in the second line, the Cauchy-Schwarz inequality in the third line and estimate \eqref{eq:EstimDualL1L2-mWeakly}
in the last line. 
  
 \medskip\noindent{\sl Step 4.} For $T \in (0,1]$, the estimates \eqref{eq:EstimL1L2-omega} and \eqref{eq:EstimL2Linfty-omega} also write 
$$
 \| f(T, \cdot) \|_{L^\infty_{\omega_\infty}} \le C_{12} ( T/2)^{-\eta'}   \| f(T/2,\cdot)  \|_{L^2_{ \omega_2}}, \quad  
 \| f(T/2,\cdot)  \|_{L^2_{ \omega_2}} \le C_{12}^* ( T/2)^{-\eta}   \| f_0 \|_{L^1_\omega},
 $$
 so that 
\beqn\label{eq:estimL1Linfty-bis}
 \| f(T, \cdot) \|_{L^\infty_{\omega_\infty}} \le C_{12} C_{12}^* 2^{\eta+\eta'} T ^{-\eta- \eta'}      \| f_0 \|_{L^1_\omega}. 
\eeqn
which is nothing but \eqref{eq:theo-ultra} for $T \in (0,1]$ with $\Theta := \eta + \eta'$ and $C \ge C_{12} C_{12}^* 2^{\eta+\eta'}$. For $T >1$ we write $f(T) = S_\LLL(1) f(T-1)$, then we compute
$$
\begin{aligned}
\| f(T,\cdot) \|_{L^\infty_\omega}
&= \| S_\LLL(1) f(T-1,\cdot) \|_{L^\infty_\omega} \\
&\lesssim \|  f(T-1,\cdot) \|_{L^1_\omega} \\
&\lesssim e^{\kappa (T-1)} \|  f_0 \|_{L^1_\omega} ,
\end{aligned}
$$ 
where we have used \eqref{eq:estimL1Linfty-bis} at the second line and Lemma~\ref{lem:EstimKFPLp12} at the third one, which concludes the proof of \eqref{eq:EstimL1LrWeak_bis} for $T >1$.  
\end{proof}

\section{Well-posedness in a general framework}\label{sec-WellposednessGal}

In this section, we establish the well-posedness of the KFP equation in a general weighted Lebesgue space  framework and in a weighted Radon measures space framework. We deduce the existence and uniqueness of a family of fundamental solutions.
These results are based on new a priori estimates that we establish with formal computations, in order not to complicate the presentation. These estimates can be justified by coming back to the standard arguments used in the proof of Theorem~\ref{theo:WP-KFP&KFP*}.

\subsection{Additional a priori estimates in the $L^1$ framework.}
We recall that any solution $f$ to the KFP equation \eqref{eq:KFP} satisfies
\bear\label{eq:Galframe-AprioriEstim1}
&&\| f_t \|_{L^p_\omega(\OO)}  \lesssim \| f_0 \|_{L^p_\omega(\OO)}, \quad \forall \, t \in [0,T],
\\ \label{eq:Galframe-AprioriEstim1BIS}
&&\| f_t \|_{L^1_{\omega\langle v \rangle^\varsigma}(\UU)}  \lesssim \| f_0 \|_{L^1_\omega(\OO)}, 
\\ \label{eq:Galframe-AprioriEstim2}
&& \| \nabla_v f \|_{L^2_\omega(\UU)} \lesssim \| f_0 \|_{L^2_\omega(\OO)},
\\ \label{eq:Galframe-AprioriEstim3}
&&  \|  f_t \|_{L^r_{\omega_r}(\OO)}   \le C_t\| f_0 \|_{L^p_\omega(\OO)}, \quad \forall \, t \in (0,T], 
\eear
for any admissible weight function $\omega$, any exponent $1 \le p \le r \le \infty$ and some admissible weight function $\omega_r$
from \eqref{eq:lem:EstimKFPLp12}, from \eqref{eq:disssipSL-Lp} and \eqref{eq:varpiL-varpisharp}, from  \eqref{eq:bddL2&HHH} and from Theorem~\ref{theo:ultra} with $\omega_r = \omega$ in the case of a strongly confining weight function $\omega$ or from Proposition~\ref{prop:EstimL1L2-omegaWeak-final} and a standard interpolation argument in the case of a weakly confining weight function $\omega$. 
The two last estimates together immediately give 
\beqn\label{eq:Galframe-AprioriEstim4}
 \| \nabla_v f \|_{L^2((t_0,T) \times \OO)}  \le C_{t_0,T} \| f_0 \|_{L^p_\omega(\OO)}, \quad \forall \, T > t_0 > 0.
\eeqn

\smallskip
We will use an additional a priori estimate that we establish now. For further references, for $k > 0$, we define the functions $T_k$ by 
$$
T_k (s) := \max (\min(s,k),-k).
$$

\begin{lem}\label{lem:nablaTKf}
Consider an admissible weight function $\omega$ and a solution $f$  to the KFP equation  \eqref{eq:KFP}--\eqref{eq:KolmoBdyCond}--\eqref{eq:initialDatum} associated to an initial datum $0 \le f_0\in L^1_{\omega}$. There hold, at least formally,
  \beqn\label{eq:Galframe-AprioriEstim5}
 \| \nabla_v T_K(f/\MMM) \MMM^{1/2} \|_{L^2(\UU)} \le C_{T,K} \|f_0 \|_{L^1_\omega}
\eeqn
and
\beqn\label{eq:Galframe-AprioriEstim6}
 \sup_{t \in [0,T]} \| f \mathbf{1}_{f \ge (k+1) \MMM} \|_{L^1(\OO)} \le \| f_0 \mathbf{1}_{f_0 \ge  k \MMM} \|_{L^1(\OO)} 
 + C_T \|f_0 \|_{L^1_\omega}, 
\eeqn
for any $T,K,k > 0$ and some constants $C_{T,K}$ and $C_T$. 

 \end{lem}

\begin{proof}[Proof of Lemma~\ref{lem:nablaTKf}] 
For a renormalizing function  $\beta$,  
a positive weight function $m$ and a nonnegative test function $\varphi$, we compute 
\bean
&&\frac{d}{dt} \int_\OO \beta(\frac{f}{m})\varphi 
=   \int \beta'(\frac{f}{m}) \frac{\varphi }{ m}  \bigl\{ - v \cdot \nabla_x f +  \Delta_v f+ b\cdot \grad_v f+ cf \bigr\}
\\
&&  
=     \int_\Sigma (- v \cdot n_x)  \beta(\frac{f}{m}) \varphi  - \int \beta''(\frac{f}{m}) |\nabla_v (\frac{f}{m})|^2 \varphi   
\\
&&\quad   
+   \int  \bigl\{    (v \cdot \nabla_x \varphi)   \beta(\frac{f}{m}) -  (v \cdot \nabla_x m)  \frac{f}{m} \beta'(\frac{f}{m})   \frac{\varphi }{ m} \bigr\}
+  \int  \frac{f}{m} \beta'(\frac{f}{m}) c \varphi 
\\
&&\quad  +  \int \bigl\{ \alpha(\frac{f}{m}) \Div_v (  \frac{\varphi }{ m}   \nabla_v m  ) 
+  \beta(\frac{f}{m}) \Div_v (  m   \nabla_v\frac{\varphi }{ m} ) 
-   \beta'(\frac{f}{m}) \frac{f}{m}  \nabla_v\frac{\varphi }{ m}  \cdot  \nabla_v m    \bigr\}
\\
&&  +  \int  \bigl\{  \beta(\frac{f}{m}) \varphi  [- \Div_v b  ]     - (b\cdot \nabla_v \varphi)   \beta(\frac{f}{m}) +  (b \cdot \nabla_v m)  \frac{f}{m} \beta'(\frac{f}{m})   \frac{\varphi }{ m} \bigr\}, 
\eean
with $\alpha'(s) = s \beta''(s)$. Assuming $\beta : \R \to \R_+$ even and convex such that  $\beta(0) = \beta'(0) = 0$, the 
boundary term can be handled in the following way 
\bean
&&\int_\Sigma (- v \cdot n_x)  \beta(\frac{\gamma f }{ m}) \varphi 
= \int_{\Sigma_-} (v \cdot n_x)_-  \beta(\frac{\RRR \gamma_+ f }{ m}) \varphi 
- \int_{\Sigma_+}(v \cdot n_x)_+ \beta (\frac{ \gamma_+ f }{ m}) \varphi 
\\
&&\quad \le \int_{\Sigma_-} (v \cdot n_x)_-   \Bigl\{ \iota_D \beta(\frac{\DDD \gamma_+ f }{ m})+ \iota_S  \beta(\frac{ \SSS \gamma_+ f }{ m} ) \Bigr\} \varphi 
-  \int_{\Sigma_+} (v \cdot n_x)_+ \beta (\frac{ \gamma_+ f }{ m}) \varphi 
\\
&&\quad \le \int_{\Sigma_+} (v \cdot n_x)_+   \iota_D \Bigl\{ \beta( \MMM \frac{  {\widetilde{\gamma_+ f}} }{ m \circ \VV_x} ) \varphi \circ \VV_x 
-   \beta (\frac{ \gamma_+ f }{ m}) \varphi \Bigr\} \\
\\
&&\qquad + \int_{\Sigma_+} (v \cdot n_x)_+   \iota_S \Bigl\{ \beta(   \frac{  {\gamma_+ f} }{ m \circ \VV_x} ) \varphi \circ \VV_x 
-   \beta (\frac{ \gamma_+ f }{ m}) \varphi \Bigr\}, 
\eean
where we have used the convexity of $\beta$ in the second line and the change of variable $v \mapsto \VV_x v$ on the last equality. 
Taking  $m =  \varphi : = \MMM$,  we get 
\bean
\int_\Sigma (- v \cdot n_x)  \beta(\frac{\gamma f }{ \MMM}) \MMM
\le \int_{\Sigma_+} (v \cdot n_x)_+   \iota_D \Bigl\{ \beta(  {   \widetilde{\gamma_+ f}}  )  
-   \beta ( \frac{ \gamma_+ f }{ \MMM})\Bigr\}  \MMM  \le 0, 
 \eean
 where we have classically used the very definition of $ \widetilde{\gamma_+ f}$ and the Jensen inequality in order to get the last inequality. 
 With these choices of functions $\beta$, $m$ and $\varphi$, the first identity simplifies 
 \bean
&&\frac{d}{dt} \int_\OO \beta(\frac{f }{ \MMM})\MMM
\\
&&  
=     \int_\Sigma (- v \cdot n_x)  \beta(\frac{f }{ \MMM}) \MMM  +  \int  [v \cdot \nabla_x \MMM] \bigl( \beta(\frac{f }{ \MMM}) -  \frac{f }{ \MMM}  \beta'(\frac{f }{ \MMM}) \bigr)
 \\
&& - \int \beta''(\frac{f }{ \MMM}) |\nabla_v (\frac{f }{ \MMM})|^2 \MMM    +  \int  \alpha(\frac{f }{ \MMM}) \Delta_v   \MMM
 \\
&&  +  \int  \bigl\{  \beta(\frac{f }{ \MMM}) \MMM  [- \Div_v b  ]     - (b\cdot \nabla_v \MMM)   \beta(\frac{f }{ \MMM}) +  (b \cdot \nabla_v \MMM+ c \MMM  )  \frac{f }{ \MMM} \beta'(\frac{f }{ \MMM})   \bigr\}.
\eean
We finally particularize $\beta'' = \mathbf{1}_{[k,k+1]}$, $k \ge 0$, so that 
$$
0 \le \beta(s), \, s \beta'(s), \, \alpha(s) \le s, \ |s \beta'(s) - \beta(s)| \le s \quad \forall \, s \in \R, \ k \ge 0. 
$$
Observing that 
$$
\frac{|v \cdot \nabla_x \MMM| }{ \MMM} + \frac{|b \cdot \nabla_v \MMM| }{ \MMM} + \frac{|\Delta_v \MMM| }{ \MMM} + |c| + |\Div_v b| \le C \langle v \rangle^3,
$$
we deduce with this choice of $\beta$ that 
\bean
 \frac{d}{dt} \int_\OO \beta(\frac{f }{ \MMM})\MMM +  \int \beta''(\frac{f }{ \MMM}) |\nabla_v (\frac{f }{ \MMM})|^2 \MMM  
\le C  \int |f|  \langle v \rangle^3. 
\eean
Because $\omega \gtrsim  \langle v \rangle^3$, we may use \eqref{lem:EstimKFPLp12} with $p=1$ in order to bound the above RHS term and then  both \eqref{eq:Galframe-AprioriEstim5} and \eqref{eq:Galframe-AprioriEstim6} follow. 
\end{proof}

\subsection{Well-posedness in a $L^p_\omega$ framework}

For further references, we note
 $$
 \LLLL f :=  \partial_t f + \LLL f, \quad 
 \LLLL^* \varphi :=  -\partial_t \varphi + \LLL^*\varphi,
 $$
 where $\LLL$ is defined in \eqref{eq:intro-defLLLf} and $\LLL^*$ is defined in \eqref{def:CCC*}.
%
We also denotes $\BBB_3$ the set of functions $\beta \in C^2(\R)$ such that $\beta'$ has compact support. 

\begin{theo}\label{theo:E-UinLp} 
We consider an admissible weight function $\omega$ and an exponent $p \in [1,\infty]$. 
 For any $ f_0 \in L^p_\omega(\OO)$, there exists a unique function $f \in C(\R_+;L^p_\omega(\OO))$ satisfying the estimates \eqref{eq:Galframe-AprioriEstim1}, \eqref{eq:Galframe-AprioriEstim3}, 
 \eqref{eq:Galframe-AprioriEstim4}, \eqref{eq:Galframe-AprioriEstim5}, \eqref{eq:Galframe-AprioriEstim6} and 
which is a renormalized solution to \eqref{eq:KFP}, that is 
\beqn\label{eq-paraLp_RenormalizationFormula}
\int_{\UU}  \{ \beta(f)   \LLLL^* \varphi  + \beta''(f) |\nabla_vf |^2 \varphi \} + \int_\Gamma   \beta(\gamma \, f) \, \varphi \,\,   n_x \cdot {v} \,  d{v} d\sigma_{\! x}dt 
= \int_{\OO} \beta(f_0) \varphi(0,\cdot) dvdx, 
\eeqn
for any $\varphi \in \DD(\bar\UU)$ and $\beta \in \BBB_3$.  
Furthermore, the one parameter family of mappings $S_\LL(t) : L^p_\omega \to L^p_\omega$, 
defined by $S_\LL(t) f_0 := f(t,\cdot)$ for $t \ge 0$ and $f_0 \in L^p_\omega$,  is a positive semigroup of linear and bounded operators. 
\end{theo}

It is worth emphasizing that  because $\beta \in \BBB_3$, we have $\supp \beta'' \subset [-K,K]$ for some $K > 0$ and thus 
\beqn\label{eq:paraLp-renorm&TK}
\beta''(g) \, |\nabla_v g |^2 = \beta''(g) \, \mathbf{1}_{|g| \le K} |\nabla_v g |^2 = \beta''(g) \, |\nabla_v T_K(g) |^2.
\eeqn
 Together with \eqref{eq:Galframe-AprioriEstim5}, that implies that the second term in \eqref{eq-paraLp_RenormalizationFormula} makes sense. 
 Also observe that $\gamma f \in L^2_\omega((t_0,T) \times \Sigma,d\xi_2)$ for any $0< t_0 < T$, thanks to the trace Theorem~\ref{theo:Kolmogorov-trace}, so that $\beta(\gamma f) \in L^\infty(\Gamma)$
and the boundary term makes sense.
A similar result holds for the dual KFP equation \eqref{eq:dualKFP}. 
 
 \begin{proof}[Proof of Theorem~\ref{theo:E-UinLp}]
 We split the proof into two steps. 
  
 \smallskip\noindent
 {\sl Step 1. Existence part.} Because of the linearity of the equation and the weak maximum principle, we may only consider nonnegative initial data (and solutions). 
 We first assume $1 \le p < 2$. We take $0 \le f_{0,n} \in L^1_\omega \cap L^2_\omega$ such that $f_{n,0} \to f_0$ in $L^p_\omega$. 
 Thanks to Theorem~\ref{theo:WP-KFP&KFP*}, we may associate a sequence $(f_n)$ of solutions in $C_t(^2_\omega) \cap L^2_tH^1_\omega$ such that 
 $(f_n)$ satisfies uniformly in $n \ge 1$ the estimates \eqref{eq:Galframe-AprioriEstim1}, \eqref{eq:Galframe-AprioriEstim2}, \eqref{eq:Galframe-AprioriEstim3} and \eqref{eq:Galframe-AprioriEstim4}.
 Because the equation is linear the function $f_n - f_m$ satisfies  \eqref{eq:Galframe-AprioriEstim1}, namely 
$$
\sup_{t \in [0,T]} \| (f_n-f_m)(t) \|_{L^p_\omega}  \le \| f_{0,n} -f_{0,m} \|_{L^p_\omega}, 
$$
and therefore $(f_n)$ is a Cauchy sequence in $C([0,T];L^p_\omega)$. We  define $f := \lim f_n$. Similarly from \eqref{eq:Galframe-AprioriEstim4}, the sequence  $(f_n)$ is a Cauchy sequence in $L^2(t_0,T; H^1_\omega)$ for any $t_0 \in (0,T)$. 
We thus have 
$$
\int_{\UU_{t_0}} \beta(f) \LLLL \varphi + \int_{\UU_{t_0}} \beta''(f) |\nabla_v f|^2 \varphi + \int_{\Gamma_{t_0}} \beta(\gamma f) \varphi n_x \cdot v = \int_\OO \beta(f(t_0)) \varphi(t_0,\cdot), 
$$
for any $t_0 > 0$ and $\varphi \in \DD([0,T] \times \bar \OO)$, where $\UU_{t_0} := (t_0,T) \times \OO$ and $\Gamma_{t_0} := (t_0,T) \times \Sigma$.
Because of  \eqref{eq:Galframe-AprioriEstim1} and  \eqref{eq:Galframe-AprioriEstim5}, we may pass to the limit $t_0 \to 0$ and we get 
$$
\int_{\UU} \beta(f) \LLLL \varphi + \int_{\UU } \beta''(f) |\nabla_v f|^2 \varphi + \int_{\Gamma } \beta(\gamma f) \varphi n_x \cdot v = \int_\OO \beta(f_0) \varphi(0,\cdot)
$$
for any $\varphi \in \DD([0,T] \times \bar \OO)$.
The same holds for the dual KFP problem for $g_0 \in L^p_m$, $1 \le p \le 2$. By duality, we obtain the existence of a solution for any $1 \le p \le \infty$ for both problems.

 \smallskip\noindent
 {\sl Step 2. Uniqueness part.}
We consider two solutions $f_i$ to the KFP equation in the sense of Theorem~\ref{theo:E-UinLp} and we set $f := f_2 - f_1$.

 \smallskip 
$\bullet$ Take $g_T \in C_c(\OO)$ and let us consider $g \in L^\infty([0,T] \times \OO) \cap C([0,T) \times \OO)$ the solution associated to the backward dual problem \eqref{eq:dualKFP} which existence is given by Theorem~\ref{theo:WP-KFP&KFP*} and regularity is given by Theorem~\ref{theo:WP-KFP&KFP*} and \cite[Theorem 3]{MR3923847}. Because $f_t \in L^2(\OO)$ for $t > 0$, we may use the Theorem~\ref{theo:WP-KFP&KFP*} and thus write 
 $$
 \int_\OO f(T) g_T dxdv =  \int_\OO f(t) g(t) dxdv, \quad \forall \, t \in (0,T).
 $$
 
 \smallskip 
 
 $\bullet$ 
 By construction, we have $f \in C([0,T];L^1)$. 
 We indeed have $f(t_k) \to f(t)$ a.e. on $\OO$ as a consequence of the continuity result in Theorem~\ref{theo:Kolmogorov-trace} applied to the function $g :=\beta(f)$ with $0 \le \beta(s) \le |s|^{1/2}$. On the other hand, $\{ f(t); \, t \in [0,T] \}$ is weakly $L^1$ relatively compact as a consequence of the $L^1_\omega$ bound  \eqref{eq:Galframe-AprioriEstim1} (with $p=1$) and the equi-integrability estimate \eqref{eq:Galframe-AprioriEstim6} (recall that $\Omega$ is bounded). The claimed strong continuity follows. 
  We may thus easily pass to the limit $t \to 0$ in the above formula in order to get 
  $$
 \int_\OO f(T) g_T  dxdv =  \lim_{t\to0} \int_\OO f(t) g(t) dxdv = 0,
 $$
 by using $f(0) = 0$  and $g \in  L^\infty(0,T;L^\infty(\OO))$. 
Because $g_T \in C_c(\OO)$ is arbitrary, we deduce that $f(T) = 0$ for any $T > 0$ and the uniqueness is proved. 
%
%
\end{proof}

\subsection{Additional a priori estimates in a Radon measures framework}  
\label{subsec:addRadonMeasuresFrame}
 We present an additional a priori estimate  which holds for   nonnegative solutions in  a $ M^1_{\omega,0}$ framework. 
 More precisely, we claim that any nonnegative  solution $f_t$ associated to an initial datum $0 \le f_0 \in M^1_{\omega,0}$ satisfies, at least formally, 
\beqn\label{eq:MeasureSolBdd1}
\limsup_{t\to0}  \int_\OO f_{t} \phi^c  \le    \int_\OO f_{0}   \phi^c, 
\eeqn
for any $\phi \in \DD(\Omega)$ such that $0 \le \phi  \le 1$, and we  denote $\phi ^c := 1 - \phi $. 
We indeed first observe that 
$$
\frac{d}{dt} \int_\OO f \phi ^c = h_1 -  h_2, 
$$
with
$$
h_1 :=  \int_\OO \{ (v \cdot \nabla_x \phi ^c) + (c- \Div_v b)  \phi ^c)f
$$
and
$$
h_2 :=   \int_\Sigma \gamma f n_x \cdot v = \int_{\Sigma_+} (1-\iota_S -\iota_D ) \gamma_+f n_x \cdot v ,
$$
so that $h_1 \in  L^2(0,T)$ and $h_2 \ge 0$. For the first bound, on the one hand, we observe that $\omega \gtrsim \langle v \rangle^{r_0}$, $r_0 := 3d/2 + 1 + \gamma/2$ and  $\omega \langle v \rangle^\varsigma \gtrsim \langle v \rangle^{r_2}$, $r_2 := 3\gamma/2$, so that from \eqref{eq:Galframe-AprioriEstim1} and \eqref{eq:Galframe-AprioriEstim1BIS}, we have 
$$
f \in L^\infty(0,T;L^1_\omega) \cap L^1(0,T;L^1_{\omega\langle v \rangle^\varsigma}) \subset L^\infty(0,T;L^1_{\omega_0}) \cap L^1(0,T;L^1_{\omega_2}) \subset L^2(0,T;L^1_{\omega_1}), 
$$
with $\omega_1 := \langle v \rangle^\gamma$. On the other hand, we have $\langle v \rangle + |c- \Div_v b| \lesssim \omega_1$.  
Integrating in time the above differential equation and using the Cauchy-Schwarz inequality, we obtain
$$
\int_\OO f_t \phi ^c = 
\int_\OO f_0 \phi ^c + H_1 - H_2, 
 $$
 with $H_1 \in C^{0,1/2}([0,T])$, $H_1(0) = 0$, and $H_2 \ge 0$. More precisely
$$
 \int_\OO f_{t} \phi ^c \le  \int_\OO f_0 \phi ^c + t^{1/2} C \| \phi  \|_{W^{1,\infty}} \| f \|_{L^2_{\omega_1}(\UU)}, 
$$
for any $t \in (0,T)$, and thus \eqref{eq:MeasureSolBdd1} follows. Now, for any weight function $m : \R^d \to (0,\infty)$ such that $m$ is decreasing and $m\omega$ is increasing and any function $\psi \in \DD(\R^d)$ such that ${\bf 1}_{B_R} \le \psi \le 1$, we set $\chi := \phi\psi$, $\chi^c := 1-\chi$ and we observe that 
$$
{\chi^c \over m} = {\phi^c \over m} {\bf 1}_{B_R} + \Bigl( {\phi^c \over m} +{ \phi \over m}  \psi^c  \Bigr) {\bf 1}_{B^c_R} ,
$$
so that 
 \beqn\label{eq:MeasureSolBdd3}
 \int_\OO f_{t} {\chi^c \over m}  
 \le   {1 \over m(R)} \int_\OO f_{t}   \phi^c  + {2 \over (m\omega)(R)}  \int_\OO f_t \omega, \quad \forall \, t \ge 0
\eeqn
 As a consequence of \eqref{eq:MeasureSolBdd1} and \eqref{eq:Galframe-AprioriEstim1}, we thus deduce
 \beqn\label{eq:MeasureSolBdd2}
\limsup_{t\to0}  \int_\OO f_{t} {\phi^c \over m}  \le   {1 \over m(R)} \int_\OO f_{0}   \phi^c  + {2 \over (m\omega)(R)}  \int_\OO f_0 \omega.
\eeqn

\subsection{Radon measures solutions and  fundamental solutions}

\begin{theo}\label{theo:E-UinM1} 
For any admissible weight function $\omega$ and  any $f_0 \in M^1_{\omega,0}$, there  exists a unique   solution   $f \in C(\R_+;M^1_\omega) \cap C((0,\infty);L^\infty_\omega)$ associated to the KFP equation \eqref{eq:KFP}-\eqref{eq:KolmoBdyCond} in a sense that we discuss below.  

As a consequence, for any $z_0 := (x_0,v_0) \in \OO$, there  exists a unique fundamental solution   $F \in C(\R_+;M^1_\omega) \cap C((0,\infty);L^\infty_\omega)$ associated to the KFP equation  \eqref{eq:KFP}-\eqref{eq:KolmoBdyCond} and the initial datum $\delta_{z_0}$.   
\end{theo}

\begin{proof}[Proof of Theorem~\ref{theo:E-UinM1}]

\smallskip\noindent
{\sl Step~1. Existence.}
Because of the a priori estimates \eqref{eq:Galframe-AprioriEstim1}, \eqref{eq:Galframe-AprioriEstim3}, 
 \eqref{eq:Galframe-AprioriEstim4}, \eqref{eq:Galframe-AprioriEstim5}, \eqref{eq:Galframe-AprioriEstim6}, we may proceed exactly as during the proof of Theorem~\ref{theo:E-UinLp}, 
 and we obtain without difficulty the existence of 
$$
f \in C([0,T]; \DD'(\OO)) \cap L^\infty(0,T;L^1_\omega(\OO))  \cap L^\infty(t_0,T;L^\infty_\omega(\OO)), 
\quad \nabla_v f \in L^2(t_0,T;L^2_\omega(\OO)), 
$$
for any $0 < t_0 < T$ and any admissible weight function $\omega$, which is a renormalisation solution on $(0,T) \times \bar\OO$ and a  weak solution $[0,T) \times \OO$ corresponding 
to the initial condition $f_0$. 
More precisely, we have both 
\beqn\label{eq-M1_RenormalizationFormula}
\int_{\UU}  \{ \beta(f)   \LLLL^* \varphi  + \beta''(f) |\nabla_vf |^2 \varphi \} + \int_\Gamma   \beta(\gamma \, f) \, \varphi \,\,   n_x \cdot {v} \,  d{v} d\sigma_{\! x}dt
= 0,
\eeqn
for any  $\varphi \in \DD((0,T) \times \bar\OO)$ and any $\beta \in \BBB_3$, and
\beqn\label{eq-M1_weakFormula}
\int_{\UU}   f \LLLL^* \varphi  
= \int_{\OO} \varphi(0,\cdot) f_0(dxdv), 
\eeqn
for any $\varphi \in \DD([0,T) \times \OO)$.
By construction, we may also assume that $f$ satisfies \eqref{eq:MeasureSolBdd1}. 
Because of Theorem~\ref{theo:WP-KFP&KFP*}, for any weight function $m:\R^d \to (0,\infty)$ such that $m$ is decreasing and $m\omega$ is increasing and converges to $\infty$, 
and any
\beqn\label{eq:borne-varphi-Linftym}
\varphi \in C([0,T]; L^1_m(\OO))  \cap L^\infty(0,T;L^\infty_m(\OO)), 
\quad \nabla_v \varphi  \in L^2(0,T;L^2_m(\OO)), 
\eeqn
solution  to the backward dual KFP equation associated to a final datum $\varphi_T \in C_c(\OO)$, we have 
\beqn\label{eq:uniqueM1-identityTt}
\int f(T) \varphi_T = 
\int f(t)\varphi(t), \quad \forall \, t \in (0,T).
\eeqn

\smallskip\noindent
{\sl Step 2. Improved identity.} We claim that \eqref{eq:uniqueM1-identityTt} also holds for $t=0$. 
From the weak formulation, we have $t \mapsto f(t) \in C([0,T];\DD'(\OO))$. Because of the $L^\infty(0,T;M^1_\omega)$ bound, we deduce that  
 \beqn\label{eq:UniqueM1-ftwto0}
f(t) \wto f_0 \quad\hbox{in}\quad (C_c(\OO))' \quad\hbox{as}\quad t \to 0.
 \eeqn
 \smallskip 
 Because the solution to the backward dual problem also satisfies $0 \le \varphi \in C([0,T) \times \OO)$ as a consequence of \cite[Theorem 3]{MR3923847} when $0 \le \varphi_T \in C_c(\OO)$ and with the notations of Section~\ref{subsec:addRadonMeasuresFrame}, 
  we may next write 
 \bean
 \bigl|\int_\OO f(t)\varphi(t) - \int_\OO f_0 \varphi(0) \Bigr|
&\le&
 \bigl|\int_\OO f(t) \varphi(t) \chi - \int_\OO f_0 \varphi(0) \chi  \Bigr| + \int_\OO f(t) \varphi(t) \chi^c  + \int_\OO f_0 \varphi(0) \chi^c , 
\eean
for any $\chi \in \DD(\OO)$ such that $0 \le \chi \le 1$. 
For the first term, we have 
\bean
 \bigl|\int_\OO f(t) \varphi(t) \chi - \int_\OO f_0 \varphi(0) \chi  \Bigr|
\le
  \bigl|\int_\OO f(t) (\varphi(t) \chi -  \varphi(0) \chi ) \Bigr| +  \bigl|\int_\OO (f(t)  -   f_0 )\varphi(0) \chi  \Bigr| \to 0, 
\eean
as $t \to \infty$, thanks to  $\varphi \chi  \in C_c([0,T) \times \OO)$ and \eqref{eq:UniqueM1-ftwto0}. Particularizing  $\chi := \phi \psi$, with $\phi$ and $\psi$ defined in Section~\ref{subsec:addRadonMeasuresFrame}, and using \eqref{eq:borne-varphi-Linftym} and \eqref{eq:MeasureSolBdd2},  we have 
\bean
\limsup_{t \to 0}  \int_\OO f(t) \varphi(t) \chi^c   
&\lesssim& \limsup_{t \to 0} \int_\OO f(t)m^{-1}  \chi^c  
\\
&\lesssim&  {1 \over m(R)} \int_\OO f_0 \phi ^c   + {1 \over (m\omega)(R)} \| f_0 \|_{L^1_\omega}.
\eean
Using the above estimates as well as \eqref{eq:MeasureSolBdd3} at time $t=0$, we conclude with
 \bean 
 \limsup_{t \to 0} \bigl|\int_\OO f(t) \varphi(t) - \int_\OO f_0 \varphi(0) \Bigr|
 &\le& {1 \over m(R)} \int_\OO f_0 \phi ^c   + {1 \over (m\omega)(R)} \| f_0 \|_{L^1_\omega}. 
\eean
Assuming now that $\phi \ge {\bf 1}_{\Omega_\eps}$ and using the very definition of $f_0 \in M^1_{\omega,0}$, the RHS term vanishes in the limit $\eps \to 0$ and $R \to \infty$. We may thus pass to the limit in \eqref{eq:uniqueM1-identityTt}, 
and we obtain 
\beqn\label{eq:uniqueM1-identityT0}
\int f(T) \varphi_T = 
\int f_0 \varphi(0). 
\eeqn

\smallskip\noindent
{\sl Step~3. Uniqueness.}
 We consider two solutions $f_1, f_2$ associated to the same initial datum $0 \le f_0 \in M^1_{\omega,0}$. From Step 2, we have 
 $$
 \int_\OO f_1(T) \varphi_T 
 = 
 \int_\OO f_2(T) \varphi_T, 
 $$
 for any $T > 0$ and $\varphi_T \in C_c(\OO)$, and thus $f_1 = f_2$. 
\end{proof}

\section{About the Harnack inequality}
\label{sec-Harnack}

In this section we establish a strong maximum principle for the solutions of the kinetic Fokker-Planck equation in the form of a Harnack inequality, which is very similar to those in \cite[Theorem 2.15]{MR4412380} and  \cite{MR4017782}.

\begin{theo}\label{theo:Harnack}
Consider a  weak solution $0 \le f \in L^2((0,T) \times \OO) \cap L^2((0,T) \times \Omega; H^1(\R^d))$  
to the Kinetic Fokker-Planck equation \eqref{eq:KFP}. For any $0 < T_0 < T_1 < T$ and $\eps > 0$, there holds 
\beqn\label{eq:Harnack}
\sup_{\OO_\eps} f_{T_0} \le C  \inf_{\OO_{\eps}} f_{T_1}, 
\eeqn
for some constant $C = C(T_0,T_1,\eps) > 0$, where we recall that $\OO_\eps$ is defined in \eqref{def:OmegaEpsOeps}.
\end{theo}

The proof will be obtained in two steps. In a first step we shall obtain a local version of the Harnack inequality, and then in a second step we shall use a chain argument in order to get \eqref{eq:Harnack}. The local Harnack inequality is a direct consequence of \cite[Theorem~5 \& Proposition~12]{MR4453413}, see also \cite{MR3923847} for previous results in that direction,  which apply to 
super- and sub-solutions to the Kinetic Fokker-Planck equation with vanishing damping term
\beqn\label{eq:KFP-M}
\partial_t g = \mathfrak{M} g := -  v \cdot \nabla_x g +  \Delta_v g + b \cdot \nabla_v g  . 
\eeqn
For the reader's convenience, we state these results now. For that purpose, for $r > 0$ and $z_0 := (t_0,x_0,v_0) \in \R^{1+2d}$ we define the set $$
Q_r(z_0) := \{ (t,x,v) \in \R^{1+2d} \mid -r^2 < t -t_0 \le 0, \, |x-x_0 - (t-t_0)v_0| < r^3, \, |v-v_0| < r \}
$$
as well as the map
$$
\TT_{r,z_0} : (t,x,v) \mapsto (t_0+r^2 t , x_0+ r^3x + r^2 t v_0 , v_0+ r v),
$$
and we observe that $\TT_{R,z_0} (Q_r(0,0,0)) = Q_{r R}(z_0)$.

\begin{theo}\label{theo:Harnack-M} 
Let $T>0$.

\smallskip
(1) There exist $\zeta \in (0,1)$ and $C_1 \in (0,\infty)$ such that for any $z_0 \in \UU$, any $0<R \le 1$ with $\overline{Q_R(z_0)} \subset \UU$,  and any nonnegative weak super-solution $g$ to \eqref{eq:KFP-M}, there holds 
\beqn\label{eq:super-LzetaInf}
\| g \|_{  L^\zeta(\tilde Q^{-}_{\eta R}(z_0))} \le C_1 \inf_{Q_{\eta R}(z_0)} g,
\eeqn
where $\eta=1/40$ and $\tilde Q^{-}_{\eta R}(z_0) := \TT_{R,z_0} (Q_{\eta}(-\tau, 0,0)) = Q_{\eta R} ( z_0 - ( R^2 \tau , R^2 \tau v_0 , 0) )$ with $\tau := 19 \eta^2/2$.

\smallskip
(2) For any $z_0 \in \UU$, any $0<r< r' \le 1$ such that $\overline{Q_{r'}(z_0)} \subset \UU$, and any  $\zeta >0$, there is $C_3 >0$ such that any nonnegative weak sub-solution $h$ to \eqref{eq:KFP-M} satisfies 
\beqn\label{eq:super-LinftyL1}
 \| h \|_{L^\infty(Q_r(z_0))}  \le C_3 \| h \|_{L^\zeta(Q_{r'}(z_0))}. 
\eeqn
\end{theo}

\begin{proof}[Proof of Theorem~\ref{theo:Harnack}] 
We split the proof into two steps.

\medskip\noindent
\textit{Step 1: Local Harnack inequality.}
We claim that for any $z_0 \in \UU$ and $0 < R \le 1$ such that $\overline{Q_R(z_0)} \subset \UU$, there exists a constant $C>0$ such that
\begin{equation}\label{eq:local-Harnack}
\sup_{\tilde Q^{-}_{\eta R/2}(z_0)} f \le C \inf_{Q_{\eta R/2}(z_0)} f,
\end{equation}
where $\tilde Q^{-}_{\eta R/2}(z_0) := Q_{\eta R/2} (z_0 - (R^2 \tau , R^2 \tau v_0 , 0))$.

On the one hand, we take $\lambda > \| c \|_{L^\infty(Q_R(z_0))}$ and we set $g := e^{\lambda t} f$. The function $g$ satisfies 
$$
\partial_t g  =  \mathfrak{M} g + (\lambda+ c)g \ge \mathfrak{M} g \quad\hbox{in}\quad Q_R(z_0), 
$$
so that $g$ is a nonnegative weak super-solution  to \eqref{eq:KFP-M}. We deduce from Theorem~\ref{theo:Harnack-M}--(1) that 
$$
e^{\lambda t_0} e^{-\lambda R^2 (\eta^2+\tau)} \| f \|_{ L^\zeta(\tilde Q^{-}_{\eta R}(z_0))} 
\le \| g \|_{ L^\zeta(\tilde Q^{-}_{\eta R}(z_0))}  
\le C_1   \inf_{Q_{\eta R}(z_0) } g  
\le C_1  e^{\lambda t_0} \inf_{Q_{\eta R}(z_0) }  f .
$$

On the other hand, the function $h := e^{-\lambda t} f$ with $\lambda > \| c \|_{L^\infty ( Q_{\eta R} (z_0 - (R^2 \tau , R^2 \tau v_0 , 0)))}$ satisfies 
$$
\partial_t h  =  \mathfrak{M} h + (c - \lambda)h  \le \mathfrak{M} h \quad\hbox{in}\quad \tilde Q^{-}_{\eta R}(z_0)) = Q_{\eta R} (z_0 - (R^2 \tau , R^2 \tau v_0 , 0)). 
$$
Therefore $h$ is a nonnegative weak sub-solution to \eqref{eq:KFP-M}, and thus we deduce from Theorem~\ref{theo:Harnack-M}--(3) that 
$$
\begin{aligned}
e^{-\lambda (t_0-R^2 \tau)} \| f \|_{L^\infty(Q_{\eta R/2} (z_0 - (R^2 \tau , R^2 \tau v_0 , 0))}
&\le \| h \|_{L^\infty(Q_{\eta R/2} (z_0 - (R^2 \tau , R^2 \tau v_0 , 0)))} \\
&\le C_3 \| h \|_{L^\zeta (Q_{\eta R} (z_0 - (R^2 \tau , R^2 \tau v_0 , 0)))} \\
&\le  e^{-\lambda (t_0-R^2 \tau)} e^{ \lambda \eta^2 R^2} C_3 \| f \|_{L^\zeta( Q_{\eta R} (z_0 - (R^2 \tau , R^2 \tau v_0 , 0)) )}. 
\end{aligned}
$$
We conclude the local Harnack inequality~\eqref{eq:local-Harnack} by gathering the two previous estimates.

\medskip\noindent
\textit{Step 2: Proof of \eqref{eq:Harnack}.}
Once the local Harnack inequality \eqref{eq:local-Harnack} holds, one can deduce \eqref{eq:Harnack} by following the second step in the proof of~\cite[Theorem 2.15]{MR4412380}, which uses the Harnack chain from~\cite{MR4017782}.
\end{proof}

\section{Constructive asymptotic estimate} \label{sec:KR}

\subsection{An abstract constructive Krein-Rutman-Doblin-Harris theorem}
\label{subsec:KR-statement} 

We formulate a general abstract constructive Krein-Rutman-Doblin-Harris theorem in the spirit of the ones presented in the recent work \cite[Section~6]{sanchez:hal-04093201}. 

\smallskip
We consider a positive   semigroup $S = (S_t) = (S(t))$ on a Banach lattice $X$, which means that $X$ is a Banach space endowed with a  closed positive cone  $X_+$ (we write $f \ge 0$ if $f \in X_+$) and that  $S_t$ is a bounded linear mapping such that $S_t : X_+ \to X_+$ for any $t \ge 0$. 
We also assume that $S$ is in duality with a dual semigroup $S^*$ defined on a dual  Banach lattice $Y$,  with  closed positive cone  $Y_+$. 
More precisely, we assume that $X \subset Y'$ or $Y \subset X'$,
 so that the bracket $\langle \phi,f \rangle$ is well defined for any $f \in X$, $\phi \in Y$, 
 that $f \in X_+$ (resp. $\phi \ge 0$) iff $\langle \psi , f\rangle \ge 0$ for any $\psi \in Y_+$ (resp. iff  $\langle \phi, g \rangle \ge 0$ for any $g \in X_+$) and  
that $\langle  \phi , S(t) f\rangle = \langle  S^*(t)\phi , f \rangle$, for any $f \in X$, $\phi \in Y$ and $t \ge 0$. 
We denote by $\LL$ the generator of $S$ with domain $D(\LL)$ and by   $\LL^*$ the generator of $S^*$ with domain $D(\LL^*)$.
 We are interested in the existence of positive eigenvectors for  both $\LL$ and $\LL^*$, and in their quantified exponential stability.

\smallskip
When $\| \cdot \|_k$ is a norm on $X$ (resp. $Y$), we denote $X_k := (X, \| \cdot \|_k)$ (resp. $Y_k := (Y, \| \cdot \|_k)$). 
For $\psi \in Y_+$ and $g \in X_+$, we define the seminorms
 $$
 [f]_\psi := \langle \psi ,  |f| \rangle, \ \forall \, f \in X, \quad 
[\phi]_g:= \langle |\phi|, g \rangle, \ \forall \, \phi \in Y.
 $$

\smallskip
In order to obtain a very accurate and constructive description of the longtime asymptotic behavior of the semigroup $S$, we introduce additional assumptions.

\smallskip
$\bullet$  We first make the strong dissipativity assumption 
 \bear
\label{eq:NEWHarris-Primal-LyapunovCond}
\|   S (t) f  \|_k
&\le& C_0 e^{\lambda t}   \|  f  \|_k +   C_1 \int_0^t e^{\lambda(t-s)}  \|  S(s) f \|_{0} ds, 
\\
\label{eq:NEWHarris-Dual-LyapunovCond}
\|   S^* (t) \phi \|_k
&\le&  C_0 e^{\lambda t}   \|  \phi \|_k +   C_1 \int_0^t e^{\lambda(t-s)}     \| S^*(s)\phi \|_0 ds, 
 \eear
for any $f \in X$, $\phi \in Y$, $t > 0$ and $k=0,1$, where $\lambda \in \R$,  $C_i \in (0,\infty)$ and  $\| \cdot \|_k$, $k=0,1$ denote two families of dual norms on $X$ and $Y$  such that $X_1 \subset X_0$ and $Y_1 \subset Y_0$. 
More precisely, we assume 
\bear\label{eq:KRnorm1}
&&\| f \|_0 \le \| f \|_1, \quad \| \phi \|_0 \le \| \phi \|_1, 
\quad |\langle \phi, f \rangle| \le \| \phi \|_0 \| f \|_1, \quad
|\langle \phi, f \rangle| \le \| \phi \|_1 \| f \|_0, 
\eear
for any $f \in X$ and $\phi \in Y$.  

\smallskip
$\bullet$  
We also assume for instance  one of the two following conditions
\begin{subequations}\label{eq:NEWH2}
\beqn\label{eq:NEWH2f0} 
\exists \, \lambda_0 \in \R, \  \lambda_0 > \lambda, \ \exists \, t_0 > 0, \ \exists \, f_0 \in X_+ \backslash \{ 0 \}, \quad S (t_0)  f_0 \ge e^{\lambda_0 t_0} f_0,
\eeqn
or
\beqn\label{eq:NEWH2psi0}
\exists \, \lambda_0 \in \R, \  \lambda_0 > \lambda, \ \exists \, t_0 > 0, \ \exists \, \phi_0 \in Y_+ \backslash \{ 0 \}, \quad S^* (t_0) \phi_0 \ge e^{\lambda_0 t_0} \phi_0, 
\eeqn
\end{subequations}
and we refer to \cite[Lemma~2.4]{sanchez:hal-04093201} for variants of these conditions regarding the existence of positive supereigenvectors.

\smallskip
$\bullet$ Next, we make a slightly relaxed  Doblin-Harris  positivity  assumption 
   \bear
\label{eq:NEWDoblinHarris-primal}
&& 
S_T f \ge  \eta_{\eps,T} g_\eps [S_{T_0} f]_{\psi_\eps}, \quad \forall \, f \in X_+, 
\\ 
 \label{eq:NEWDoblinHarris-dual}
&& 
S^*_T \phi \ge  \eta_{\eps,T} \psi_\eps [S^*_{T_0}  \phi]_{g_\eps},  \quad \forall \, \phi \in Y_+, 
 \eear
 for any  $T \ge T_1 >  T_0 \ge0$ and $\eps > 0$, where $ \eta_{\eps,T} > 0$ and where $(g_\eps)$ and $(\psi_\eps)$ are two bounded and decreasing families of $X_{+}$ and $Y_{+}$. 
We say that the above condition is relaxed because we possibly have $T_0 > 0$ while in the usual  Doblin-Harris condition \eqref{eq:NEWDoblinHarris-primal} or \eqref{eq:NEWDoblinHarris-dual} holds with $T_0 = 0$.

\smallskip
$\bullet$ We finally assume the following  compatibility interpolation like conditions  
  \bear
\label{eq:NEWHarris-LyapunovCondNpsieps}
&&  
 \| f \|_0 \le \xi_\eps  \| f \|_1 + \Xi_\eps [f]_{\psi_\eps}, \ \forall \, f \in X,  \  \eps \in (0,1], 
\\ 
\label{eq:NEWHarris-LyapunovCondNpsiepsDual}
&& 
\| \phi \|_0 \le \xi_\eps  \| \phi \|_1 + \Xi_\eps  [\phi]_{g_\eps}, \ \forall \, \phi \in Y,  \  \eps \in (0,1], 
 \eear
for two positive families $(\xi_\eps)$ and $(\Xi_\eps)$ such that $\xi_\eps\searrow 0$ and $\Xi_\eps \nearrow \infty$ as $\eps \searrow 0$.

 \smallskip

 It is worth pointing out that the above assumptions are written in a very symmetric way between the primal and dual spaces and semigroups.
They are yet too rough for addressing the issue of the existence of positive eigenvectors.
This existence problem is not our main purpose since it has been widely treated for instance in~\cite{sanchez:hal-04093201} (see also the references therein).
Nevertheless, for keeping the presentation as self-contained as possible, we consider some strengthened (and quite natural) assumptions that allow us to derive the existence part,
keeping in mind that many variants are possible and referring the interested reader to Sections~2 and~3 in \cite{sanchez:hal-04093201}.

\smallskip
$\bullet$ On the one hand, we assume that $X_1$ is a Banach space and 
\beqn
\label{eq:NEWHarris-GrowthX1}
\| S(t) f \|_{1} \le C'_0 e^{\lambda' t} \| f \|_1, 
\eeqn
for some $\lambda' \in \R$, some $C'_0 \ge 1$, any $f \in X$ and any $t \ge 0$. Of course this is a consequence of~\eqref{eq:NEWHarris-Primal-LyapunovCond} and the Gronwall inequality with $\lambda' := \lambda + C_1$ and $C'_0 := C_0$ under the mild assumption that $t \mapsto \| S(t) f \|_1$ is a (everywhere defined) measurable and locally bounded function on $\R_+$.  
 
\smallskip
$\bullet$ On the other hand, instead of  \eqref{eq:NEWHarris-Dual-LyapunovCond}, we rather assume that
\beqn\label{eq:NEWHarris-Dual-LyapunovCondBIS1}
 S^* = V + W * S^*
\eeqn
 with 
\beqn\label{eq:NEWHarris-Dual-LyapunovCondBIS2}
 \| V(t) \phi \|_0 \le   C_0 e^{\lambda t}   \|  \phi \|_0, \quad 
  \| W (t)  \phi \|_1 \le   C_1 e^{\lambda t}   \|  \phi \|_0, \quad W \ge 0, 
\eeqn
which is a variant of \eqref{eq:NEWHarris-Dual-LyapunovCond} for $k=1$ and which obviously implies \eqref{eq:NEWHarris-Dual-LyapunovCond} for $k=0$.
We also assume that bounded sequences of $Y_1$ are weakly compact sequences in the $\sigma(Y_0,X_1)$ sense.  

\begin{theo}\label{theo:KRDoblinHarris}
Consider a semigroup $S$ on a Banach lattice $X$ which satisfies the conditions~\eqref{eq:NEWHarris-Primal-LyapunovCond} and \eqref{eq:KRnorm1}--\eqref{eq:NEWHarris-Dual-LyapunovCondBIS2}. 
Then, there exists a unique eigentriplet $(\lambda_1, f_1,\phi_1) \in \R \times X \times Y$ such that 
$$
\LL f_1 = \lambda_1 f_1, \quad f_1 \ge 0, \quad
\LL^* \phi_1 = \lambda_1 \phi_1, \quad \phi_1 \ge 0, 
$$
together with the normalization conditions $\| \phi_1 \|_0 =1$, $\langle \phi_1, f_1 \rangle = 1$. Furthermore, there exist some constructive constants $C \ge 1$ and $\lambda_2 < \lambda_1$ such that 
\beqn\label{eq:KRTh-constructiveRate}
\|S(t) f - \langle f,\phi_1 \rangle f_1 e^{\lambda_1 t}  \|_1 \le C e^{\lambda_2 t} 
\|  f - \langle f,\phi_1 \rangle f_1 \|_1
\eeqn
for any $f \in X$ and $t \ge 0$.
 \end{theo}

\smallskip
In the next section, we first comment on this theorem.
Next we apply the result to the KFP equation in Section~\ref{subsec:applicationsKRtoKFP}.
The last section is devoted to the proof of Theorem~\ref{theo:KRDoblinHarris}.

\subsection{Discussion}

Let us make a few comments on Theorem~\ref{theo:KRDoblinHarris} and its assumptions.

 \smallskip
$\bullet$  The above result is a variant, and in some sense a consequence, of \cite[Theorem 6.3]{sanchez:hal-04093201}, see also \cite[Theorem~5.3]{MR3489637} and \cite[Theorem 2.1]{Bansaye2022}. 
However, the set of assumptions here only involves the semigroups $S$ and $S^*$ and not the eigenelements $(\lambda_1, f_1,\phi_1)$ as it was the case in \cite[Theorem 6.3]{sanchez:hal-04093201}. 
That makes clearer the properties  on the semigroup $S$ really necessary to get the conclusions. The framework is very general and in particular it is not restricted to the measures space in duality with the bounded measurable functions space as it is the case in \cite{Bansaye2022}. Our result is truly constructive what was not the case in the approach developed in  \cite{MR3489637}.

 \smallskip
$\bullet$ In the conservative case, namely $\lambda_1 = 0$, $\phi_1 \equiv 1 \in Y_1 \subset L^\infty$, and solely assuming \eqref{eq:NEWHarris-Primal-LyapunovCond} with $\lambda < 0$, \eqref{eq:NEWDoblinHarris-primal}, \eqref{eq:NEWHarris-LyapunovCondNpsieps} and $X_1$ is a Banach space,  the same conclusion \eqref{eq:KRTh-constructiveRate} holds true by just following the same proof. Such a result is a general Banach lattice variant of the classical Doblin-Harris theorem available in the measures space in duality with the bounded measurable functions space framework, see \cite{MR2857021,MR4534707,Bansaye2022,sanchez2024VCK} for more details and references.

\smallskip
$\bullet$   It is also worth emphasizing that  \eqref{eq:NEWDoblinHarris-dual} with $\phi := \psi_\eps$ implies 
$$
S^*_T \psi_\eps \ge \eta_{\eps,T}  [S^*_{T_0} \psi_\eps]_{g_\eps} \psi_\eps  =: e^{\lambda'_0 T} \psi_\eps, 
$$
what is a  condition similar to \eqref{eq:NEWH2psi0}. We may thus alternatively first assume \eqref{eq:NEWDoblinHarris-primal}, \eqref{eq:NEWDoblinHarris-dual} and next assume that \eqref{eq:NEWHarris-Primal-LyapunovCond},  \eqref{eq:NEWHarris-Dual-LyapunovCond} hold for some $\lambda < \lambda_0'$. 
In other words, our constructive Krein-Rutman-Doblin-Harris theorem is really a consequence of a suitable strong dissipativity condition and of a suitable positivity condition on both primal and dual semigroups together with a   compatibility conditions over the several involved norms and seminorms.
This strong dissipativity condition is automatically satisfied when the semigroup has appropriate smoothing effects (measured in terms of gain of regularity, exponent of integrability or weight function) as it is the case here for the kinetic Fokker-Planck equation (see Section~\ref{subsec:applicationsKRtoKFP} below) but can be not true for less regularizing semigroup as for the linear Boltzmann model for instance.

\smallskip 

$\bullet$ 
An alternative natural way to formulate the  Doblin-Harris  positivity conditions \eqref{eq:NEWDoblinHarris-primal}, \eqref{eq:NEWDoblinHarris-dual} is to rather assume
a family of weak Harnack conditions
\bear
\label{eq:NEWweakHarnack-primal}
&&  
S_T f \ge  g_\eps  \int_{T_0}^T [S_t f ]_{\psi_\eps} dt  \ \hbox{ if } \,  f \ge 0, 
\\ 
 \label{eq:NEWweakHarnack-dual}
&&  
S^*_T \phi  \ge  \psi_\eps  \int_{T_0}^T [S^*_t \phi ]_{g_\eps}  dt  \ \hbox{ if } \,  \phi \ge 0, 
 \eear
 for some constants $T > T_0 \ge 0$, together with a family of supereigenvectors (or barrier) conditions 
\bear
\label{eq:NEWbarrier-primal}
S^*_t \psi_\eps \ge e^{\nu_\eps t} \psi_\eps, \quad
S_t g_\eps \ge e^{\nu_\eps t} g_\eps, \quad \forall \, t \ge T_0, 
  \eear
  for any $\eps > 0$ and some $\nu_\eps \in \R$. Using the first inequality in \eqref{eq:NEWbarrier-primal}, we find for $f\in X_+$
\bean
 \int_{T_0}^T [S_t f ]_{\psi_\eps}dt 
=
 \int_{T_0}^T \langle f, S^*_t \psi_\eps \rangle  dt 
\ge
 \int_{T_0}^T \langle f, e^{\nu_\eps t} \psi_\eps \rangle  dt =  \int_{T_0}^T  e^{\nu_\eps t} dt [f ]_{\psi_\eps},
\eean
and we thus immediately deduce \eqref{eq:NEWDoblinHarris-primal} (with $T_0 = 0$) from \eqref{eq:NEWweakHarnack-primal} and  \eqref{eq:NEWbarrier-primal}. 
We may similarly deduce \eqref{eq:NEWDoblinHarris-dual} (with $T_0 = 0$) from \eqref{eq:NEWweakHarnack-dual} and  \eqref{eq:NEWbarrier-primal}.

\smallskip
$\bullet$ We briefly discuss the link between our set of hypotheses and the strong maximum principle which is also classically used in the Krein-Rutman theory. 
For that purpose, we introduce the notion of strict positivity by writing $f \in X_{++}$ or $f > 0$ (resp.\ $\phi \in Y_{++}$ or $\phi > 0$) if $\langle   \psi ,f  \rangle > 0$ for any $\psi \in Y_+ \backslash \{ 0\}$
(resp. $\langle \phi, g \rangle > 0$ for any $g  \in X_+ \backslash \{ 0\}$).
Under assumptions  \eqref{eq:NEWHarris-LyapunovCondNpsieps} and  \eqref{eq:NEWHarris-LyapunovCondNpsiepsDual}, we claim that \eqref{eq:NEWDoblinHarris-primal} with $T_0 = 0$, or 
\eqref{eq:NEWDoblinHarris-primal} and \eqref{eq:NEWDoblinHarris-dual} together with~\eqref{eq:NEWH2f0},
imply the classical strong maximum principle, and we recall that this last one classically writes 
\beqn\label{eq:StrongMP}
f \in D(\LL) \cap X_+ \backslash \{0 \}, \ \mu \in \R, \ (\mu - \LL) f  =: g\ge 0 \quad\hbox{implies}\quad f > 0. 
\eeqn
Before proving this claim we establish the following elementary facts 
\begin{enumerate}
\item[(i)\,] $f \in X_+ \backslash \{0\}$  iff $f \in X_+$ and $\langle   \psi_\eps ,f  \rangle > 0$ for any $\eps \in (0,\eps_f)$, $\eps_f > 0$ small  enough,

\item[(ii)] $f \ge  \alpha_\eps g_\eps$, $\alpha_\eps > 0$, for any $\eps > 0$, imply  $f > 0$, 
\end{enumerate}

\noindent
as consequences of \eqref{eq:NEWHarris-LyapunovCondNpsieps} and  \eqref{eq:NEWHarris-LyapunovCondNpsiepsDual}.

One the one hand, for any fixed $f \in X \backslash \{ 0 \}$ the family of interpolation estimates  \eqref{eq:NEWHarris-LyapunovCondNpsieps} implies
$$
0 < \frac12 \| f \|_0 \le \| f\|_0  - \xi_\eps \| f \|_1 \le  \Xi_\eps [f]_{\psi_\eps},
$$
$\eps \in (0,\eps_f)$, $\eps_f > 0$ small  enough, which gives (i). In particular, $\psi_\eps \not = 0$ for $\eps > 0$ small  enough (what can also be added as an assumption in the definition of $\psi_\eps$!). 
We similarly  have $[\phi]_{g_\eps} > 0$ for any $\phi  \in Y_+ \backslash \{ 0 \}$ and any  $\eps \in (0,\eps_\phi)$, $\eps_\phi > 0$ small  enough, and thus $g_\eps \not=0$ for $\eps > 0$ small enough. In particular, we have 
$$
\langle \psi_\eps, g_\eps \rangle \ge \langle \psi_\eps, g_{\eps_0} \rangle \ge \Xi_\eps^{-1}  \frac12 \| g_{\eps_0}  \|_0 > 0, 
$$
for any $\eps \in (0,\eps_{g_{\eps_0}})$ and $\eps_0 > 0$ such that $g_{\eps_0} \not=0$. Assume now $f \ge \alpha_\eps g_\eps$, $\alpha_\eps > 0$, for any $\eps > 0$. 
For any $ \phi \in Y_+ \backslash \{ 0 \}$, we then have 
$$
\langle \phi, f \rangle \ge  \alpha_{\eps_\phi}  \langle \phi,  g_{\eps_\phi} \rangle  > 0.
$$
We have established that $f > 0$, and thus (ii) is proved. 

\smallskip
 We come now  to the proof of the strong maximum principle and thus consider $(\mu,f,g)$ satisfying the requirements of \eqref{eq:StrongMP}. We fix $\nu$ strictly larger than $\mu$ and  strictly larger than the growth bound of $S$ so that we may write \[
f  = (\nu - \LL)^{-1} ((\nu-\mu) f + g) = \int_0^\infty e^{-\nu t} S_t ((\nu-\mu) f + g) dt
\]
to get by positivity of $g$
\beqn\label{eq:fgeqStf}
f\ge (\nu-\mu) \int_0^\infty e^{-\nu t} S_t  f  dt. 
\eeqn
From \eqref{eq:NEWDoblinHarris-primal} with $T_0=0$, we deduce 
\beqn\label{eq:fgeStf}
f
\ge g _\eps   (\nu-\mu)  \int_{T_1}^\infty e^{-\nu t}  \eta_{\eps,t}  dt   \,  \langle  \psi_\eps ,    f\rangle, \quad \forall \, \eps > 0.
\eeqn   
From (i)  above and because $f \in  X_+ \backslash \{0 \}$, there exists $\eps_f > 0$ such that $\langle   \psi_\eps ,f  \rangle > 0$ for any $\eps \in (0, \eps_f)$. 
Together with \eqref{eq:fgeStf}, we deduce that $f \ge \alpha_\eps g_\eps$, with $\alpha_\eps > 0$, for any $\eps \in (0, \eps_f)$, 
and that in turn implies $f > 0$ from  (ii) above.

\smallskip

When~\eqref{eq:NEWDoblinHarris-primal} and~\eqref{eq:NEWDoblinHarris-dual} are satisfied with $T_0 > 0$, we may derive~\eqref{eq:StrongMP} by using the additional condition~\eqref{eq:NEWH2f0}.
We apply~\eqref{eq:NEWDoblinHarris-primal} with $T=t-T_1+T_0$ to the vector $S_{T_1-T_0}f$ to get that
\[S_tf\geq \eta_{\eps,t-T_1+T_0}g_\eps\langle\psi_\eps,S_{T_1}f\rangle\]
for any $t\geq 2T_1-T_0$.
Injecting this inequality in~\eqref{eq:fgeqStf},  we obtain
\begin{align*}
f&\ge g _\eps   (\nu-\mu)  \int_{2T_1-T_0}^\infty e^{-\nu t}  \eta_{\eps,t-T_1+T_0}  dt   \,  \langle  \psi_\eps ,  S_{T_1} f\rangle
\\
&\ge g _\eps   (\nu-\mu)  \int_{2T_1-T_0}^\infty e^{-\nu t}  \eta_{\eps,t-T_1+T_0}  dt   \,  \langle S^*_{T_1} \psi_\eps ,   f \rangle, 
\end{align*}
for any $\eps>0$. Together with  \eqref{eq:NEWDoblinHarris-dual}, that implies
\beqn\label{eq:H2&DoblinImpliquentPMfort}
f \ge g _\eps   (\nu-\mu)  \int_{2T_1-T_0}^\infty e^{-\nu t}  \eta_{\eps,t-T_1+T_0}  dt   \,  \eta_{\eps',T_1}   \langle S^*_{T_0} \psi_\eps ,  g_{\eps'} \rangle  \langle \psi_{\eps'} ,   f \rangle. 
\eeqn
On the other hand, taking $n\in\N$ large enough so that $nt_0\geq T_0$ and iterating~\eqref{eq:NEWH2f0}, we get that $S_{nt_0}f_0\geq e^{nt_0}f_0$ and as a consequence $S_{nt_0}f_0\in X_+\setminus\{0\}$.
We infer that necessarily $S_{T_0}f_0\in X_+\setminus\{0\}$, and the existence of  $\eps_0$ such that $\langle S^*_{T_0}\psi_\eps,f_0\rangle=\langle\psi_\eps,S_{T_0}f_0\rangle>0$ for all $\eps\in(0,\eps_0)$.
 We thus deduce in particular $S^*_{T_0}\psi_\eps\in X_+\setminus\{0\}$, and since we also have $f\in X_+\setminus\{0\}$, we deduce the existence of $\eps' >0$ such that $\langle S^*_{T_0}\psi_\eps,g_{\eps'}\rangle>0$ and $\langle \psi_{\eps'} ,   f \rangle>0$. Coming back to \eqref{eq:H2&DoblinImpliquentPMfort}, we have proved, for any $\eps\in(0,\eps_0)$, the existence of $\alpha_\eps>0$ such that $f\geq\alpha_\eps g_\eps$ and this guarantees that $f>0$.

Symmetrically, the assumptions~\eqref{eq:NEWDoblinHarris-primal}, \eqref{eq:NEWDoblinHarris-dual} and~\eqref{eq:NEWH2psi0} imply that the dual operator $\LL^*$ satisfies the strong maximum principle.

In particular, we deduce that the first eigenvectors exhibited in Theorem~\ref{theo:KRDoblinHarris}  satisfies $f_1 > 0$ and $\phi_1 > 0$.

\subsection{Application to the KFP equation: proof of Theorem~\ref{theo:KR}} 
\label{subsec:applicationsKRtoKFP}
 
In this section, we consider the kinetic Fokker-Planck equation \eqref{eq:KFP}, \eqref{eq:KolmoBdyCond},  \eqref{eq:initialDatum} and  we prove  Theorem~\ref{theo:KR} by using Theorem~\ref{theo:KRDoblinHarris}.
We define $X_0 := L^2_{\omega_2}$ for a strongly confining admissible exponential  weight function $\omega_2$ and $X = X_1 := L^r_{\omega_r}$ with $r \in (2,\infty)$ given by  Proposition~\ref{prop:EstimL1Lr-omega}  and an  admissible exponential  weight function $\omega_r$ such that $L^r_{\omega_r} \subset L^2_{\omega_2}$.  We next define $Y_0 := L^{r'}_{m_r}$ with $r'$ the conjugate exponent associated to $r$, $m_r := \omega^{-1}_r$ and $Y_1 := L^{2}_{m_2}$, with $m_2 := \omega^{-1}_2$.  Many other choices are possible. This choice however contrasts with the usual $L^1-L^\infty$ framework considered when using Doblin-Harris type arguments. 

\smallskip
We now check that the assumptions of Theorem~\ref{theo:KRDoblinHarris} are met. 

\smallskip
We recall that $\OO_\eps$ has been defined in \eqref{def:OmegaEpsOeps} and we denote $\eps_0 > 0$ such that $\OO_{\eps_0} \not= \emptyset$. 

\smallskip
$\bullet$ {\sl A  positive supereigenvector condition.} 
For a given function  $0 \le h_0 \in C^2_c(\OO)$ normalized by $\| h_0 \|_{L^2_{\omega_2}} = 1$ and such that $\supp h_0 \subset \OO_{\eps_0}$,  and for $\lambda > \omega(S_\LLL)$ the growth bound of $S_\LLL$, we define $f_0\in D(\LLL)$ as the solution to 
$$
(\lambda - \LLL) f_0= h_0\ \hbox{ in } \ \OO,
\quad 
\gamma_- f_0  =   \RRR \gamma_+ f_0\  \hbox{ on } \  \Sigma_-. 
$$

The existence and uniqueness of such a solution is a classical consequence of the existence of the semigroup $S_\LLL$ given by Proposition~\ref{prop-KFP-L2primal}.
Repeating the proof of the condition (H2) in~\cite[Section~11.4]{sanchez:hal-04093201}, there exists a constructive constant $c>0$ such that $f_0\geq c\,h_0$.
Coming back   to the equation, we have 
$$
\LLL f_0 =  \lambda f_0 - h_0 \ge   (\lambda -   c^{-1}) f_0\ \hbox{ in } \ \OO,
$$
which is a variant of \eqref{eq:NEWH2f0}, and in particular from  \cite[Remark~2.5]{sanchez:hal-04093201},  it  implies \eqref{eq:NEWH2f0} with $\lambda_0 := (\lambda -   c^{-1})$.

\smallskip\smallskip
$\bullet$ {\sl Strong dissipativity conditions.}  We define 
$$
\BB := \LLL - \AA, \quad \AA f := M \chi_R f , \quad \mathbf{1}_{B_R} \le \chi_R \le \mathbf{1}_{B_{2R}}, 
$$
with $B_r := \{ v \in \R^d; \, |v |  \le r \}$ and $\chi_R$ a smooth function. We then define the semigroup $S_\BB$ associated to $\BB$ and the reflection condition \eqref{eq:KolmoBdyCond} which existence is given by Proposition~\ref{prop-KFP-L2primal}.
We claim that for any $a^* \in \R$, we may choose $M,R > 0$ large enough in such a way that 
$S_\BB$ satisfies 
\beqn\label{eq:preuveKR-SBLpLp}
\| S_{\BB}(t) f \|_{L^p_\omega} \le e^{a^*t} \| f \|_{L^p_\omega}, \quad \forall \, t \ge 0, \ \forall \, f \in L^p_\omega, 
\eeqn
for any Lebesgue space $L^p_\omega$ with admissible exponential weight function $\omega$. 
Coming back indeed to the proof of  Lemma~\ref{lem:EstimKFPLp12}, for $p=1,2$, and more precisely to \eqref{eq:disssipSL-Lp}, the function $f(t) := S_\BB(t) f_0$ satisfies  
\bean
\frac{d}{dt} \int_\OO f^p \tilde \omega^p
&\le&  \int_\OO f^p   \tilde \omega^p \varpi^\BB, 
\eean
with $\varpi^\BB := \varpi^{\LLL}_{\tilde\omega,p} - M\chi_R$. 
Because of  \eqref{eq:varpiL-varpisharp},
we may thus fix $M,R >0$ large enough, in such a way that $\varpi^\BB \le a^*$. That implies \eqref{eq:preuveKR-SBLpLp} for $p=1,2$. We deduce that  \eqref{eq:preuveKR-SBLpLp} holds for any $p \in [1,\infty]$ as we proved the similar growth estimate for $S_\LLL$, and we thus refer to Theorem~\ref{theo:WP-KFP&KFP*} for more details.

On the other hand, from Theorem~\ref{theo:ultra} applied to the semigroup $S_\BB$, we know that 
\beqn\label{proof-KRKFP-L2Lr}
\| S_\BB(t) f \|_{L^r_{\omega_r}} \lesssim \frac{e^{Ct} }{ t^\Theta} \| f \|_{L^2_{\omega_2}}.
\eeqn
We finally recall that from Theorem~\ref{theo:existsLpM1}, we have 
\beqn\label{proof-KRKFP-LqLq}
\| S_\LLL (t) f \|_{L^q_{\omega_q}} \le e^{Ct} \| f \|_{L^q_{\omega_q}},
\eeqn
for the exponents $q=2,r$, for any admissible weight function $\omega_q$, for any $f \in L^q_{\omega_q}$ and any $t \ge 0$. 

\smallskip
Iterating the Duhamel formula
$$
S_\LLL = S_\BB + S_\BB \AA * S_\LLL, 
$$
we get 
$$
S_\LLL = \VV + \WW*S_\LLL, 
$$
with 
$$
\VV := S_\BB + \dots + (S_\BB \AA)^{*(N-1)} * S_\BB, \quad 
\WW := (S_\BB \AA)^{*N}.
$$
Here $*$ denotes the usual convolution operation for functions defined on $\R_+$ and we define recursively $U^{*1} = U$, $U^{*k} = U^{*(k-1)}*U$. 
Combining \eqref{proof-KRKFP-L2Lr} and \eqref{proof-KRKFP-LqLq}, we may use \cite[Proposition 2.5]{MR3465438} (see also \cite{MR3779780,MR3488535}) and we deduce that 
$$
\| \VV (t) f \|_{L^r_{\omega_r}} \lesssim e^{at}\| f \|_{L^r_{\omega_r}}, \quad  
\quad
\| \WW (t) f \|_{L^r_\omega} \lesssim e^{at}\| f \|_{L^2_{\omega_2}},  
$$
for $N \ge 1$ large enough, any $a > a^*$, any $t \ge 0$ and any $f \in L^r_{\omega_r}$, 
and thus 
\beqn\label{proof-KRKFP-SLLrLrL2}
\| S_\LLL(T) f_0 \|_{L^r_{\omega_r}} \le C e^{aT} \|   f_0 \|_{L^r_{\omega_r}} +  C e^{CT}\| f_0 \|_{L^2_{\omega_2}}, 
\eeqn
for any $t \ge 0$ and $f_0 \in L^r_{\omega_r}$. 
That is nothing but \eqref{eq:NEWHarris-Primal-LyapunovCond} for $k=1$. The same estimate for $k=0$ is clear, it is nothing but  Lemma \ref{lem:EstimKFPLp12}. 
The proof of \eqref{eq:NEWHarris-Dual-LyapunovCondBIS1}-\eqref{eq:NEWHarris-Dual-LyapunovCondBIS2} is  similar and it is thus skipped. 
  
\smallskip\smallskip
$\bullet$ {\sl Doblin-Harris condition.} 
%
Let us   fix $0 \le f_0 \in L^2_\omega$ and denote $f_t := S_\LLL(t) f_0$. For $T_0> 0$ and $\eps \in (0,\eps_0)$, we know from Theorem~\ref{theo:Harnack} that 
for any $T_1>T_0 > 0$ and for every $T \geq T_1$, we have 
 $$
 \sup_{\OO_\eps} f_{T_0} \le C  \inf_{\OO_{\eps}} f_{T}, 
 $$
 for a constant $C$ independent of $f$. 
 We deduce
 \bean
  f_{T} 
  &\ge& \bigl( \inf_{\OO_{\eps}} f_{T} \bigr) \mathbf{1}_{\OO_{\eps}}
  \\
&\ge& \frac1C \bigl( \sup_{\OO_{\eps}} f_{T_0} \bigr) \mathbf{1}_{\OO_{\eps}}
  \\
&\ge& \frac1C \frac{1 }{ |\OO_\eps|} \langle S_{T_0} f_0, \mathbf{1}_{\OO_\eps} \rangle    \mathbf{1}_{\OO_{\eps}}
\eean
what is  \eqref{eq:NEWDoblinHarris-primal} 
with $g_\eps = \psi_\eps := \mathbf{1}_{\OO_\eps}$. The proof of \eqref{eq:NEWDoblinHarris-dual} is identical.

 \smallskip
$\bullet$ {\sl The interpolation condition.} Let us consider two exponents $p > q$ and two locally bounded weight functions $\omega_p$, $\omega_q$ such that $\omega_p \ge \omega_q$ and $ \omega_q/\omega_p \in L^{qr'}$ with $r := p/q \in (1,\infty)$, and in particular $L^p_{\omega_p} \subset L^q_{\omega_q}$.  We have 
\bean
\| f \|_{L^q_{\omega_q}} 
&\le& \| f \mathbf{1}_{\OO_\eps} \|_{L^q_{\omega_q}} +  \| f\omega_q \mathbf{1}_{\OO^c_\eps} \|_{L^q}   
\\
&\le& \| f \mathbf{1}_{\OO_\eps} \|^\theta_{L^p_{\omega_q}}  \| f \mathbf{1}_{\OO_\eps} \|^{1-\theta}_{L^1_{\omega_q}} + 
   \| f\omega_p \|_{L^p} \|\omega_q/\omega_p \mathbf{1}_{\OO^c_\eps}  \|_{L^{r'q}}  
\\
&\le&  ( \eps^{ \frac{1 }{ \theta}}  +  \|\omega_q/\omega_p \mathbf{1}_{\OO^c_\eps}  \|_{L^{r'q}}   )  \| f  \|_{L^p_{\omega_p}}  +  \eps^{ \frac{1} {\theta-1}}  \| f \mathbf{1}_{\OO_\eps}  \|_{L^1_{\omega_q}}, 
\eean
where we have used the classical interpolation inequality (with $1/q = \theta/p + 1-\theta$, $\theta \in (0,1)$) and the Holder inequality in the second line,
and the Young inequality in the third line. 
That implies  both \eqref{eq:NEWHarris-LyapunovCondNpsieps} and \eqref{eq:NEWHarris-LyapunovCondNpsiepsDual}.

\medskip
The same conditions hold for the dual problem, so that we may apply Theorem~\ref{theo:KRDoblinHarris} in order to conclude that  Theorem~\ref{theo:KRDoblinHarris} holds in the space $X_1= L^r_{\omega}$. 
We deduce that  Theorem~\ref{theo:KRDoblinHarris} holds in any weighted Lebesgue spaces associated to admissible weight functions by using the extension trick as developed in \cite{MR3779780,MR3488535} to which we refer for details. It is also worth emphasizing that the uniform estimates in \eqref{eq:theoKR-strictpo&Linftybound} directly follow from the ultracontractivity estimate established in Theorem~\ref{theo:ultra} for the primal and the dual semigroups
and that the strict positivity properties  in \eqref{eq:theoKR-strictpo&Linftybound} directly follow from the discussion about the strong maximum principle just after the statement of Theorem~\ref{theo:KRDoblinHarris}. 
 Furthermore, $f_1, \phi_1 \in C(\OO)$ as a direct consequence of Theorem~\ref{theo:WP-KFP&KFP*} and \cite[Theorem 3]{MR3923847}.

\subsection{Proof of Theorem~\ref{theo:KRDoblinHarris}} This section is devoted to the proof of Theorem~\ref{theo:KRDoblinHarris} which is split into six steps. 
 We closely follow  the material presented in  \cite[Section~2,3]{sanchez:hal-04093201} (see also \cite{PL2}) in Steps~1, 2, 3, 4 and the  material presented in  \cite{MR4534707} in Steps~5 and 6.

\medskip\noindent
{\sl Step 1. Existence of $\phi_1$.} 
From the fact that $S^*$ is a positive semigroup, \eqref{eq:NEWH2psi0} and \cite[Lemma 2.6]{sanchez:hal-04093201}, we know 
that
$$
\lambda_1 := \inf 
\{ \kappa \in \R; \ z-\LL \hbox{ is invertible  for any } z \ge \kappa\} \ge \lambda_0
$$ and 
\beqn\label{eq:lambda1approx}
\exists \, \lambda_n \searrow \lambda_1, \ \exists \, \hat \phi_n \in D(\LL^*) \cap Y_+,\  \varphi_n := \lambda_n \hat \phi_n - \LL^*  \hat \phi_n \ge 0, \ \| \hat \phi_n \|_0 = 1, \ \| \varphi_n \|_0 \to 0.
\eeqn
Because $\lambda_n > \lambda_1$, the following representation formula
$$
\hat \phi_n := (\lambda_n - \LL^*)^{-1} \varphi_n  = \int_0^\infty S^* (t) e^{-\lambda_n t}  \varphi_n dt
$$
holds true. Introducing the sequences 
$$
v_n := \VV_n \varphi_n, \quad \VV_n := \int_0^\infty V(t) e^{-\lambda_n t}   dt, 
$$
and 
\bean
w_n 
&:=&
\int_0^\infty (W * S^*)(t) e^{-\lambda_n t}  \varphi_n dt
\\
&=&
\int_0^\infty W (t) e^{-\lambda_n t}  dt  \int_0^\infty S^* (t) e^{-\lambda_n t}  \varphi_n dt 
\\
&=& \WW_n \hat \phi_n, \quad \WW_n := 
\int_0^\infty W (t) e^{-\lambda_n t}  dt ,
\eean
and using \eqref{eq:NEWHarris-Dual-LyapunovCondBIS1}, we deduce that 
\beqn\label{eq:NEW-vn&wn}
\hat \phi_n = v_n + w_n.
\eeqn
By construction and \eqref{eq:NEWHarris-Dual-LyapunovCondBIS2}, we have $\| v_n \|_0 \to 0$, $(w_n)$ is bounded in $Y_1$ and thus weakly compact in $Y_0$. That implies that $(\hat \phi_n)$ is weakly compact in $Y_0$. There thus exist a subsequence $(\hat \phi_{n_\ell})$  and $\phi_1 \in Y_0$ such that $\hat \phi_{n_\ell} \wto \phi_1$ weakly in $Y_0$. In particular, $\phi_1 \ge 0$ and $ \LL^* \phi_1 = \lambda_1 \phi_1$. 
On the other hand, from \eqref{eq:NEW-vn&wn}, we have 
$$
1 = \| \hat \phi_n \|_0 \le 
 \| v_ n \|_ 0 + \|  w_n \|_0, 
$$
with $\| v_n \|_ 0  \to 0$ again 
and 
\bean
\| w_n \|_0  
&\le& \xi_\eps \| \WW_n \hat \phi_n \|_1 + \Xi_\eps [\WW_n \hat \phi_n ]_{g_\eps} 
\\
&\le& \xi_\eps C  \| \hat \phi_n \|_0 + \Xi_\eps [\WW_n \hat \phi_n ]_{g_\eps}, 
\eean
where we have used \eqref{eq:NEWHarris-LyapunovCondNpsiepsDual} in the first line and  \eqref{eq:NEWHarris-Dual-LyapunovCondBIS2} in the second line. 
Choosing $n\ge 1$ large enough so that $\| v_n \|_0 \le 1/4$ and $\eps > 0$ small enough so that $\xi_\eps C \le 1/4$, we deduce from the two above estimates and the fact that $\WW_n \ge 0$ that 
$$
\frac12 \le  \Xi_\eps [\WW_n \hat \phi_n ]_{g_\eps} = \Xi_\eps \langle w_n, g_\eps \rangle, \quad \forall \, n \ge 0.
$$
Using that $g_\eps \in Y_0'$, we may pass to the limit in the above inequality and we deduce
\beqn\label{eq:NEW-bddinf1phi1}
\frac1{2\Xi_\eps} \le   \langle \phi_1 , g_\eps \rangle,  
\eeqn
in particular $\phi_1 \not\equiv 0$ and that concludes the proof of the existence of a dual eigenelement.  
We have thus established the  existence of a first dual eigenelement, that is $(\lambda_1,\phi_1) \in \R \times Y$ such that 
\beqn\label{eq:NEWexistlambda1phi1}
\LL^* \phi_1 = \lambda_1 \phi_1, \quad \phi_1 \ge 0, \quad \phi_1 \not= 0, \quad \lambda_1 \ge \lambda_0.
\eeqn
It is worth emphasizing that we only have $\| \phi_1 \|_0 \le  1$ from the lsc property of the norm $\| \cdot \|_0$.

\medskip
 \noindent
{\sl Step~2.  More about the dual eigenfunction.}  
From \eqref{eq:NEWHarris-Dual-LyapunovCond}  applied to $\phi_1$, we have 
$$
e^{\lambda_1 t} \| \phi_1 \|_1 =  \|   S^* (t) \phi_1 \|_1
\le   C_0 e^{\lambda t}   \|  \phi _1\|_1 + C_1 \int_0^t e^{\lambda(t-s) + \lambda_1 s} ds  \| \phi_1\|_0,
$$
so that 
$$
(1 -     C_0 e^{(\lambda-\lambda_1) t}  ) \| \phi_1 \|_1 
\le   C_1 \int_0^t e^{(\lambda-\lambda_1) \tau} d\tau  \| \phi_1\|_0. 
$$
We recall that  $\lambda_1 \ge \lambda_0$ from   \eqref{eq:NEWexistlambda1phi1},   thus 
$$
\frac{\lambda_0-\lambda}{C_1} ( 1 -  C_0 e^{(\lambda-\lambda_0) t}  ) \|  \phi _1  \|_1 \le  \|  \phi_1 \|_0, 
$$
and finally, passing to the limit $t \to \infty$, we deduce that
\beqn\label{eq:NEWHarris-Dual-LyapunovCondphi1}
\|\phi_1\|_1\leq\frac{C_1}{\lambda_0-\lambda}\|\phi_1\|_0.
\eeqn
We normalize the dual eigenfunction with the norm $\| \cdot \|_0$.
 Note that since for the eigenvector $\phi_1$ built in the step 1 we had $\|\phi_1\|_0\leq1$, 
the lower bound~\eqref{eq:NEW-bddinf1phi1} remains valid for the new normalization $\|\phi_1\|_0=1$.
Using~\eqref{eq:NEWDoblinHarris-dual}, we thus deduce 
\beqn\label{eq:NEWborneinf2phi1}
\phi_1 \ge e^{\lambda_1 (T_0-T_1)}  [\phi_1 ]_{g_\eps} \psi_\eps \ge \frac{e^{\lambda_1 (T_0-T_1)}  }{ 2 \Xi_\eps} \psi_\eps.
\eeqn

\medskip
 \noindent
{\sl Step~3. The Lyapunov condition.}  We define $ \widetilde S (t) := e^{-\lambda_1 t} S(t)$. 
From \eqref{eq:NEWHarris-Primal-LyapunovCond}, we have 
\bean
\|    \widetilde S (t) f  \|_1
&\le& C_0 e^{(\lambda-\lambda_0) t}   \|  f  \|_1  +   C_1 \int_0^t e^{(\lambda-\lambda_0)(t-s)}  \|  \widetilde S(s) f  \|_0 ds
\\
&\le&  \gamma'_L   \|  f  \|_1  +   K' \|  f  \|_0, 
\eean
for any $t \ge 0$, and with $\gamma'_L := C_0 e^{(\lambda-\lambda_0) t}$, $K' := C_0 e^{(\lambda-\lambda_0) t} (e^{C_1  t} - 1)$.  
From \eqref{eq:NEWHarris-LyapunovCondNpsieps} and \eqref{eq:NEWborneinf2phi1}, we have 
$$
 \|  f  \|_0 \le  \xi_\eps \| f \|_1 + \Xi_\eps [f]_{\psi_\eps}  \le   \xi_\eps \| f \|_1 + 2\Xi_\eps^2 e^{\lambda_1 (T_1-T_0)} [f]_{\phi_1}.
$$
We then fix $T \ge T_1 > 0$ such that $\gamma'_L < 1$ and next $\eps > 0$ such that  $\gamma_L := \gamma'_L + \xi_\eps K' < 1$ and 
we deduce  that $S$ satisfies the Lyapunov condition 
\beqn\label{eq:NEWHarris-Primal-LyapunovCondStilde}
\|    \widetilde S_T  f  \|_1 \le  \gamma_L   \|  f  \|_1 +   K [ f  ]_{\phi_1}, 
\eeqn
with $K :=  2\Xi_\eps^2 e^{\lambda_1 (T_1-T_0)} K'$.

 \medskip
 \noindent
{\sl  Step~4.} Take $f \ge 0$ such that $\| f \|_1 \le A [f]_{\phi_1}$ with $A > K/(1-\gamma_L)$.
We have 
\bean
\| \widetilde S_{T_0} f\|_1 &\le& C_0'e^{(\lambda'-\lambda_0)T_0} \| f \|_1
\\
&\le& C_0'e^{(\lambda'-\lambda_0)T_0} A [f]_{\phi_1}
\\
&=& C_0'e^{(\lambda'-\lambda_0)T_0} A  [\widetilde S_{T_0}f]_{\phi_1}
\\
&\le& C_0'e^{(\lambda'-\lambda_0)T_0} A \| \phi_1 \|_1\| \widetilde S_{T_0}f \|_0
\\
&\le& C_0'e^{(\lambda'-\lambda_0)T_0} A  \| \phi_1 \|_{1} (\xi_\eps \| \widetilde S_{T_0}f \|_1 + \Xi_\eps [ \widetilde S_{T_0}f ]_{\psi_\eps})  
 \eean 
for any $\eps > 0$, where we have used successively the growth estimate \eqref{eq:NEWHarris-GrowthX1} in the first line, the condition on $f$ in the second line, 
the eigenfunction property of $\phi_1$ in the third line, the duality bracket estimate \eqref{eq:KRnorm1} in the fourth line and the interpolation inequality \eqref{eq:NEWHarris-LyapunovCondNpsieps} in the last line. Choosing $\eps > 0$ small enough, we immediately obtain 
\bean
\| \widetilde S_{T_0}f \|_1 \le 2  C_0'e^{(\lambda'-\lambda_0)T_0}  A \| \phi_1 \|_{1} \Xi_\eps  [ \widetilde S_{T_0}f ]_{\psi_\eps}.
 \eean
Together with 
$$
  [f]_{\phi_1} =   [\widetilde S_{T_0}f]_{\phi_1}  \le 
   \| \widetilde S_{T_0}f \|_1
$$
and the relaxed  Doblin-Harris  positivity condition \eqref{eq:NEWDoblinHarris-primal},  we conclude to the conditional Doblin-Harris  positivity estimate 
\beqn\label{eq:DoblinHarrisConditional}
\widetilde S_T f \ge c g_\eps [  f]_{\phi_1}
\eeqn
for all $T\geq T_1$, with $c^{-1} = 2 C_0'e^{(\lambda'-\lambda_0)T_0}  A  \| \phi_1 \|_1 \Xi_\eps e^{\lambda'(T- T_0)}\eta^{-1}_{\eps,T}  $.
  
\medskip
 \noindent
{\sl  Step~5.} We define $\NN := \{ f \in X_1; \, \langle f , \phi_1 \rangle = 0 \}$. 
As a consequence of the last conditional Doblin-Harris  positivity estimate, we show that 
there exists $\gamma_H \in (0,1)$ such that holds the following local coupling condition
  \begin{equation} \label{eq:coupling} \bigl( f \in
    \cN,  
    \ \| f \|_1 \le A   [   f ]_{\phi_1}  \bigr) \quad \text{implies} \quad
    [ \widetilde S_T f ]_{\phi_1}  \leq \gamma_H     [   f ]_{\phi_1}, 
  \end{equation}
  still under the same condition $A > K/(1-\gamma_L)$.
  Take indeed $f \in \NN$. 
Because $\langle f , \phi_1 \rangle = 0$, the Doblin-Harris condition \eqref{eq:DoblinHarrisConditional} tells us that 
\beqn\label{eq:Harris-fpm}
\widetilde S_T f_\pm  \ge c g_\eps [  f_\pm ]_{\phi_1} = r g_\eps, \quad r :=  c [f]_{\phi_1}/2, 
\eeqn
and we may thus write  
  \bean 
  |\widetilde S_T f | &=& |\widetilde S_T f_+ - rg_\eps - \widetilde S_Tf_- + r g_\eps|
  \\
  &\leq& |\widetilde S_T f_+ - rg_\eps| + |\widetilde S_Tf_- - r g_\eps|
  \\
  &=& \widetilde S_T f_+ - rg_\eps + \widetilde S_Tf_- - r g_\eps \ = \ \widetilde S_T |f| - 2rg_\eps, 
  \eean
  where we have used  the inequality \eqref{eq:Harris-fpm} in the third line. We deduce 
    \begin{equation*}
   [ \widetilde S_T f ]_{\phi_1} \leq
    [ \widetilde S_T |f|  ]_{\phi_1} - 2 r  [  g_\eps ]_{\phi_1} = (1 - c [  g_\eps ]_{\phi_1}) [ f  ]_{\phi_1}, 
  \end{equation*}
where we have used $\widetilde S^*_T \phi_1 = \phi_1$, and that ends the proof of \eqref{eq:coupling}.

\smallskip

 We now introduce a new norm $\Nt \cdot \Nt$ on $X_1$ defined by
 \begin{equation}\label{eq:newNormV}
\Nt f  \Nt    := [ f ]_{\phi_1}+ \beta \| f \|_1, 
  \end{equation}
and we claim that there exist $\beta > 0$ small enough and $\gamma \in (0,1)$ such that 
  \begin{equation}
    \label{eq:Harris-contrac}
  \Nt  \widetilde S_T f   \Nt   \leq  \gamma     \Nt  f   \Nt ,   \quad
    \text{for any } f\in \NN. 
  \end{equation}
Note that
  $\Nt \cdot \Nt$ and $\|\cdot\|_1$ are equivalent norms, with
\begin{equation*}
  (1 + \beta)^{-1}\Nt f  \Nt  \leq \| f \|_1
  \leq \beta^{-1} \Nt  f  \Nt. 
\end{equation*}
In order to establish the contraction estimate \eqref{eq:Harris-contrac},  we fix  $f \in \cN$ and estimate the norm $ \Nt  \widetilde S_T f   \Nt $ in two alternative cases:

   \medskip\noindent \textbf{First case.} \emph{Contractivity for  small $X_1$ norm.} When 
      \begin{equation}
    \label{eq:moment-cond-rev}
    \| f \|_1  < A   [ f ]_{\phi_1}, 
  \end{equation}
  the local coupling condition \eqref{eq:coupling} implies 
  \begin{align*}
    [  \widetilde S_T  f ]_{\phi_1} \le \gamma_H [  f ]_{\phi_1}.
   \end{align*}
   Together with the Lyapunov condition \eqref{eq:NEWHarris-Primal-LyapunovCondStilde}, we  have
   \begin{align*}
  \Nt   \widetilde S_T f   \Nt    
\leq (\gamma_H+\beta K) [  f ]_{\phi_1}
            + \beta \gamma_L    \|  f \|_1
            \leq
            \gamma_1 \Nt f \Nt,
   \end{align*}
   with
   $$
   \gamma_1 := \max\{   \gamma_H + \beta K,  \gamma_L \}.
   $$
Choosing $\beta > 0$ small enough such that  $\beta K <  1 - \gamma_H$, we get $\gamma_1 < 1$ and that gives the contractivity property     \eqref{eq:Harris-contrac} in this
   case.

  \medskip\noindent \textbf{Second case.} \emph{Contractivity for large $X_1$ norm.} 
     Assume on the contrary that
   \begin{equation}
    \label{eq:moment-cond}
      \| f \|_1  \ge A  [ f ]_{\phi_1}.
  \end{equation}
  Directly from \eqref{eq:NEWHarris-Primal-LyapunovCondStilde} we deduce that
  \begin{equation*}
    \| \widetilde S_T f \|_1 \leq \gamma_L \| f \|_1 + K [  f ]_{\phi_1}
    \leq
    (\gamma_L + K/A) \| f \|_1,
  \end{equation*}
  with $\gamma_L + K/A < 1$ by assumption. On the other hand, we have 
  $$
   [  \widetilde S_T f  ]_{\phi_1} \le \langle  \widetilde S_T |f| , \phi_1 \rangle =  \langle    |f| , \phi_1 \rangle  , 
   $$
   by using the positivity property of $\widetilde S_T$ and the eigenvector property of $\phi_1$. 
Using both last estimates together, we deduce
  \begin{align*}
    \Nt   \widetilde S_T f   \Nt    
    &=
      [  \widetilde S_T  f ]_{\phi_1} + \beta       \| \widetilde S_T  f \|_1
    \\
    &\leq
      [   f ]_{\phi_1} + \beta    (\gamma_L + K/A)    \|  f \|_1         
    \\
    &\leq
      (1-\beta \delta_0) [   f ]_{\phi_1} +  \beta    (\gamma_L + K/A + \delta_0 )    \|  f \|_1, 
  \end{align*}
  for any $\delta_0 \ge 0$.   We thus get
  \begin{align*}
    \Nt   \widetilde S_T f   \Nt     \leq
    \gamma_2  \Nt   \widetilde S_T f   \Nt,
  \end{align*}
  with $\gamma_2 := \max(1-\beta \delta_0,\gamma_L + K/A + \delta_0)$.
  We get the contractivity property \eqref{eq:Harris-contrac} in this
  case by choosing $\delta_0 > 0$ small enough (and keeping the choice
  of $\beta > 0$ made in the previous case) so that
  $\gamma_2 \in (0,1)$.  The proof of \eqref{eq:Harris-contrac} is
  completed by setting $\gamma := \max \{ \gamma_1,\gamma_2 \}$.

\smallskip
 \noindent
{\sl  Step~6.}  In order to prove the existence and uniqueness of the eigenvector $f_1 \in X_1$, 
  we fix $g_0 \in \MM := \{ g \in X_1, g \ge 0, \, \langle g, \phi_1 \rangle = 1 \}$, and we define recursively
  $g_k := \widetilde S_T g_{k-1}$ for any $k \ge 1$.  Thanks to
  \eqref{eq:Harris-contrac}, we get
  $$
  \sum_{k=1}^\infty \Nt g_k - g_{k-1} \Nt \le \sum_{k=0}^\infty \gamma^k \Nt g_1- g_0 \Nt < \infty, 
  $$
  so that $(g_k)$ is a Cauchy sequence in $\MM$. We set
  $f_1 := \lim g_k \in \MM$ which is a stationary state for the mapping $\widetilde S_T$, as seen by
  passing to the limit in the recursive equations defining $(g_k)$. From  \eqref{eq:Harris-contrac} again, this is the unique 
  stationary state for this mapping in $\MM$. From the semigroup property, we have $ \widetilde S_t f_1 = \widetilde S_t \widetilde S_T f_1 =  \widetilde S_T (\widetilde S_t f_1)$
  for any $t > 0$, so that $ \widetilde S_t f_1$ is also a stationary state in $\MM$, and thus $\widetilde S_t f_1 = f_1$ for any $t > 0$, by uniqueness. That precisely means that
  $f_1$ is a positive eigenvector associated to $\lambda_1$ for the original problem.
  
  \smallskip
  For $f \in X$, we see that $h := f - \langle f,\phi_1 \rangle \phi_1 \in \NN$, and using recursively   \eqref{eq:Harris-contrac}, we deduce 
  $$
\Nt \widetilde S_{nT} h \Nt \le  \gamma^n \Nt   h \Nt, \quad \forall \, n \ge 0,  
$$
from what \eqref{eq:KRTh-constructiveRate} follows by standard arguments. \qed

\bigskip

\paragraph{\textbf{Acknowledgments.}}
The authors warmly thank Clément Mouhot for his PhD course {\it De Giorgi-Nash-Moser theory for kinetic equations} he taught in Université Paris Dauphine-PSL during Spring 2023 
which has been a source of inspiration for the present work and also the discussions  in many occasions about the same subject.  
The authors also warmly thank François Murat for the enlightening discussions about the parabolic equation in a $L^1$ and a $M^1$ framework. K.C.\ was partially supported by the Project CONVIVIALITY ANR-23-CE40-0003 of the French National Research Agency. 
{R.M.\ acknowledges the funding from the European Union’s Horizon 2020 research and innovation programme under the Marie Skłodowska-Curie grant agreement No~945332} \includegraphics[width = .4cm]{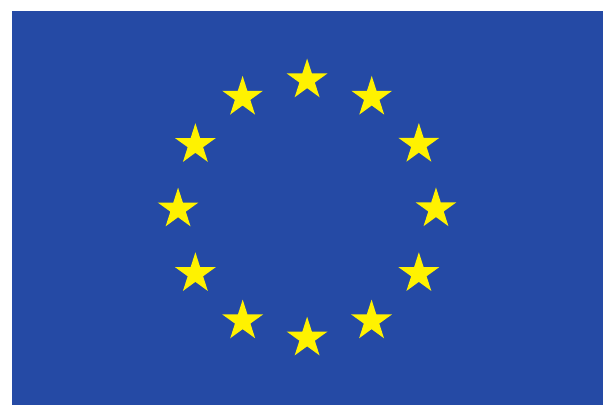}.

\bigskip


\begin{thebibliography}{10}

\bibitem{AAMN2021}
Dallas Albritton, Scott Armstrong, Jean-Christophe Mourrat, and Matthew Novack.
\newblock {Variational methods for the kinetic Fokker-Planck equation}.
\newblock {\em Anal. PDE}, 17(6):1953--2010, 2024.

\bibitem{MR4017782}
Francesca Anceschi, Michela Eleuteri, and Sergio Polidoro.
\newblock A geometric statement of the {H}arnack inequality for a degenerate
  {K}olmogorov equation with rough coefficients.
\newblock {\em Commun. Contemp. Math.}, 21(7):1850057, 17, 2019.

\bibitem{MR1301931}
L.~Arkeryd and N.~Maslova.
\newblock On diffuse reflection at the boundary for the {B}oltzmann equation
  and related equations.
\newblock {\em J. Statist. Phys.}, 77(5-6):1051--1077, 1994.

\bibitem{AIN2}
Pascal Auscher, Cyril Imbert, and Lukas Niebel.
\newblock Fundamental solutions to {K}olmogorov-{F}okker-{P}lanck equations
  with rough coefficients: existence, uniqueness, upper estimates.
\newblock {\em SIAM J. Math. Anal.}, 57(2):2114--2137, 2025.

\bibitem{AIN1}
Pascal Auscher, Cyril Imbert, and Lukas Niebel.
\newblock Weak solutions to {K}olmogorov-{F}okker-{P}lanck equations:
  regularity, existence and uniqueness.
\newblock Preprint arXiv:2403.17464, 2024.

\bibitem{Bansaye2022}
Vincent Bansaye, Bertrand Cloez, Pierre Gabriel, and Aline Marguet.
\newblock A non-conservative {H}arris ergodic theorem.
\newblock {\em J. Lond. Math. Soc. (2)}, 106(3):2459--2510, 2022.

\bibitem{Bernard2023}
{\'E}tienne Bernard.
\newblock On the existence of weak solution of the kinetic {F}okker-{P}lanck
  equation in a bounded domain with absorbing boundary.
\newblock {\em Bull. Sci. Math.}, 197:Paper No.\ 103523, 24 p., 2024.


\bibitem{MR4179249}
Armand Bernou.
\newblock A semigroup approach to the convergence rate of a collisionless gas.
\newblock {\em Kinet. Relat. Models}, 13(6):1071--1106, 2020.

\bibitem{MR1489429}
D.~Blanchard and F.~Murat.
\newblock Renormalised solutions of nonlinear parabolic problems with {$L^1$}
  data: existence and uniqueness.
\newblock {\em Proc. Roy. Soc. Edinburgh Sect. A}, 127(6):1137--1152, 1997.

\bibitem{MR1876648}
Dominique Blanchard, Fran\c{c}ois Murat, and Hicham Redwane.
\newblock Existence and uniqueness of a renormalized solution for a fairly
  general class of nonlinear parabolic problems.
\newblock {\em J. Differential Equations}, 177(2):331--374, 2001.

\bibitem{MR1025884}
Lucio Boccardo and Thierry Gallou\"{e}t.
\newblock Nonlinear elliptic and parabolic equations involving measure data.
\newblock {\em J. Funct. Anal.}, 87(1):149--169, 1989.

\bibitem{MR1949176}
F.~Bouchut.
\newblock Hypoelliptic regularity in kinetic equations.
\newblock {\em J. Math. Pures Appl. (9)}, 81(11):1135--1159, 2002.

\bibitem{MR1200643}
Fran\c{c}ois Bouchut.
\newblock Existence and uniqueness of a global smooth solution for the
  {V}lasov-{P}oisson-{F}okker-{P}lanck system in three dimensions.
\newblock {\em J. Funct. Anal.}, 111(1):239--258, 1993.

\bibitem{MR4113786}
Emeric Bouin, Jean Dolbeault, St\'{e}phane Mischler, Cl\'{e}ment Mouhot, and
  Christian Schmeiser.
\newblock Hypocoercivity without confinement.
\newblock {\em Pure Appl. Anal.}, 2(2):203--232, 2020.

\bibitem{BMM-HalfSpace**}
\'Emeric Bouin, St{\'e}phane Mischler, and Cl\'ement Mouhot.
\newblock Half-space decay for linear kinetic equations.
\newblock Preprint arXiv:2507.10506.

\bibitem{MR4534707}
Jos\'{e}~A. Ca\~{n}izo and St\'{e}phane Mischler.
\newblock {H}arris-type results on geometric and subgeometric convergence to
  equilibrium for stochastic semigroups.
\newblock {\em J. Funct. Anal.}, 284(7):Paper No. 109830, 2023.

\bibitem{MR4069622}
Chuqi Cao.
\newblock The kinetic {F}okker-{P}lanck equation with weak confinement force.
\newblock {\em Commun. Math. Sci.}, 17(8):2281--2308, 2019.

\bibitem{CM-Landau**}
Kleber Carrapatoso and St{\'e}phane Mischler.
\newblock The {L}andau equation in a domain.
\newblock Preprint arXiv:2407.09031.

\bibitem{CM-KFP**}
Kleber Carrapatoso and St{\'e}phane Mischler.
\newblock {The Kinetic Fokker-Planck equation in a domain: Ultracontractivity,
  hypocoercivity and long-time asymptotic behavior}.
\newblock {\em Rendiconti Lincei Matematica e Applicazioni}, 35(4):646--680, 2025.

\bibitem{MR1634851}
Jos\'{e}~A. Carrillo.
\newblock Global weak solutions for the initial-boundary-value problems to the
  {V}lasov-{P}oisson-{F}okker-{P}lanck system.
\newblock {\em Math. Methods Appl. Sci.}, 21(10):907--938, 1998.

\bibitem{MR0712583}
E.~B. Davies.
\newblock Spectral properties of metastable {M}arkov semigroups.
\newblock {\em J. Functional Analysis}, 52(3):315--329, 1983.

\bibitem{MR0766493}
E.~B. Davies and B.~Simon.
\newblock Ultracontractivity and the heat kernel for {S}chr\"odinger operators
  and {D}irichlet {L}aplacians.
\newblock {\em J. Funct. Anal.}, 59(2):335--395, 1984.

\bibitem{MR0093649}
Ennio De~Giorgi.
\newblock Sulla differenziabilit\`a e l'analiticit\`a delle estremali degli
  integrali multipli regolari.
\newblock {\em Mem. Accad. Sci. Torino. Cl. Sci. Fis. Mat. Nat. (3)}, pages
  25--43, 1957.

\bibitem{MR875086}
Pierre Degond.
\newblock Global existence of smooth solutions for the
  {V}lasov-{F}okker-{P}lanck equation in {$1$} and {$2$} space dimensions.
\newblock {\em Ann. Sci. \'{E}cole Norm. Sup. (4)}, 19(4):519--542, 1986.

\bibitem{MR1787105}
L.~Desvillettes and C.~Villani.
\newblock On the trend to global equilibrium in spatially inhomogeneous
  entropy-dissipating systems: the linear {F}okker-{P}lanck equation.
\newblock {\em Comm. Pure Appl. Math.}, 54(1):1--42, 2001.

\bibitem{MR3324910}
Jean Dolbeault, Cl\'{e}ment Mouhot, and Christian Schmeiser.
\newblock Hypocoercivity for linear kinetic equations conserving mass.
\newblock {\em Trans. Amer. Math. Soc.}, 367(6):3807--3828, 2015.

\bibitem{MR1969727}
J.-P. Eckmann and M.~Hairer.
\newblock Spectral properties of hypoelliptic operators.
\newblock {\em Comm. Math. Phys.}, 235(2):233--253, 2003.

\bibitem{sanchez:hal-04093201}
Claudia Fonte~Sanchez, Pierre Gabriel, and St{\'e}phane Mischler.
\newblock {On the Krein-Rutman theorem and beyond}.
\newblock {\em Mem. Eur. Math. Soc.}, accepted, 2025.

\bibitem{sanchez2024VCK}
Claudia Fonte~S{\'a}nchez and St{\'e}phane Mischler.
\newblock On the voltage-conductance kinetic equation.
\newblock Preprint arXiv:2408.02471.

\bibitem{MR3923847}
Fran\c{c}ois Golse, Cyril Imbert, Cl\'{e}ment Mouhot, and Alexis~F. Vasseur.
\newblock Harnack inequality for kinetic {F}okker-{P}lanck equations with rough
  coefficients and application to the {L}andau equation.
\newblock {\em Ann. Sc. Norm. Super. Pisa Cl. Sci. (5)}, 19(1):253--295, 2019.

\bibitem{MR3779780}
M.~P. Gualdani, S.~Mischler, and C.~Mouhot.
\newblock Factorization of non-symmetric operators and exponential
  {$H$}-theorem.
\newblock {\em M\'{e}m. Soc. Math. Fr. (N.S.)}, (153):137, 2017.

\bibitem{MR4453413}
Jessica Guerand and Cl\'{e}ment Mouhot.
\newblock Quantitative {D}e {G}iorgi methods in kinetic theory.
\newblock {\em J. \'{E}c. polytech. Math.}, 9:1159--1181, 2022.

\bibitem{MR4556285}
Arnaud Guillin, Boris Nectoux, and Liming Wu.
\newblock Quasi-stationary distribution for {H}amiltonian dynamics with
  singular potentials.
\newblock {\em Probab. Theory Related Fields}, 185(3-4):921--959, 2023.

\bibitem{MR1946444}
Yan Guo.
\newblock The {L}andau equation in a periodic box.
\newblock {\em Comm. Math. Phys.}, 231(3):391--434, 2002.

\bibitem{MR2857021}
Martin Hairer and Jonathan~C. Mattingly.
\newblock Yet another look at {H}arris' ergodic theorem for {M}arkov chains.
\newblock In {\em Seminar on {S}tochastic {A}nalysis, {R}andom {F}ields and
  {A}pplications {VI}}, volume~63 of {\em Progr. Probab.}, pages 109--117.
  Birkh\"{a}user/Springer Basel AG, Basel, 2011.

\bibitem{MR2130405}
Bernard Helffer and Francis Nier.
\newblock {\em Hypoelliptic estimates and spectral theory for {F}okker-{P}lanck
  operators and {W}itten {L}aplacians}, volume 1862 of {\em Lecture Notes in
  Mathematics}.
\newblock Springer-Verlag, Berlin, 2005.

\bibitem{MR2034753}
Fr{\'e}d{\'e}ric H{\'e}rau and Francis Nier.
\newblock Isotropic hypoellipticity and trend to equilibrium for the
  {F}okker-{P}lanck equation with a high-degree potential.
\newblock {\em Arch. Ration. Mech. Anal.}, 171(2):151--218, 2004.

\bibitem{MR3382587}
David~P. Herzog and Jonathan~C. Mattingly.
\newblock A practical criterion for positivity of transition densities.
\newblock {\em Nonlinearity}, 28(8):2823--2845, 2015.

\bibitem{MR222474}
Lars H\"{o}rmander.
\newblock Hypoelliptic second order differential equations.
\newblock {\em Acta Math.}, 119:147--171, 1967.

\bibitem{MR1503147}
A.~Kolmogorov.
\newblock Zuf\"{a}llige {B}ewegungen (zur {T}heorie der {B}rownschen
  {B}ewegung).
\newblock {\em Ann. of Math. (2)}, 35(1):116--117, 1934.

\bibitem{MR4412380}
Tony Leli\`evre, Mouad Ramil, and Julien Reygner.
\newblock A probabilistic study of the kinetic {F}okker-{P}lanck equation in
  cylindrical domains.
\newblock {\em J. Evol. Equ.}, 22(2):Paper No. 38, 74, 2022.

\bibitem{MR4347490}
Tony Leli\`evre, Mouad Ramil, and Julien Reygner.
\newblock Quasi-stationary distribution for the {L}angevin process in
  cylindrical domains, {P}art {I}: {E}xistence, uniqueness and long-time
  convergence.
\newblock {\em Stochastic Process. Appl.}, 144:173--201, 2022.

\bibitem{MR0153974}
J.-L. Lions.
\newblock {\em \'{E}quations diff\'{e}rentielles op\'{e}rationnelles et
  probl\`emes aux limites}.
\newblock Die Grundlehren der mathematischen Wissenschaften, Band 111.
  Springer-Verlag, Berlin-G\"{o}ttingen-Heidelberg, 1961.

\bibitem{PL2}
Pierre-Louis Lions.
\newblock {P}remières valeurs et fonctions propres.
\newblock Cours du Collège de France 2020-2021,
  https://www.college-de-france.fr/site/pierre-louis-lions/course-2020-2021.htm.

\bibitem{MR1166050}
Pierre-Louis Lions and Beno\^{i}t Perthame.
\newblock Lemmes de moments, de moyenne et de dispersion.
\newblock {\em C. R. Acad. Sci. Paris S\'{e}r. I Math.}, 314(11):801--806,
  1992.

\bibitem{MR3488535}
S.~Mischler and C.~Mouhot.
\newblock Exponential stability of slowly decaying solutions to the
  kinetic-{F}okker-{P}lanck equation.
\newblock {\em Arch. Ration. Mech. Anal.}, 221(2):677--723, 2016.

\bibitem{MR3465438}
S.~Mischler, C.~Qui\~{n}inao, and J.~Touboul.
\newblock On a kinetic {F}itzhugh-{N}agumo model of neuronal network.
\newblock {\em Comm. Math. Phys.}, 342(3):1001--1042, 2016.

\bibitem{MR3489637}
S.~Mischler and J.~Scher.
\newblock Spectral analysis of semigroups and growth-fragmentation equations.
\newblock {\em Ann. Inst. H. Poincar\'{e} Anal. Non Lin\'{e}aire},
  33(3):849--898, 2016.

\bibitem{MR1776840}
St\'{e}phane Mischler.
\newblock On the initial boundary value problem for the
  {V}lasov-{P}oisson-{B}oltzmann system.
\newblock {\em Comm. Math. Phys.}, 210(2):447--466, 2000.

\bibitem{MR2721875}
St\'{e}phane Mischler.
\newblock Kinetic equations with {M}axwell boundary conditions.
\newblock {\em Ann. Sci. \'{E}c. Norm. Sup\'{e}r. (4)}, 43(5):719--760, 2010.

\bibitem{MR3591133}
St\'{e}phane Mischler and Qilong Weng.
\newblock On a linear runs and tumbles equation.
\newblock {\em Kinet. Relat. Models}, 10(3):799--822, 2017.

\bibitem{MR0170091}
J\"{u}rgen Moser.
\newblock A new proof of {D}e {G}iorgi's theorem concerning the regularity
  problem for elliptic differential equations.
\newblock {\em Comm. Pure Appl. Math.}, 13:457--468, 1960.

\bibitem{MR159139}
J\"{u}rgen Moser.
\newblock A {H}arnack inequality for parabolic differential equations.
\newblock {\em Comm. Pure Appl. Math.}, 17:101--134, 1964.

\bibitem{MR0100158}
J.~Nash.
\newblock Continuity of solutions of parabolic and elliptic equations.
\newblock {\em Amer. J. Math.}, 80:931--954, 1958.

\bibitem{MR771811}
H.~Neunzert, M.~Pulvirenti, and L.~Triolo.
\newblock On the {V}lasov-{F}okker-{P}lanck equation.
\newblock {\em Math. Methods Appl. Sci.}, 6(4):527--538, 1984.

\bibitem{MR3778533}
F.~Nier.
\newblock Boundary conditions and subelliptic estimates for geometric
  {K}ramers-{F}okker-{P}lanck operators on manifolds with boundaries.
\newblock {\em Mem. Amer. Math. Soc.}, 252(1200):v+144, 2018.

\bibitem{MR2068847}
Andrea Pascucci and Sergio Polidoro.
\newblock The {M}oser's iterative method for a class of ultraparabolic
  equations.
\newblock {\em Commun. Contemp. Math.}, 6(3):395--417, 2004.

\bibitem{MR2248986}
Luc Rey-Bellet.
\newblock Ergodic properties of {M}arkov processes.
\newblock In {\em Open quantum systems. {II}}, volume 1881 of {\em Lecture
  Notes in Math.}, pages 1--39. Springer, Berlin, 2006.

\bibitem{MR4527757}
Luis Silvestre.
\newblock H{\"o}lder estimates for kinetic {F}okker-{P}lanck equations up to
  the boundary.
\newblock {\em Ars Inven. Anal.}, pages Paper No. 6, 29, 2022.

\bibitem{MR2562709}
C{\'e}dric Villani.
\newblock Hypocoercivity.
\newblock {\em Mem. Amer. Math. Soc.}, 202(950):iv+141, 2009.

\bibitem{MR4253803}
Yuzhe Zhu.
\newblock Velocity averaging and {H}\"{o}lder regularity for kinetic
  {F}okker-{P}lanck equations with general transport operators and rough
  coefficients.
\newblock {\em SIAM J. Math. Anal.}, 53(3):2746--2775, 2021.

\end{thebibliography}

\end{document}